\documentclass[11pt,reqno]{amsart}
\usepackage{amsmath,amssymb,amsthm,amsfonts,latexsym, amsthm, amscd, mathrsfs, stmaryrd}
\usepackage{mathtools}
\usepackage[colorlinks=true,linkcolor=blue,citecolor=blue]{hyperref}
\usepackage{cite}
\usepackage{graphicx,color,colortbl}
\usepackage{upgreek}
\usepackage{hhline}
\numberwithin{equation}{section}
\usepackage{dynkin-diagrams}
\usepackage{epic}
\allowdisplaybreaks
\numberwithin{equation}{section}
\usepackage{float}
\usepackage{tikz}
\usetikzlibrary{
    cd,
	shapes,
	arrows,
	positioning,
	decorations.markings,
	decorations.pathmorphing,
	circuits.logic.US,
	circuits.logic.IEC,
	fit,
	calc,
	plotmarks,
	matrix
}
\usepackage{todonotes,cancel}

\usepackage{enumerate}

\newtheorem{proposition}{Proposition}[section]
\newtheorem{lemma}[proposition]{Lemma}
\newtheorem{corollary}[proposition]{Corollary}
\newtheorem{theorem}[proposition]{Theorem}

\newtheorem{thmintro}{Theorem}

\newtheorem{conjintro}[thmintro]{Conjecture}

\theoremstyle{definition}
\newtheorem{definition}[proposition]{Definition}

\newtheorem{conjecture}[proposition]{Conjecture}

\theoremstyle{remark}
\newtheorem{remark}[proposition]{Remark}
\newcommand{\Rmnum}[1]{\expandafter\@slowromancap\romannumeral #1@}
\def \g{\mathfrak{g}}

\def \k{\mathfrak{k}}

\def \N{\mathbb{N}}
\def \Q{\mathbb{C}}
\def \C{\mathbb{C}}
\def \Z{\mathbb{Z}}
\def \I{\mathbb{I}}
\def \Br{\mathrm{Br}}

\def \cm{{\mathcal M}}

\def \bF{\mathbb{F}}

\def \bc {\mathbf{c}}
\def \bd {\mathbf{d}}
\def \fg{\mathfrak{g}}
\def \A{\mathcal{A}}

\def \wI{\I_{\circ}}
\def \wItau{\I_{\circ,\tau}}
\def \bI{\I_{\bullet}}

\def \gr{\mathrm{gr}}

\def \PZ{\mathcal{Z}}
\def \F{\mathbb{F}}

\def \cR{\mathcal{R}} 

\def \cP{\mathcal{P}}

\def \cT{\mathcal{T}}
\def \bs{\mathbf{r}} %\mathbf{s}}
\def \bF{\mathbb{F}}
\def \bw{w_\bullet}
\def \bwi{w_{\bullet,i}}

\def \bbw{{\boldsymbol{w}}}

\def \ba{\mathbf{a}}
\def \tx{\widetilde{x}}
\def \ty{\widetilde{y}}
\def \tz{\widetilde{z}}

\def \su{\mathfrak{u}}

\def \Im{\mathrm{Im}}

\def \Id{\mathrm{Id}}
\newcommand{\U}{\mathbf{U}}
\newcommand{\qbinom}[2]{\begin{bmatrix} #1\\#2 \end{bmatrix} }

\def \ov{\overline}
\def \un{\underline}

\def \X{\mathcal{X}}

\newcommand{\nc}{\newcommand}
\nc{\greentext}[1]{\textcolor{green}{#1}}
\nc{\redtext}[1]{\textcolor{red}{#1}}
\nc{\bluetext}[1]{\textcolor{blue}{#1}}
\nc{\brown}[1]{\browntext{ #1}}
\nc{\green}[1]{\greentext{ #1}}
\nc{\red}[1]{\redtext{ #1}}
\nc{\blue}[1]{\bluetext{ #1}}

\def \TT{\mathbf T}

\newcommand{\wt}{\text{wt}}

\def \U{\mathrm U}
\def \Ui{\mathrm{U}^\imath}
 
\def \ty{\widetilde{y}}

\def \bB{\mathbf{B}}

\def \bE{\mathbf{E}}
\def \bF{\mathbf{F}}

\def \bc{\mathbf{c}}

\def \Gr{\mathrm{Gr}\,}
\def \het{\mathrm{ht}}
\def \tr{\mathrm{tr}}

\def \Irr{\mathrm{Irr}\;}
\def \Spec{\mathrm{MaxSpec}\;}
\def \MaxSpec{\mathrm{MaxSpec}\;}

\def \sgn{\mathrm{sgn}}
\def \rank{\mathrm{rank}\;}
\def \P{\mathcal{P}}
\def \tT{\mathcal{T}}

\def \cO{\mathcal{O}}

\def \cOf{\mathcal{O}_{\mathrm{int}}}

\def \D{D}

\def \Uiv{\Ui_{v}}
\def \Ziv{Z_v^\imath}
\def \Zi{Z_0^\imath}
\def \Fr{\mathrm{Fr}}
\def \Fri{\mathrm{Fr}^\imath}
\def \orProd{\prod^{\longleftarrow}}
\newcommand{\arxiv}[1]{\href{http://arxiv.org/abs/#1}{\tt arXiv:\nolinkurl{#1}}}

\begin{document}

\title[Quantum symmetric pairs at roots of unity]{Representations of quantum symmetric pairs at roots of unity}

\author[Jinfeng Song]{Jinfeng Song}
\address{Department of Mathematics, The Hong Kong University of Science and Technology, Clear Water Bay, Hong Kong SAR, P.R.China}
\email{jfsong@ust.hk}

\author[Weinan Zhang]{Weinan Zhang}
\address{Department of Mathematics and New Cornerstone Science Laboratory, The University of Hong Kong, Pokfulam, Hong Kong SAR, P.R.China}
\email{mathzwn@hku.hk}

\begin{abstract}
    Let $\theta$ be an involution of a complex semisimple Lie algebra $\g$ and $(\U_v,\Uiv)$ be the associated quantum symmetric pair at an odd root of unity $v$. In this paper, generalizing the approach of De Concini-Kac-Procesi for quantum groups, we study the structures and irreducible representations of the iquantum group $\Uiv$. 
    
    We establish a Frobenius center of $\Uiv$ as a coideal subalgebra of the Frobenius center of the quantum group $\U_v$. Via a quantum Frobenius map, we show that the Frobenius center of $\Uiv$ is isomorphic to the coordinate algebra of a Poisson homogeneous space $\X$ of the dual Poisson-Lie group $G^*$. We define a filtration on $\Uiv$ such that the associated graded algebra is $q$-commutative. Using this filtration, we show that the full center of $\Uiv$ is generated by the Frobenius center and the Kolb-Letzter center, and we determine the degree of $\Uiv$. 
    We show that irreducible representations of $\Uiv$ are parametrized by $\theta$-twisted conjugacy classes. We determine the maximal dimension of those irreducible representations, and show that the dimension of an irreducible representation is maximal if the corresponding twisted conjugacy class has maximal dimension. We also study the branching problem for irreducible $\U_v$-modules when restricting to $\Uiv$. 
\end{abstract}

\maketitle

 \setcounter{tocdepth}{1}
\tableofcontents

\section{Introduction}

\subsection{Backgrounds}
Let $\g$ be a complex semisimple Lie algebra and $\U$ be the corresponding Drinfeld-Jimbo quantum group with a quantum parameter $q$. While there are rich theories on quantum groups when $q$ is generic, many interesting phenomena appear when $q$ is specialized at roots of unity. For instance, quantum groups at roots of unity are closely related to modular representation theory (cf. \cite{Lus90,AJS94}) and provide a large class of Cayley-Hamilton algebras (cf. \cite{DCPRR05}). More recently, they have been applied in geometric representation theory and the Langlands program (cf. \cite{GK93,ABG04}).

Let $\theta$ be a Lie algebra involution of $\g$, and $\k\subset \g$ be the fixed-point subalgebra. It is known that $\k$ is a reductive Lie algebra. Associated with the symmetric pair $(\g,\k)$, the \emph{quantum symmetric pair} $(\U,\Ui)$ introduced by Letzter \cite{Let02}, consists of the Drinfeld-Jimbo quantum group $\U$ and a coideal subalgebra $\Ui\subset \U$, called the \emph{iquantum group}. Following the pioneering work of Bao and Wang \cite{BW18a}, it has become clear that iquantum groups are vast generations of quantum groups, and many deep theories regarding quantum groups admit highly nontrivial generalizations to the setting of quantum symmetric pairs; see the survey \cite{Wa23} and references therein. 

There are two distinguished integral forms of quantum groups: the \emph{Lusztig form} and the \emph{De Concini-Kac form}. As for their generalizations to quantum symmetric pairs, the study of Lusztig type form of iquantum groups was initiated by Bao-Wang in their construction of the icanonical basis \cite{BW18b}, and its specialization at roots of unity was developed in \cite{BS21,BS22}. On the other hand, the De Concini-Kac type form of iquantum groups was introduced only recently by the first author \cite{So24} and its specialization at roots of unity remains largely unknown. 

Despite the burgeoning development of the structure theory of quantum symmetric pairs, the basic properties of representations of $\Ui$ remain widely open. In particular, due to the lack of a natural triangular decomposition for iquantum groups, the classification of irreducible modules have merely been established in several special types so far; cf. \cite{Wat21,Wat25,Wen20}. 

The goal of the current paper is to provide a systematic study of the structures and irreducible representations of the De Concini-Kac type form of iquantum groups at odd roots of unity. In particular, for these algebras, we obtain a parametrization of irreducible modules and describe their dimensions. This can be viewed as a generalization of the series of works by De Concini-Kac-Procesi \cite{DCK90,DCKP92,DCP93} on quantum groups. Our theory recovers the theory of quantum groups when $\theta$ is the trivial involution.

\subsection{De Concini-Kac quantum groups}\label{sec:sckq}
Let us outline the representation theory of quantum groups at roots of unity developed by De Concini-Kac-Procesi. Let $\ell$ be a positive odd integer that is relatively prime to root lengths of $\g$, and let $v$ be a primitive $\ell$th root of unity. Let $\U_{\A}\subset \U$ be the \emph{De Concini-Kac integral form}, and let $\U_v$ be the specialization of $\U_{\A}$ obtained from $q\mapsto v$. 

It follows from the work of De Concini-Kac-Procesi \cite{DCK90,DCKP92} that the irreducible representations of $\U_v$ are essentially parametrized by the conjugacy classes of the associated simply connected group $G$. More explicitly, let $(B^+,B^-)$ be a pair of opposite Borel subgroups of $G$ and let $G^*\subset B^+\times B^-$ be the \emph{dual Poisson-Lie group}. Then there is a central Hopf subalgebra $Z_0$ of $\U_v$, known as the \emph{Frobenius center}, which is isomorphic to the coordinate algebra of $G^*$ via the quantum Frobenius map
\[
\Fr:\C[G^*] \overset{\sim}{\longrightarrow} Z_0 \subset \U_v.
\]
Hence one obtains a canonical map
\begin{equation}\label{eq:irr}
\Irr \U_v\longrightarrow \MaxSpec Z_0\cong G^*,
\end{equation}
sending an irreducible representation of $\U_v$ to the maximal ideal $\text{Ann}_{Z_0}(V)$ of annihilators of $V$ in $Z_0$. For $x\in G^*$, let $\mathfrak{m}_x\subset Z_0$ be the associated maximal ideal and $\U_{v,x}=\U_v/\mathfrak{m}_x\U_v$ be the corresponding fibre algebra. Let $\varphi_0:G^*\rightarrow G$ be the map given by $(g_1,g_2)\mapsto g_1^{-1}g_2$. Then $\varphi_0$ is a $2^{\rank \g}$-covering map onto the open Bruhat cell of $G$. One main result of \cite{DCKP92} states the isomorphism $\U_{v,x}\cong \U_{v,y}$ of finite-dimensional algebras, assuming that $\varphi_0(x)$ and $\varphi_0(y)$ are in the same conjugacy class of $G$. In particular, by composing the map \eqref{eq:irr} with $\varphi_0$, one obtains a finite map 
\[
\Irr \U_v\longrightarrow G
\]
with isomorphic fibres along conjugacy classes of $G$. This picture was axiomatized later as the so-called \emph{Poisson order} in \cite{BG03}.

\subsection{The Frobenius center of iquantum groups}
The current paper generalizes the theory of De Concini-Kac-Procesi to quantum symmetric pairs. In order to study the representations of iquantum groups at roots of unity, we first construct their Frobenius centers.

Let $\Ui_{\A}=\Ui\cap \U_{\A}$ be the integral form of $\Ui$ introduced in \cite{So24}, and $\Uiv$ be the specialization of $\Ui_{\A}$ obtained by setting $q\mapsto v$. 
Let $K^\perp\subset G^*$ be the identity component of the subgroup $\{(g,\theta(g))\in G^*\mid g\in G\}$. It is known (cf. {\em loc. cit.}) that $K^\perp$ is a coisotropic subgroup of $G^*$, and hence the affine quotient $$\X=K^\perp\backslash G^*$$ is naturally a Poisson homogeneous space of $G^*$. It was shown in \cite{So24} that, when $q$ is specialized at $1$, the iquantum group $\Ui_1$ is isomorphic to the coordinate algebra $\C[\X]$ as Poisson algebras.

\begin{thmintro}[Theorem \ref{thm:Fri} \& Proposition~\ref{prop:coideal}] 
The map $\Fr$ restricts to an algebra embedding
\[
\Fri:\C[\X]\cong \Ui_1 \longrightarrow \Uiv,
\]
which identifies the coordinate algebra $\C[\X]$ with a central subalgebra $\Zi$ of $\Uiv$. Moreover, under the embedding $\Uiv\hookrightarrow \U_v$, $\Zi$ is identified with a coideal Poisson subalgebra of $Z_0$. 
\end{thmintro}

We call $\Fri$ the \emph{quantum Frobenius map}, and $\Zi$ the \emph{Frobenius center} of $\Uiv$. When $\theta$ is the trivial involution, we recover the construction of the central subalgebra $Z_0$ of $\U_v$ \cite{DCK90,DCKP92}. Unlike the quantum group case where generators of $Z_0$ are simply $\ell$th powers of root vectors, the generators of $\Zi$ involve a variant of the idivided powers introduced by Bao and Wang \cite{BW18a, BW21}, which are certain highly non-trivial rank 1 elements. Nevertheless, the leading terms of these elements are well-understood (cf. \cite{LYZ25}) and this leads to an explicit PBW basis of $\Uiv$ as a $\Zi$-module. 
%The situation for general iquantum groups is more involved than s,

%We moreover construct an explicit basis of $\Uiv$ as a $\Zi$-module. 

\begin{thmintro}[Theorem~\ref{thm:findim}]
The iquantum group $\Uiv$ is a free $\Zi$-module of rank $\ell ^{\dim \k}$ and there is an explicit PBW basis of $\Uiv$ as a $\Zi$-module. 
\end{thmintro}

\subsection{Simple modules and twisted conjugacy classes}
The Frobenius center $\Zi$ of the iquantum group induces a canonical map 
\begin{equation}\label{eq:chi}
\Irr \Uiv\longrightarrow \MaxSpec \Zi\cong \X
\end{equation}
sending an irreducible representation $V$ to the maximal ideal $\text{Ann}_{\Zi}(V)$ of annihilators of $V$ in $\Zi$. For $x\in \X$, let $\mathfrak{m}_x\subset \Zi$ be the associated maximal ideal and $\Ui_{v,x}=\Ui_v/\mathfrak{m}_x\Ui_v$ be the corresponding fibre algebras. 

The involution $\theta$ integrates to an algebraic group involution of $G$. For $g\in G$, the set $\mathcal{C}_\theta(g)=\{hg\theta(h)^{-1}\mid h\in G\}$ is called the \emph{$\theta$-twisted conjugacy class} of $g$. Let 
\begin{equation}\label{eq:phi}
\varphi:\X\longrightarrow G
\end{equation}
be the map given by $[ (g_1,g_2)]\mapsto \theta(g_1)^{-1}g_2$. We show in Proposition \ref{prop:syml} that $\varphi$ is a $2^{|\I_\bullet|}$-covering map onto its image. Here $(\I=\I_\circ\cup\I_\bullet,\tau)$ denotes the underlying \emph{Satake diagram}; see Section \ref{sec:qsp}. 

Composing the maps in \eqref{eq:chi} and \eqref{eq:phi}, we obtain a finite map
\begin{equation}
        \Psi:\Irr \Uiv\longrightarrow G,
\end{equation}
whose image can be determined directly using Proposition~\ref{prop:syml}. The next result shows that fibres of $\Psi$ are isomorphic along $\theta$-twisted conjugacy classes.

\begin{thmintro}[Theorem \ref{thm:thetaconjugacy}]\label{thm:iiso}
    For $x,y\in \X$, suppose that $\varphi(x)$ and $\varphi(y)$ are in the same $\theta$-twisted conjugacy class of $G$. Then $\Ui_{v,x}$ is isomorphic to $\Ui_{v,y}$ as finite-dimensional algebras.
\end{thmintro}

%The proof of Theorem \ref{thm:iiso} utilizes the theory of \emph{Poisson orders} introduced by Brown and Gordon \cite{BG03}, and the twisted Poisson structure $\pi_\theta$ on $G$ studied by Lu and Yakimov \cite{Lu14,LY08}.

In \cite{LY08,Lu14}, Lu and Yakimov introduced a twisted Poisson structure $\pi_\theta$ on $G$ such that every $\theta$-twisted conjugacy class is a union of symplectic leaves. Using this twisted Poisson structure $\pi_\theta$, we show in Proposition~\ref{prop:syml} that the above map $\varphi$ is Poisson and then relate the symplectic leaves of $\X$ with $\theta$-twisted conjugacy classes of $G$. In general, the preimages of $\theta$-twisted conjugacy classes are not connected; however, we show that there is a torus action between distinct connected components, which lifts to rescalings on iquantum groups. Together with the general theory of Poisson order \cite{BG03}, this allows us to prove Theorem~\ref{thm:iiso}.

\subsection{Dimensions of simple modules}

In their work \cite{DCK90}, De Concini-Kac determined the \emph{degree} (= the maximal dimension of irreducible representations) of the quantum group $\U_v$ at a root of unity. They, together with Procesi \cite{DCKP93a}, also proved that an irreducible representation of $\U_v$ has maximal dimension if the associated conjugacy class has maximal dimension.

We generalize their results to the setting of iquantum groups.

\begin{thmintro}[Corollary \ref{cor:dimv} \& Theorem \ref{thm:genrep}]\label{thm:dim}
    Let $V$ be an irreducible representation of $\Uiv$, and let $N_0=(\dim \k-\rank \k)/2$. Then the dimension of $V$ is at most $\ell^{N_0}$. Suppose that $\Psi(V)$ belongs to the $\theta$-twisted conjugacy class of maximal dimension ($=\dim \g-\rank \k$). Then the dimension of $V$ is exactly $\ell^{N_0}$.
\end{thmintro}

The proof of Theorem \ref{thm:dim} is based on a filtration of the algebra $\Uiv$, where the associated graded algebra is $q$-commutative; see Section \ref{sec:deg} for the construction of this filtration. This filtration generalizes the remarkable filtration on the De Concini-Kac quantum group introduced in \cite[\S 1.7]{DCK90}. This filtration helps us to describe the full center of $\Uiv$, and it may be of independent interests.

Furthermore, when $\Psi(V)$ lies in a lower-dimensional $\theta$-twisted conjugacy class, we propose the following conjecture regarding the dimension of general irreducible representations of $\Uiv$.

\begin{conjintro}[Conjecture \ref{conj:dim}]\label{conj:intro}
    Let $V$ be an irreducible representation of $\Uiv$, and let $a(V)$ be the dimension of the $\theta$-twisted conjugacy containing $\Psi(V)$. Then the dimension of $V$ is divided by $\ell ^{(a(V)-\dim \g+\dim\k)/2}$.
\end{conjintro}

Conjecture \ref{conj:intro} generalizes the conjecture of De Concini-Kac-Procesi \cite[Conjecture 6.8]{DCKP92} regarding the dimensions of irreducible representations of quantum groups at roots of unity. We remark that the De Concini-Kac-Procesi conjecture was proved by Kremnizer \cite{Kre06} and Sevostyanov \cite{Sev21}.

\subsection{The full center}
It was determined in \cite{DCKP92} that the full center of quantum groups at roots of unity is generated by two subalgebras: one is the Frobenius center $Z_0$ discussed in Section \ref{sec:sckq}, and the other is obtained by specializing the center at generic $q$ (i.e., the Harish-Chandra center).

The center of the iquantum group $\Ui$ at generic $q$ was determined by Kolb and Letzter \cite{KL08}, which we will call the \emph{Kolb-Letzter center}. In particular, they showed that the center is a polynomial algebra in $m$ variables, where $m$ is the rank of $\k$. Using the universal $K$-matrices, Kolb \cite{Ko20} later gave a conceptual construction of a basis of the center at generic $q$. In Proposition \ref{le:intn}, we show that those basis elements of the center belong to the integral form $\Ui_{\A}$, and hence one can consider their specializations at roots of unity. Let $Z^\imath_v$ be the center of $\Uiv$, and let $Z^\imath_1\subset Z^\imath_v$ be the central subalgebra obtained from the specialization of the Kolb-Letzter center.

%Their construction is over the field $\C(q)$.

\begin{thmintro}[Theorem \ref{thm:deg} \& Theorem \ref{thm:icenter}]\label{thm:cen}
    The center $Z_v^\imath$ of $\Uiv$ is generated by $\Zi$ and $Z_1^\imath$ as a $\C$-algebra. The natural map $\MaxSpec Z^\imath_v\rightarrow \MaxSpec \Zi$ is a finite map of degree $\ell^{\rank \k}$.
\end{thmintro}

In the quantum group cases, the counterpart of Theorem \ref{thm:cen} was established by De Concini-Kac-Procesi \cite{DCKP92}. Their approach relies on an explicit construction of a family of generic irreducible representations, called \emph{baby Verma modules}, together with the quantum Harish-Chandra isomorphism. However, due to the absence of triangular decompositions for $\Uiv$, it is a highly nontrivial problem to construct analogues of baby Verma modules for general iquantum groups, and no natural analogue of the quantum Harish-Chandra isomorphism is known for iquantum groups.

Our proof of Theorem \ref{thm:cen} for iquantum groups is more direct and purely algebraic. We explicitly compute the leading terms of generators of $\Zi$ and $Z_1^\imath$, and compare them with central elements in the associated graded algebra $\Gr\Uiv$. We believe that our arguments will be helpful in studying the geometry of twisted conjugacy classes as well.

\subsection{The branching problems}
In \cite{DCPRR05}, a general machinery was developed to study the decomposition of modules over a Cayley-Hamilton algebra when restricting to certain subalgebras.
We show that the pair $\Uiv\subset \U_v$, together with $\Zi\subset Z_0$, satisfies the compatibility condition introduced \emph{loc. cit.} (see Proposition~\ref{prop:comp}), and this allows us to utilize their machinery to obtain the following results on the branching problem of $\U_v$-modules restricting to $\Uiv$.

% We also study the branching problem of $\U_v$-modules restricting to $\Uiv$, generalizing a result of De Concini, Procesi, Reshetikhin, and Rosso \cite{DCPRR05}.

\begin{thmintro}[Theorem \ref{thm:dec} \& Corollary \ref{cor:dsum}]\label{thm:inbran}
    (1) When restricting to $\Uiv$, generic simple $\U_v$-modules are semisimple $\Uiv$-modules. Moreover, the multiplicities of simple $\Uiv$-modules in generic simple $\U_v$-modules satisfy the constraint \eqref{eq:hvw}. 
    
    (2) Generic simple $\Uiv$-modules are direct summands of restrictions of simple $\U_v$-modules.
\end{thmintro}

When the underlying symmetric pair is of \emph{diagonal type} $(\U_v\otimes \U_v,\U_v)$, Theorem \ref{thm:inbran} recovers a result of \cite{DCPRR05} on the tensor product decompositions for representations of quantum groups at roots of unity.

\subsection{Small quantum symmetric pairs}
The simple modules of $\U_v$ parametrized by the identity conjugacy class are of particular interest. They are exactly simple modules of a finite-dimensional Hopf algebra $\su$ with $\dim \su=\ell^{\dim \g}$ introduced by Lusztig \cite{Lus90}, known as the \emph{small quantum group}. %We define analogues for quantum symmetric pairs and compute their dimensions.

%\begin{thmintro}[Proposition \ref{prop:coid} \& Proposition \ref{prop:sPBW}]
%    Associated with the symmetric pair $(\g,\k)$, there is a coideal subalgebra $\su^\imath$ of $\su$. The dimension of $\su^\imath$ is $\ell^{\rank \k}$. We moreover establish a PBW basis on $\su^\imath$.
%\end{thmintro}

We introduce in Definition~\ref{def:smQSP} a coideal subalgebra $\su^\imath$ of $\su$ with $\dim \su^\imath=\ell^{\dim \k}$, which arises from the quantum symmetric pair $(\U_v,\Uiv)$ at roots of unity. We call the pair $(\su,\su^\imath)$ the \emph{small quantum symmetric pair}, and $\su^\imath$ the \emph{small iquantum group}. It follows from our definition that simple modules of $\su^\imath$ are exactly simple $\Uiv$-modules which are parametrized by the twisted conjugacy class of the identity element. 

Ginzburg-Kumar \cite{GK93} showed that the cohomology $\mathrm{H}^{2*}(\su,\C)$ of the small quantum group $\su$ is isomorphic to the coordinate algebra $\C[\mathcal{N}]$ of the nilpotent cone $\mathcal{N}$, and this result has been further applied in the geometric representation theory; for example, see \cite{ABG04}. It would be interesting to carry out similar constructions for small iquantum groups.

\subsection{Relative braid group symmetries}
It was shown in \cite{DCP93} that Lusztig braid group symmetries preserves $\U_\A$ and the quantum Frobenius map intertwines braid group symmetries on $\C[G^*]$ and $\U_v$.
In the theory of quantum symmetric pairs, the relative braid group symmetries $\TT_i$, generalizing Lusztig braid group symmetries, were systematically developed by Wang and the second author \cite{WZ23,WZ25}. These symmetries are shown to be integral \cite{SZ25} and hence they induces symmetries on both $\C[\X]$ and $\Uiv$. It is natural to study the compatibility between the quantum Frobenius map $\Fri$ and the relative braid group symmetries $\TT_i$.

\begin{thmintro}[Theorem \ref{thm:tfcom}]\label{thm:braid}
    Suppose that the symmetric pair is of \emph{quasi-split type}. Then we have 
    \begin{equation}\label{eq:braid}
\TT_i \circ \Fri=\Fri\circ \TT_i.
\end{equation}
In particular, $\Zi$ is closed under the action of $\TT_i$.
\end{thmintro}
We conjecture that Theorem \ref{thm:braid} holds in the full generality. The main obstacle is that the formulas of $\TT_i$ acting on higher powers of generators are only known for quasi-split types at this point; cf. \cite{WZ25} for these formulas in quasi-split type.

\subsection{Organizations}

The paper is organized as follows. In Section \ref{sec:pre}, we collect necessary backgrounds on quantum algebras. In Section \ref{sec:Fcent}, we construct and study the Frobenius center of iquantum groups. We construct small quantum symmetric pairs in Section \ref{sec:smQSP}. In Section \ref{sec:deg}, we construct a filtration of $\Uiv$ such that the associated graded algebra is $q$-commutative, and study the degree and center of the associated graded algebra. In Section \ref{sec:inii}, we determine the degree and the full center of $\Uiv$. In Section \ref{sec:par}, we relate simple modules of $\Uiv$ with $\theta$-twisted conjugacy classes of $G$. In Section \ref{sec:dim}, we determine the dimension of irreducible modules associated with regular twisted conjugacy classes. In Section \ref{sec:bran}, we study branching problems of $\U_v$-modules restricting to $\Uiv$-modules, and show that generic simple $\Uiv$-modules are direct summands of simple $\U_v$-modules. In Section \ref{sec:comp}, we study the compatibility between relative braid groups symmetries and quantum Frobenius maps.

\vspace{0.2in}

\noindent {\bf Acknowledgement: } The authors thank Weiqiang Wang for sharing a copy of Thomas Sale's thesis and many helpful discussions. The authors thank Huanchen Bao, Jiang-Hua Lu, Gus Schrader, and Ivan Losev for helpful discussions. JS is supported by the Glorious Sun Charity Fund. WZ is partially supported by the New Cornerstone Foundation through the New Cornerstone Investigator grant awarded to Xuhua He.

\section{Preliminaries}\label{sec:pre}

In this section, we collect backgrounds on quantum groups and quantum symmetric pairs. 

\subsection{Quantum groups}\label{sec:qg}

Let $\fg$ be a complex semisimple Lie algebra. Let $(a_{ij})_{i,j\in\I}$ be the associated Cartan matrix where $\I$ is a finite set of indices. Let $\{\epsilon_i\}_{i\in\I}$ be a set of relatively prime positive integers such that $\epsilon_i a_{ij}=\epsilon_j a_{ji}$ for $i,j\in \I$. Let $\mathcal{R}$ be the finite reduced root system associated with the Cartan matrix. Let $\{\alpha_i\}_{i\in\I}$ and $\{\omega_i\}_{i\in\I}$ be the set of simple roots and fundamental weights. Then there is a $|\I|$-dimensional Euclidean vector space $E$ containing $\mathcal{R}$ with the standard inner product $(\cdot , \cdot)$, such that $(\alpha_i,\alpha_j)=\epsilon_ia_{ij}$ and $(\omega_i,\alpha_j)=\delta_{ij}\epsilon_j$ for $i,j\in \I$. Let $P=\Z[\omega_i\mid i\in\I]$ be the weight lattice and $Q=\Z[\alpha_i\mid i\in \I]$ be the root lattice. Let $Q^+=\N[\alpha_i\mid i\in \I]$. For $\gamma=\sum_{i\in \I}{n_i\alpha_i}\in Q$, define the \emph{height} of $\gamma$ by $\text{ht}(\gamma)=\sum_{i\in\I}n_i$. For $i\in\I$ and $\mu\in P$, write $\langle \alpha_i^\vee,\mu\rangle=(\alpha_i,\mu)/\epsilon_i$. Let $W$ be the associated Weyl group with simple reflections $s_i$ ($i\in \I$) and the length function $l(\cdot)$. The $W$ acts on $E$ preserving the inner product. Let $\rho\in P$ denote the half sum of all positive roots.

Let $q$ be an indeterminate and $\C(q)$ be the field of rational functions in $q$ with coefficients in $\C$, the field of complex numbers. Fix a square root $q^{1/2}$ of $q$ in the algebraic closure of $\C(q)$. %Set $\bF$ to be the algebraic closure of $\C(q)$ and $\bF^\times:= \bF \setminus\{0\}$.
%Let $D=\diag(\epsilon_{i}|i\in \I)$ be the diagonal matrix such that $DC$ is symmetric and $\epsilon_i,i\in \I$ are coprime positive integers.
We denote
\[
q_i:=q^{\epsilon_i}, \qquad \forall i\in \I.
\]
Denote, for $r,m \in \N,i\in \I$,
\[
 [r]_i =\frac{q_i^r-q_i^{-r}}{q_i-q_i^{-1}},
 \quad
 [r]_i!=\prod_{s=1}^r [s]_i, \quad \qbinom{m}{r}_i =\frac{[m]_i [m-1]_i \ldots [m-r+1]_i}{[r]_i!}.
\]

The simply-connected Drinfeld-Jimbo quantum group $\U$ is defined to be the $\C(q^{1/2})$-algebra generated by $E_i,F_i, K_\mu$ for $i\in\I,\mu\in P$, which subjects to the following relations:
\begin{align}
&K_\mu K_{\nu}=K_{\mu+\nu},\; K_{0}=1, \label{eq:KK}
\\
&K_{\mu}E_i=q^{(\mu,\alpha_i)} E_i K_j, \quad K_\mu F_i=q^{-(\mu,\alpha_i)}F_i K_\mu,\\
&[E_i,F_j]=\delta_{ij}(q_i-q_i^{-1})(K_i^{-1}-K_i), \label{Q4}
\end{align}
where $K_i=K_{\alpha_i}$, and the quantum Serre relations, that is, for $i\neq j \in \I$,
\begin{align}
& \sum_{r=0}^{1-a_{ij}} (-1)^r  \qbinom{1-a_{ij}}{r}_i E_i^{r} E_j  E_i^{1-a_{ij}-r}=0,
  \label{eq:serre1} \\
& \sum_{r=0}^{1-a_{ij}} (-1)^r  \qbinom{1-a_{ij}}{r}_i F_i^{r} F_j  F_i^{1-a_{ij}-r}=0.
  \label{eq:serre2}
\end{align}

The algebra $\U=\bigoplus_{\mu\in Q}\U_\mu$ is naturally $Q$-graded, where $\deg E_i=\alpha_i$, $\deg F_i=-\alpha_i$, and $\deg K_\mu=0$, for $i\in \I$ and $\mu\in P$.

The algebra $\U$ admits a Hopf algebra structure with the comultiplication $\Delta$, the counit $\epsilon$ given by
\begin{align}
\begin{split}
\Delta(E_i)&=E_i\otimes 1 + K_i\otimes E_i, \qquad\; \epsilon(E_i)=0,
\\
\Delta(F_i)&=F_i\otimes K_i^{-1} + 1\otimes F_i,  \qquad \epsilon(F_i)=0,
\\
\Delta(K_\mu)&=K_\mu \otimes K_\mu,\qquad\qquad\qquad \epsilon(K_\mu)=1.
\end{split}
\end{align}

For $i\in\I$, let us write 
\begin{equation}\label{eq:recs}
    \mathbf{F}_i=\frac{F_i}{q_i^{1/2}(q_i-q_i^{-1})},\qquad \text{and}\qquad\mathbf{E}_i=\frac{E_i}{q_i^{-1/2}(q_i^{-1}-q_i)}.
\end{equation}
Then $\mathbf{E}_i,\mathbf{F}_i$ $(i\in\I)$, $K_\mu$ ($\mu\in P$) coincide with the usual Chevalley generators of quantum groups; cf. \cite[3.1.1]{Lus93}.

Let $T_i=T_{i,+1}'',i\in \I$ be the braid group symmetries on $\U$ as in \cite[37.1.3]{Lus93}. Let $\A=\C[q^{1/2},q^{-1/2}]$. The De Concini--Kac--Procesi form $\U_\A$ (see \cite{DCP93}) is defined to be the smallest $\A$-subalgebra of $\U$ which contains $E_i,F_i,K_\mu$ for all $i\in \I,\mu\in P$ and is closed under the actions of all $T_i,i\in \I$. 

Fix a reduced expression $w_0=s_{i_1} s_{i_2}\cdots s_{i_N}$ of the longest element $w_0$ in $W$. For $1\le k\le N$, set $\beta_k=s_{i_1}\cdots s_{i_{k-1}}(\alpha_{i_k})$, and the root vectors associated to $\beta_k$ are given by
\begin{align}
\label{eq:F-root}
&F_{\beta_k}=T_{i_1}T_{i_2}\cdots T_{i_{k-1}}(F_{i_k}),\qquad E_{\beta_k}=T_{i_1}T_{i_2}\cdots T_{i_{k-1}}(E_{i_k}).
\end{align}
For $\ba=(a_1,a_2,\ldots,a_N)\in \N^N$, denote
\begin{equation}\label{eq:fea}
  F^{\ba}=\orProd_{1\leq t\leq N}F_{\beta_t}^{a_t},
  %F_{\beta_N}^{a_N} \cdot F_{\beta_{N-1}}^{a_{N-1}}\cdots F_{\beta_1}^{a_1},
  \qquad 
  E^{\ba}=\orProd_{1\leq t\leq N}E_{\beta_t}^{a_t}.
  %E_{\beta_N}^{a_N} \cdot E_{\beta_{N-1}}^{a_{N-1}}\cdots E_{\beta_1}^{a_1}.
\end{equation}
Here for any ordered set $\{A_1,A_2,\cdots,A_k\}$, we write 
\[
\orProd_{1\leq t\leq k} A_t=A_kA_{k-1}\cdots A_1
\]
to denote the ordered monomial.

It was shown in \cite{DCP93} that the set of monomials $\{K_\mu F^{\ba} E^{\bc}| \mu\in P, \ba,\bc\in \N^N\}$ forms an $\A$-basis of $\U_\A$, which is called the \emph{PBW basis}.

\begin{remark}
The PBW basis is defined with the reversed ordering (of root vectors) to ensure orthogonality with respect to the Lusztig bilinear pairing. This fact will be used in Section \ref{sec:inii}. 
\end{remark}

\subsection{Quantum symmetric pairs}\label{sec:qsp}
A symmetric pair $(\g,\g^\theta)$ consists of a semisimple Lie algebra $\g$ and a Lie algebra involution $\theta$ on $\g$. It is well-known that irreducible symmetric pairs are classified (up to conjugations) by Satake diagrams.

A Satake diagram $(\I=\bI \cup \wI,\tau)$ (cf. \cite{Ko14}) consists of a partition $\bI\cup \wI$ of $\I$, and an involution $\tau$ of $\I$ (possibly $\tau=\Id$) such that
\begin{itemize}
\item[(1)] $a_{ij}=a_{\tau i,\tau j}$ for $i,j\in\I$, 
\item[(2)] $\tau(\bI)=\bI$,
\item[(3)]
 $\bw(\alpha_j) = - \alpha_{\tau j}$ for $j\in \bI$,
\item[(4)]
 If $j\in \wI$ and $\tau j =j$, then $\alpha_j(\rho_{\bullet}^\vee)\in \Z$.
\end{itemize}
Here $\bw$ is the longest element in the Weyl group $W_\bullet$ associated to $\bI$ and $\rho_{\bullet}^\vee$ is the half sum of positive coroots in the subcoroot system generated by $\I_\bullet$. 

The involution $\tau$ naturally extends to an involution on $P$ such that the bilinear pairing between $P$ and $Q$ is invariant under the involution $\tau$.

Given a Satake diagram $(\I=\wI\cup\bI,\tau )$, the (simply-connected) iquantum group $\Ui$ is the $\Q(q^{1/2})$-subalgebra of $\U$ generated by
\begin{align}
\label{def:iQG}
\begin{split}
&B_i=F_i-c_i q^{-\langle \alpha^\vee_i,\bw \alpha_{\tau i} \rangle/2} T_{\bw}( E_{\tau i}) K_i^{-1} \qquad (i\in \wI),
\\ 
&K_\mu \qquad\qquad\qquad\qquad (\mu\in P^\theta),
\\
&E_j,\quad B_j=F_j \qquad\qquad (j\in \bI).
\end{split}
\end{align}

Here $c_i\in\{\pm 1\}$ ($i\in\I_\circ$) such that 
\begin{align}\label{def:ci}
\begin{split}
 c_i c_{\tau i}&=(-1)^{2\alpha_i(\rho_{\bullet}^\vee)},\quad \text{ for } i\in \I_\circ;
 \\
c_i&=c_{\tau i}, \quad \text{ if } \big(\alpha_i,\theta(\alpha_i)\big)=0.
\end{split}
\end{align}
For convenience, we assume that
$c_i=-1$, if $\tau i=i$ or $w_\bullet (\alpha_i)=\alpha_i$.

Following \cite{SZ25}, we consider the following localization 
$$
\A'=\A[(1+q)^{-1}, [\epsilon_i]!^{-1}\mid i\in\I]\subseteq \mathbb{C}(q^{1/2}).
$$ 

Let $\U_{\A'}=\U_\A \otimes_{\A} \A' $ and $\Ui_{\A'} =\U_{\A'} \cap \Ui$. 
For any $\A'$-algebra $R$, write $\Ui_R=R\otimes_{\A'}\Ui_{\A'}$ and $\U_R=R\otimes \U_{\A'}$. By \cite[Theorem~3.11]{SZ25} and \cite[Corollary~4.3]{LYZ25}, $\Ui_{\mathcal{A}'}$ admits an $\A'$-basis as a free $\A'$-module with the leading terms given by a subset of PBW-basis of $\U_{\A'}$. Therefore for any $\A'$-algebra $R$, one has the natural embedding $\Ui_R\hookrightarrow\U_R$. In the rest of the paper we will always identify $\Ui_R$ as a subalgebra of $\U_R$. For $x\in \U_{\A'}$ (resp., $\Ui_{\A'}$) and an $\A'$-algebra $R$, if there is no confusion we still write $x$ to denote its image in $\U_R$ (resp., $\Ui_R$).

\subsection{Semi-classical limits}\label{sec:limit}

Let $\C_1$ be the $\A'$-algebra, where $\mathbb{C}_1$ is the field of complex numbers and the $\A'$-algebra structure is given by the map $q^{1/2}\mapsto 1$. Let us write $$\U_1=\U_{\C_1}\qquad \text{ and } \qquad \Ui_1=\Ui_{\mathbb{C}_1}.$$
Denote by $\un{x}$ the image of $x\in \U_\A$ in $\U_1$. It was shown in \cite{DCP93} that $[x,y]\in (q^{1/2}-1)\U_\A$ for any $x,y\in \U_\A$. Thus, $\U_1$ is commutative and is moreover equipped with a Poisson structure defined by 
\begin{equation}\label{def:SCPoisson}
\{\un{x},\un{y}\}=\un{\frac{[x,y]}{2(q^{1/2}-1)}}.
\end{equation}
Similarly, $\Ui_1$ is equipped with a Poisson structure and the natural embedding $\Ui_1\rightarrow \U_1$ is Poisson. We next recall the geometric description of the Poisson algebras $\U_1$ and $\Ui_1$.

Let $G$ be the complex semisimple simply connected group with the Lie algebra $\g$. Let $(B^+,B^-)$ be a pair of Borel subgroups where $H=B^+\cap B^-$ is the maximal torus. Given a Satake diagram in Section \ref{sec:qsp}, there is a group involution $\theta: G\rightarrow G$ which lifts the underlying involution $\theta$ on $\g$; cf. \cite[Section~2.7]{SZ25}\cite{Spr09}. The \emph{dual Poisson-Lie group} $G^*$ is the subgroup
\begin{align}\label{eq:Gdual}
G^*=\{(b_1,b_2)\in B^+\times B^-\mid \pi_H^+(b_1)\pi^-_H(b_2)=id\}\subseteq B^+\times B^-,
\end{align}
where $\pi_H^\pm: B^\pm\rightarrow H$ denote the canonical projections. It carries a natural Poisson structure $\Pi_{G^*}$ (cf. \cite[Section 2.1]{Lu14}).

Let $K^\perp$ be the identity component of the subgroup
\[
\{(g,\theta(g))\in G^*\mid g\in B^+\}\subseteq G^*.
\] 
By \cite{So24}, $K^\perp$ is a coisotropic subgroup of $G^*$, and hence the affine quotient $$\mathcal{X}=K^\perp \backslash G^*$$ carries a natural Poisson structure $\Pi_\X$, making $\X$ into a Poisson homogeneous space of $G^*$. The algebra $\C[\X]$ of regular functions on $\X$ is identified with the Poisson subalgebra of $\mathbb{C}[G^*]$ consisting of left $K^\perp$-invariant functions.

By \cite[Theorem~12.1]{DCP93} and \cite[Theorem~1]{So24}, one has canonical Poisson algebra isomorphisms $$\U_1\overset{\sim}{\longrightarrow} \C[G^*]\qquad \text{ and }\qquad \Ui_1\overset{\sim}{\longrightarrow} \C[\X],$$ 
which fit into the commutative diagram
\begin{equation}
        \begin{tikzcd}
            & \U_1 \arrow[r,"\sim"] & \mathbb{C}[G^*] \\
            & \Ui_1 \arrow[u,hook] \arrow[r,"\sim"] & \mathbb{C}[\X] \arrow[u,hook]
        \end{tikzcd}
\end{equation}

% \iffalse
% \begin{proposition}
% Let $(\U,\Ui)$ be a quantum symmetric pair of arbitrary finite type.
% \begin{itemize}
% \item[(1)] {\em \cite{DCP93}} There is a canonical isomorphism of Poisson algebras $\varphi:\U_1\overset{\sim}{\rightarrow} \C[G^*]$.

% \item[(2)] {\em \cite{So24}} There is a canonical isomorphism of Poisson algebras $\varphi:\Ui_1\overset{\sim}{\rightarrow} \C[K^\perp\backslash G^*]$. Moreover, we have the following commutative diagram
% \begin{equation}
%         \begin{tikzcd}
%             & \C[K^\perp\backslash G^*] \arrow[r,hook] & \C[G^*] \\
%             & \Ui_{1} \arrow[u,"\varphi^\imath"] \arrow["_\C\iota"',r] & \U_{1} \arrow[u,"\varphi"']
%         \end{tikzcd}
% \end{equation}
% \end{itemize}
% \end{proposition}
% \fi

\subsection{Quantum groups at roots of unity}

Let $\ell$ be a positive odd integer which is relatively prime to $\epsilon_i$ for $i\in\I$. Fix a primitive $\ell$th root of unity $\widetilde{v}\in \C$. Then $v=\widetilde{v}^2$ is again a primitive $\ell$th root of 1. Let $\C_v$ be the $\A'$-algebra, where $\mathbb{C}_v$ is the field of complex numbers and the $\A'$-algebra structure is given by the map $q^{1/2}\mapsto \widetilde{v}$.

Define $\U_v=\C_v\otimes_{\A'}\U_{\A'}$ and denote $\pi_v:\U_{\A'}\rightarrow \U_v, x\mapsto 1\otimes x$. We next recall the Frobenius center of $\U_v$ following the work of De Concini--Kac--Procesi (cf. \cite{DCK90,DCP93}).

It was shown {\em loc. cit.} that $E_i^\ell, F_i^\ell,K_\mu^\ell$ for $i\in \I,\mu\in P$ are central in $\U_v$. Here by abuse of notations, we use the same notations to denote the elements in $\U_v$ after base changes.

Following \cite{DCK90,DCP93}, let $Z_0$ be the smallest subalgebra of $\U_v$ which contains $E_i^\ell, F_i^\ell,K_\mu^\ell$ and is invariant under Lusztig braid group $\Br(W)$-action. By definition, $Z_0$ is a central subalgebra of $\U_v$. The algebra $Z_0$ is equipped with a Poisson structure as follows (see \cite{DCK90,DCKP93b}): for any $x,y\in Z_0$, take $\widetilde{x},\widetilde{y}\in \U_{\A'}$ such that $\pi_v(\widetilde{x})=x,\pi_v(\widetilde{y})=y$; then the Poisson bracket $\{x,y \}$ is defined by
\begin{align}\label{eq:Poissonv}
\{x,y \} = \pi_v\Big(\frac{[\widetilde{x},\widetilde{y}]}{\ell^2 (q v^{-1}-1)}\Big).
\end{align}
By \cite[\S 19]{DCP93}, $Z_0$ is a Hopf subalgebra of $\U_v$.

% \begin{remark}
% The formula \eqref{eq:Poissonv} differs from the one in \cite[\S 4.3]{DCKP93b} by a rescaling. 
% \end{remark}

\begin{proposition}{\em \cite{DCK90,DCP93}}
\label{prop:DCKv}
\begin{itemize}
\item[(1)] $Z_0$ is a polynomials algebra over $\C[K_\mu^\ell|\mu\in P]$ generated by elements $F_{\beta_k}^\ell,E_{\beta_k}^\ell,0\le k\le l$. 

\item[(2)] The set of monomials $F^{\ba}K_\mu E^{\bc}$ for $\ba=(a_1,\ldots,a_l),\bc=(c_1,\ldots,c_l)$ with $0\le a_i,c_i<\ell$ and $\mu$ ranges in a set of representatives of $P \slash \ell P$ form a basis of $\U_v$ over $Z_0$.

\item[(3)] There is a canonical isomorphism $\Fr:\U_1\rightarrow Z_0$ of Poisson Hopf algebras such that $\un{E_i}\mapsto -E_i^\ell,\un{F_i}\mapsto F_i^\ell,\big(\un{K_\mu}\big)^{\pm1}\mapsto K_\mu^{\pm\ell}$.

\item[(4)] For any $i\in \I$, we have $\Fr \circ T_i =T_i \circ \Fr$.
\end{itemize}

\end{proposition}

%  % \iffalse
% \begin{lemma}\label{lem:Fbeta}\cite[Lemma 3.3]{Tin17}
% Fix a reduced expression $w_0=s_{i_1} s_{i_2}\cdots s_{i_l}$ of the longest element $w_0$ in $W$. Let $\beta$ be a non-simple positive root and write $\beta=\beta_1+\beta_2$ for some $\beta_1,\beta_2\in \cR^+$. Then we have
% \[
% 2 F_\beta^\ell=\{F_{\beta_2}^\ell,F_{\beta_1}^\ell\}
% +\langle \beta_1,\beta_2\rangle F_{\beta_1}^\ell F_{\beta_2}^\ell.
% \]
% \end{lemma}

% \begin{proof}
% Following arguments similar to \cite[Lemma 3.3]{Tin17}, one can show that 
% \[
% F_\beta=\frac{q^{-1/2}}{q-q^{-1}}[F_{\beta_2},F_{\beta_1}]_{q^{-\langle \beta_1,\beta_2\rangle}}.
% \]
% By \eqref{def:SCPoisson}, the above identity implies that 
% \begin{align}\label{eq:SCroot1}
% 2 \ov{F_\beta} =\{\ov{F_{\beta_2}} ,\ov{F_{\beta_1}} \}
% +\langle \beta_1,\beta_2\rangle \ov{F_{\beta_1}} \ov{F_{\beta_2}}.
% \end{align}
% The desired identity is obtained by applying $\Fr$ to \eqref{eq:SCroot1} (recall that $\Fr$ is a Poisson map).
% \end{proof}
% \fi

\subsection{Rank one elements}\label{sec:iDP}

Let $i\in\I_\circ$. If $i=\tau i=\bw i$, then the idivided power $B_{i,\ov{p}}^{(m)}$ introduced in \cite{BW18a,BeW18,CLW21} depends on a parity $\ov{p}\in \Z\slash 2\Z$ and we use $\bB_i^{(m)}$ to denote $B_{i,\ov{m+1}}^{(m)}$. If $\tau i\neq i$, then we use $\bB_i^{(m)}$ to denote the usual divided power $\frac{\bB_i^m}{[m]_i!}$. If $\tau i=i\neq \bw i$, then we use $\bB_i^{(m)}$ to denote the idivided power $b_i^{(m)}$ introduced in \cite[Section 5]{BW21}.

In all of above three cases, we define 
\begin{align}
B_i^{[m]}=[m]_i! q_i^{m/2} (q_i-q_i^{-1})^m \bB_i^{(m)}.
\end{align}
The elements $B_i^{[m]}$ can be described explicitly as below.

\begin{itemize}
\item[(i)] Let $i=\tau i=\bw i$. The elements $B_i^{[m]}$ are given by
\begin{align}
\label{def:idv}
\begin{split}
&B_i^{[2k]}=
  \prod_{r=1}^{k}  (B_i^2+ (q_i-q_i^{-1})^2 [2r-1]_i^2), 
\\
&B_i^{[2k+1]}=B_i\prod_{r=1}^{k} (B_i^2+ (q_i-q_i^{-1})^2 [2r]_i^2).
\end{split}
\end{align}

\item[(ii)] 
Let $i\neq  \tau i$. The elements $B_i^{[m]}$ are given by
\begin{align}
\label{eq:diagiDP}
    B_i^{[m]}=B_i^m.
\end{align} 

\item[(iii)]
Let $i= \tau i \neq \bw i$. Write $Y_i=-c_iq_i^{-\langle \alpha^\vee_i,w_\bullet\alpha_{\tau i}\rangle /2}T_{w_\bullet}(E_{\tau i})K_i^{-1}$. The elements $B_i^{[m]}$ are given by
\begin{align}
\label{eq:qsiDP}
    B_i^{[m]}=\sum_{a=0}^m q_i^{-a(m-a)}\qbinom{m}{a}_i Y_i^a F_i^{m-a}.
\end{align} 
By \cite[Lemma 5.9]{BW21}, $B_i^{[m]}\in \Ui$.
\end{itemize}

It is clear that $B_i^{[m]}\in \U_{\A'}$. Hence one has $B_i^{[m]}\in \Ui_{\A'}$ for any $i\in \wI,m\in \N$. When specializing $q$ at an $\ell$th root of unity, elements $B_i^{[\ell]}$ will serve as Poisson generators of the Frobenius center of iquantum groups; see the proof of Theorem \ref{thm:Fri}.

\subsection{Relative braid group symmetries}\label{sec:rbg}

Let $\wItau\subset\wI$ be a fixed set of representatives for $\tau$-orbits. The real rank of a Satake diagram $(\I=\bI \cup \wI,\tau)$ is defined to be the cardinality of $\wItau$.

We recall the definition of relative Weyl group from \cite{WZ23}; see also \cite{Lus03,DK19}. Recall from Section~\ref{sec:qsp} that $\bw$ denotes the longest element in the Weyl group associated to $\bI$. For $i\in \wI$, we set $\I_{\bullet,i}:=\bI\cup \{i,\tau i\}$ and there is a real rank 1 Satake subdiagram $(\I_{\bullet,i}=\{i,\tau i\}\cup\bI,\tau\big|_{\I_{\bullet,i}})$.
Let $\cR_{\bullet,i}$ be the root system with the simple system $\{\alpha_j\mid j\in \I_{\bullet,i}\}$. Let $W_{\bullet,i}$ be the parabolic subgroup of $W$ generated by $s_i,i\in \I_{\bullet,i}$ and $w_{\bullet,i}$ be the longest element of $W_{\bullet,i}$. Define $\tau_i$ such that $w_{\bullet,i}(\alpha_j)=-\alpha_{\tau_{i} j}$ for $j\in \I_{\bullet,i}$. %Denote by $\cR_{\bullet,i}^+$ the corresponding positive system of $\cR_{\bullet,i}$.

For $i\in \wI$, define $\bs_i\in W_{\bullet,i}$ such that
\begin{align}\label{def:bsi}
\bwi= \bs_i w_\bullet \, (=w_\bullet \bs_i).
\end{align}
It is clear from the definition that $\bs_i=\bs_{\tau i}$. The relative Weyl group $W^\circ$ associated to $(\I=\bI \cup \wI,\tau)$ can be identified with the subgroup of $W$ generated by $\bs_i,i\in \wItau$. The group $W^\circ$ is itself a Coxeter group with Coxeter generators $\bs_i$. Let $\bbw_0$ be the longest element in $W^\circ$. It is known that $w_0=\bbw_0 w_\bullet=w_\bullet\bbw_0 $ and $l(w_0)=l(\bbw_0)+l(\bw)$.

\begin{proposition} {\em\cite[Theorems 6.1 and 9.9.]{WZ23}}
\label{prop:WZ} %Let $\bvs=\bvs_\dm$ be the distinguished parameter. 
For $i\in\wItau$, there exist algebra automorphisms $\TT_{i}$ on $\Ui$ which satisfy braid relations in $W^\circ$.
\end{proposition}

\begin{proposition}\label{prop:intT} 
\begin{itemize}
\item[(1)] {\em \cite{BW18b}} The symmetries $T_j$ for $j\in \bI$ preserve $\Ui_{\A'}$.

\item[(2)] {\em \cite{SZ25}} The relative braid group symmetries $\TT_i$ for $i\in \wItau$ preserve $\Ui_{\A'}$.
\end{itemize}
\end{proposition}

\subsection{Integrally closed domains}\label{sec:icd}

We gather backgrounds on representations of algebras which will be needed in this paper. We mainly follow \cite{DCK90} and references therein. 

Let $R$ be an associative $\C$-algebra, not necessarily commutative. We call $R$ a \emph{domain} if $R$ has no zero divisors. Let $R$ be an domain and $Z$ be the center of $R$. Let $Q(Z)$ be the fraction field of $Z$ and $Q(R):=Q(Z)\otimes_Z R$. $R$ is called {\em integrally closed} if for any subalgebra $R'$ of $Q(R)$ such that $R\subset R'\subset z^{-1}R$ for some $z\in Z,z\neq 0$, we have $R'=R$.

Note that for an integrally closed domain $R$ and any central subalgebra $Z_0$ of $R$, the commutative algebra $Z_0$ is an integrally closed domain in the sense of the usual definition for commutative algebras.

Let $R$ be an integral closed domain and $Z$ be the center of $R$. Let $\overline{Q(Z)}$ be the algebraic closure of $Q(Z)$. Assume that $R$ is a finite module over $Z$. Then it is known that $\overline{Q(Z)}\otimes _Z R$ is isomorphic to the ring $M_d(\overline{Q(Z)})$ of $d\times d$ matrices over $\overline{Q(Z)}$. The unique number $d$ is called the \emph{degree} of $R$. In particular, one has 
\begin{equation}
    \dim_{Q(Z)}Q(R)=d^2.
\end{equation}

Since $R$ is integrally closed, the usual trace of matrices induces a well-defined trace map $\tr: R\rightarrow Z$, called the \emph{reduced trace}. A finite-dimensional representation $\varphi:R\rightarrow \text{End}(V)$ of $R$ is called \emph{compatible with the trace map} if $\tr\circ \varphi=\varphi\circ \tr$, where $\tr$ on the left hand side is the usual trace of matrices.

Let $Z_0\subset Z$ be a finitely generated central subalgebra of $R$ such that $R$ is a finite module over $Z_0$. Let $\Irr R$ be the set of irreducible representations of $R$. There is a canonical map 
$$
\chi:\Irr R\rightarrow \MaxSpec Z_0
$$
sending an irreducible representation $V$ to its annihilator $\text{Ann}_{Z_0}V$ in $Z_0$. For $x\in \MaxSpec Z_0$, write $R_x=R/xR$ which is a finite-dimensional algebra. Note that the fibre $\chi^{-1}(x)$ consists of exactly irreducible representations of $R_x$.

Since $R$ has no zero divisors, $R$ is a finitely generated torsion-free module over $Z_0$ and then $R_{x}\neq 0$ for any $x\in \MaxSpec Z_0$. In particular, the map $\chi$ is surjective.

\begin{proposition}[\text{cf. \cite[Theorem 1.1]{DCKP93a}\cite[Theorem 4.5]{DCP93}}]
\label{prop:repalg}
    Let $R$ be a finitely generated $\C$-algebra and $Z$ be the center of $R$. Suppose that $R$ is an integrally closed domain, and a finite module over $Z$. Suppose that the degree of $R$ is $d$. The following holds.  

    (1) $Z$ is finitely generated. For any $x\in \MaxSpec Z$, there is a unique (up to equivalence) semisimple representation $V(x)$ of $R$, such that $\dim V(x)=d$ and $V(x)$ is compatible with the trace map. Moreover, any $W\in \Irr R_x$ is a direct summand of $V(x)$.

    (2) Any irreducible representation of $R$ has dimension at most $d$.

    (3) Let $Z_0\subset Z$ be a finitely generated central subalgebra such that $R$ is a finite module over $Z_0$. Then the set 
    \[
    \Omega_R(Z_0)=\{x\in \MaxSpec Z_0\mid \dim V=d \text{ for any }V\in \Irr R_x\}
    \]
    is a non-empty Zarisky open subset of $\MaxSpec Z_0$.
\end{proposition}

By definition, for any $x\in \Omega_R(Z)$, the semisimple representation $V(x)$ is simple.

\begin{remark}
    The theory can be applied when $R$ belongs to a more general family, called the \emph{maximal order} (cf. \cite[\S 4.6]{DCP93}). For our suppose, it suffices to assume that $R$ is an integrally closed domain.
\end{remark}

%%%%%%%%%%
%%%%%%%%%%

\section{The Frobenius center of iquantum groups}\label{sec:Fcent}
Let $\ell$ be an odd integer relatively prime to $\epsilon_i$ for all $i\in\I$. Fix $\widetilde{v}\in \C^*$ a primitive $\ell$th roots of unity, and let $v=\widetilde{v}^2$. Let $\Uiv=\Ui_{\A'}\otimes \C,q^{1/2}\mapsto \widetilde{v}$ be the iquantum group at roots of unity. For $i\in \I$, we set $v_i=v^{\epsilon_i}$.

In this section, we construct in Theorem~\ref{thm:Fri} the Frobenius center for the iquantum group $\Uiv$ at roots of unity and establish basic properties. We establish in Section~\ref{sec:iPBW} a PBW basis for the Frobenius center, as well as a basis of $\Uiv$ over its Frobenius center. We introduce in Section~\ref{sec:smQSP} the small quantum symmetric pairs and study their basic properties.

\subsection{A central subalgebra $\Zi$ and its Poisson structure}

Recall from Section~\ref{sec:qsp} that the map $\iota_v:\Ui_v\rightarrow \U_v$ obtained by the base change of $\iota: \Ui_{\A'}\hookrightarrow \U_{\A'}$ is an embedding. Recall the commutative $\C$-algebra $\Ui_1$ from Section \ref{sec:limit}. 

For $i\in\I_\circ$, let us write 
\begin{equation}\label{eq:yi}
Y_i=-c_iq_i^{-\langle \alpha_i^\vee,w_\bullet\alpha_{\tau i}\rangle /2}T_{w_\bullet}(E_{\tau i})K_i^{-1}\in \U_{\A}.    
\end{equation}
Then $B_i=F_i+Y_i$. Recall the element $B_i^{[m]}$ in $\Ui_{\A'}$, for $i\in\I_\circ$ and $n\in\N$, from Section \ref{sec:iDP}. %The following lemma due to Sale \cite{Sa20} is crucial to us.

\begin{lemma}\label{le:bkl}\emph{(\cite[Lemma 5.2.1 \& Lemma 5.3.1 \& Lemma 5.4.2 \& Lemma 5.4.3]{Sa20})} For $i\in\I_\circ$ and $k\in\N$, one has 
\[
B_i^{[k\ell]}=(F_i^\ell+Y_i^\ell)^k\qquad \text{in $\U_v$}.
\]
    
\end{lemma}

%\begin{lemma}[\text{\cite[Chapter 5]{Sa20}}]
%\label{lem:Sale}
%The elements $B_i^{[\ell]}$ for $i\in \wI$ lie in $Z_0$. In particular, they are central in $\U_v$ and $\Uiv$. 
%\end{lemma}

% We reorganize the computations in \cite[Chapter~5]{Sa20} as the follows.
We are now ready to state and prove the main theorem in this section. 
\begin{theorem}\label{thm:Fri}
    There is a unique $\C$-algebra embedding $\Fri:\U_1^\imath\rightarrow \U_v^\imath$, such that the diagram 
    \begin{equation}
        \begin{tikzcd}
            & \U_1^\imath \arrow[r,hook,"\Fri"] \arrow[d,hook] & \U_v^\imath \arrow[d,hook] \\ & \U_1 \arrow[r,hook,"\Fr"] & \U_v
        \end{tikzcd}
    \end{equation}
    commutes.
\end{theorem}

\begin{proof}
    It suffices to show that $\Fr(\U_1^\imath)\subset \U^\imath_v$. Then the map $\Fri$ is defined to be the restriction.

    Let $\pi_v^{\imath}:\Ui_{\A'}\rightarrow \Ui_v,x\mapsto x\otimes 1$. Similar to \eqref{eq:Poissonv}, there is a Poisson structure on $\Ui_v\cap Z_0$ given by
    \begin{align}\label{eq:iPoissonv}
    \{x,y \} = \pi_v^\imath\Big(\frac{[\tx,\ty]}{\ell^2 (qv^{-1}-1)}\Big),
    \end{align}
    where $\tx,\ty$ are preimages of $x,y$ in $\Ui_{\A'}$ under the base change $\pi_v^{\imath}:\Ui_{\A'}\rightarrow \Uiv$. This Poisson structure is well-defined since $\pi_v \circ\iota=\iota_v \circ\pi_v^\imath$.
    
    By definition, $\Ui_v\cap Z_0$ is naturally a Poisson subalgebra of $Z_0$. Since $\Fr:\U_1\rightarrow Z_0$ is an isomorphism of Poisson algebras, the preimage $\Fr^{-1}(\U^\imath_v)=\Fr^{-1}(\U^\imath_v\cap Z_0)$ is a Poisson subalgebra of $\U_1$. Recall from \cite{So24} that  $\U^\imath_1\cong\C[K^\perp\backslash G^*]$ is generated by elements
    \begin{equation}\label{eq:genk}
    \un{B_i},\; (i\in \I_\circ); \; \un{E_j},\; \un{F_j},\; (j\in \I_\bullet),
    \; \un{K_\mu}, \; (\mu\in P^\theta),
    \end{equation}
    as a Poisson algebra. Thus, it suffices to show that elements in \eqref{eq:genk} belong to $\Fr^{-1}(\U^\imath_v)$. 

    For $x\in\{E_j,F_j,K_\mu^{\pm1} \mid j\in\I_\bullet,\mu\in P^\theta\}$, one has 
    \[
    \Fr(\un{x})=x^\ell \qquad \text{in $\U_v$}.
    \]
    Hence $x\in \Fr^{-1}(\U^\imath_v)$. It remains to show that $\Fr(\un{B_i})\in \U^\imath_v$, for $i\in\I_\circ$. Recall the element $Y_i$ from \eqref{eq:yi}. It is clear that
    \[
    \Fr(\underline{B_i})=F_i^\ell+Y_i^\ell \qquad \text{in $\U_v$}.
    \]
    On the other hand, by Lemma \ref{le:bkl}, one has $F_i^\ell+Y_i^\ell=B_i^{[\ell]}\in \Uiv$. Hence, $\Fr(\underline{B_i})\in\Uiv$. Therefore, we have showed that all elements in \eqref{eq:genk} belong to $\Fr^{-1}(\U^\imath_v)$. The proof is completed.
    \end{proof}

\begin{definition}\label{def:Zi}
The \emph{Frobenius center} $\Zi$ of $\Uiv$ is defined to be the image of $\Fri$. 
\end{definition} 

By definition, $\Zi$ is a central subalgebra of $\Uiv$. The algebra $\Zi$ is equipped with a Poisson structure given by \eqref{eq:iPoissonv}. Under this Poisson structure, $\Zi$ is identified with a Poisson subalgebra of $Z_0$ such that the following diagram commute
\begin{equation}\label{eq:ZiZ0}
\begin{tikzcd}
            & \C[\X] \arrow[r,"\Fri","\sim" swap] \arrow[d,hook] & \Zi \arrow[d,hook] \\ & \C[G^*] \arrow[r,"\Fr","\sim" swap]  & Z_0
\end{tikzcd}
   % \MaxSpec Z_0^\imath\cong K^\perp\backslash G^*.
\end{equation}
% In particular, $\Zi$ is a central subalgebra of $\U^\imath_v$.

\begin{proposition}\label{prop:coideal}
$\Zi$ is a coideal Poisson subalgebra of $Z_0$.
\end{proposition}

\begin{proof}
It suffices to show that $\Delta(\Zi)\subset \Zi\otimes Z_0$. Recall that the multiplication of $G^*$ induces a right action of $G^*$ on $\X$ which satisfies the following commutative diagram
\begin{equation}\label{eq:GXaction}
\begin{tikzcd}
            & G^*\times G^* \arrow[r,"m"] \arrow[d] & G^* \arrow[d] \\ 
            & \X \times G^* \arrow[r,"m"]  & \X
\end{tikzcd}
\end{equation}
Hence, the comorphism $m^*:\C[G^*]\rightarrow\C[G^*]\otimes \C[G^*]$ sends $\C[\X]$ to $\C[\X]\otimes \C[G^*]$. By \cite{DCP93}, under the isomorphism $\Fr$, the comorphism $m^*$ is identified with the comultiplication $\Delta$ on $Z_0$. By \eqref{eq:ZiZ0}, the comorphism $\C[\X]\rightarrow \C[G^*]$ is identified with the embedding $\Zi\hookrightarrow Z_0$. Thus, we have proved $\Delta(\Zi)\subset \Zi\otimes Z_0$.
\end{proof}

%\begin{corollary}
%    The pair $(\U^\imath_v,Z_0^\imath)$ is a Poisson order.
%\end{corollary}

\subsection{PBW bases}\label{sec:iPBW}

Set $N=l(w_0)$, $M=l(w_\bullet)$ and $L=N-M$. Recall from Section \ref{sec:rbg} that $\bbw_0=w_\bullet w_0$ is the longest element in the relative Weyl group $W^\circ$. Fix a reduced expression $\bbw_0=\bs_{i_1}\bs_{i_2}\cdots\bs_{i_s}$ in $W^\circ$, and fix a reduced expression $w_\bullet=s_{j_1}s_{j_2}\cdots s_{j_M}$ for $w_\bullet$ in $W_\bullet$. By choosing the reduced expressions for $\bs_{i}$ in $W$ as in \cite[Table 1]{WZ23}, we obtain a reduced expression $\bbw_0=s_{i_1}\cdots s_{i_L}$ in $W$. For $1\leq t\leq L$, set $\beta_t=s_{i_1}\cdots s_{i_{t-1}}(\alpha_{i_t})$. For any $j\in \I_\bullet$, set $j'\in \I_\bullet$ to be the element such that $\alpha_{j'}=-w_0(\alpha_{\tau j})$. By Section \ref{sec:qsp} (3), one has $\alpha_{j'}=w_0w_\bullet(\alpha_j)$ for $j\in \I_\bullet$. Then $w_\bullet=s_{j'_1}\cdots s_{j'_M}$ is also a reduced expression in $W$. For $1\leq t\leq M$, set $\gamma_t=s_{j'_1}\cdots s_{j_{t-1}'}(\alpha_{j_t'})$. Now the ordered set $\{\beta_1,\cdots, \beta_L,\gamma_1,\cdots,\gamma_M\}$ induces the convex ordering on $\mathcal{R}^+$ associated with the reduced expression $w_0=s_{i_1}\cdots s_{i_L}s_{j_1}\cdots s_{j_M}$ in $W$.

For $w\in W$, set $\cR^+(w)=\cR^+\cap w(\cR^-)$ where $\cR^-=-\cR^+$. Then one has 
$\cR^+(\bbw_0)=\{\beta_1,\cdots, \beta_L\}$ and $\cR^+(\bw)=\{\gamma_1,\cdots,\gamma_M\}$

For $1\leq t\leq L$, recall the root vector $B_{\beta_t}$ in $\Ui$ defined in \cite[\S3.4]{SZ25}. For $1\leq t\leq M$, set $F_{\gamma_t}=T_{j_1'}\cdots T_{j_{t-1}'}(F_{j'_t})$ and $E_{\gamma_t}=T_{j_1'}\cdots T_{j_{t-1}'}(E_{j_t'})$. For tuples $\ba=(a_1,\cdots,a_L)\in\N^L$, $\bc=(c_1,\cdots,c_M)\in\N^M$, $\bd=(d_1,\cdots,d_M)\in\N^M$, and $\mu\in P^\theta$, set
\begin{align}\label{eq:Mabc}
M_{\mu,\ba,\bc,\bd}=K_\mu F_{\gamma_M}^{c_M}\cdots F_{\gamma_1}^{c_1}B_{\beta_L}^{a_L}\cdots B_{\beta_1}^{a_1}E_{\gamma_M}^{d_M}\cdots E_{\gamma_1}^{d_1}
\end{align}
in $\Ui$.

\begin{proposition}[\text{\cite[Theorem 3.11]{SZ25}}]\label{prop:intPBW}
The set of ordered monomials 
\begin{align}\label{eq:intPBW}
\big\{M_{\mu,\ba,\mathbf{c},\mathbf{d}}\mid
\ba\in \N^L,\bc\in \N^M,\bd\in \N^M,\mu\in P^\theta
\big\}
%K_\mu \prod_{\beta\in\cR^+(\bbw_0)} B_\beta^{a_\beta} \prod_{\gamma\in \cR^+_\bullet} F_\gamma^{b_\gamma} \prod_{\gamma\in \cR^+_\bullet} E_\gamma^{c_\gamma},
\end{align} 
forms an $\A'$-basis of $\Ui_{\A'}$. In particular, the images of these monomials under base changes form a $\C$-basis of $\Ui_1$ and a $\C$-basis of $\Ui_v$.
\end{proposition}

We still write $M_{\mu,\mathbf{a},\mathbf{c},\mathbf{d}}$ to denote its image in $\Uiv$ under the base change. 

% Still by Proposition \ref{prop:intPBW}, the set 
% \[
% \{M_{\mu,\ba,\bc,\bd}\mid \mu\in P^\theta,\mathbf{a}\in \N^L,\mathbf{c}\in\N^M, \mathbf{d}\in\N^M\}
% \] 
% forms a $\C$-basis of $\Uiv$. 

Let $\beta\in \cR^{+}(\bbw_0)$ and $m\in \N$. We denote 
\begin{equation}
B_\beta^{[m\ell]}:=\Fri(\un{B_\beta^m}).
\end{equation}
By definition, $B_\beta^{[m\ell]}=\big(B_\beta^{[\ell]}\big)^m$ and $B_\beta^{[m\ell]}$ are central elements in $\Uiv$. Moreover, thanks to Lemma \ref{le:bkl}, the element $B_{\alpha_i}^{[m\ell]}$ for $i\in \I_\circ$ coincides with the specialization of the element $B_i^{[m\ell]}$ defined in Section \ref{sec:iDP}.

The next proposition is a immediate consequence of Theorem~\ref{thm:Fri} and Proposition~\ref{prop:intPBW}.

\begin{proposition}\label{prop:PBWZi}
The set of ordered monomials 
\[
K_\mu^{\ell} \orProd_{1\leq t \leq M} F_{\gamma_t}^{\ell c_t } \orProd_{1\leq t\leq L}B_{\beta_t}^{[\ell a_t ]} \orProd_{1\leq t \leq M} E_{\gamma_t}^{\ell d_t},
\]
for all $\ba=(a_1,\cdots,a_L)\in\N^L$, $\bc=(c_1,\cdots,c_M)\in\N^M$, $\bd=(d_1,\cdots,d_M)\in\N^M$ and $\mu\in P^\theta$, forms a $\C$-basis of $\Zi$. 
\end{proposition}

For $\beta=\sum_{i\in \I} a_i \alpha_i \in \cR$, we set its {\em relative height} $\het^\imath(\beta)$ to be 
\begin{align}\label{def:hti}
\het^\imath(\beta)=\sum_{i\in \wI} a_i.
\end{align}
We linearly extend $\het^\imath$ to the root lattice $Q$. By definition, $\het^\imath(\gamma)=0$ for any $\gamma\in \cR_\bullet$. 

% The function $\het^\imath$ may be identified with $h^\Sigma$ in \cite{KL08}.

We define an $\N$-filtration on $\U=\bigcup_{n\in \N} \U_{\leqslant n}$ by $\U_{\leqslant n}=\bigoplus_{-\het^\imath(\mu)\le n} \U_\mu$. We write $h^\imath(x)=n$ if $x\in \U_{\leqslant n}\setminus\U_{\leqslant (n-1)}$. 
In particular, 
\begin{align}\label{def:hi}
h^\imath(F_i)=1,\qquad 
h^\imath( E_i)=h^\imath(K_\mu)=h^\imath(x)=0,
\end{align}
for $i\in \wI,x\in \U_{\I_\bullet},\mu\in P$, where $\U_{\I_\bullet}$ is the subalgebra generated by $E_j,F_j$ for $j\in\I_\bullet$. This filtration induces an $\N$-filtration on $\Ui$ such that $\Ui_{\leqslant n}=\Ui\cap \U_{\leqslant n}$. Let $\Ui_{\A',\leqslant n}= \Ui_{\A'} \cap \U_{\leqslant n}$.

% By \cite[Proof of Proposition 3.4]{SZ25}, the associated graded algebra $\gr(\Ui_{\A'} )\cong \red{\U_{P,\A'}}, B_i\mapsto F_i$.

Let $1\le t\le L$ and note that $h^\imath(F_{\beta_t})=\het^\imath(\beta_t)$. By \cite[Corollary 4.3]{LYZ25}, $\iota(B_{\beta_t})$ has a leading term $F_{\beta_t}$ in $\U_{\A'}$ with the maximal $h^\imath$-value $\het^\imath({\beta_t})$. Hence, by Proposition~\ref{prop:intPBW}, the set of monomials $M_{\mu,\ba,\bc,\bd}$ in \eqref{eq:Mabc} with $\sum_{t=1}^L a_t \het^\imath(\beta_t)\leq n$ form a basis for $\Ui_{\A',\leqslant n}$.

Let $R$ be an $\A'$-algebra and set $\Ui_{R,\leqslant n}=\Ui_{\A',\leqslant n}\otimes R$. Then we obtain a filtration $\Ui_R=\bigcup_{n\in \N} \Ui_{R,\leqslant n}$ such that $ B_i\in \Ui_{R,\leqslant 1}\setminus \Ui_{R,\leqslant 0}$ for $i\in \wI$. Similarly, there is a filtration on $\U_R=\U_{\A'}\otimes R$. In particular, we have a filtration on $\Ui_1$ and a filtration on $\Uiv$, and the embeddings $\iota_1:\Ui_1\rightarrow \U_1,\iota_v:\Uiv\rightarrow \U_v$ preserve the filtrations. Moreover one has 
\begin{equation}\label{eq:filFr}
\text{Fr}(\U_{1,\leqslant n})\subset \U_{v,\leqslant \ell n},\qquad \text{for $n\in\mathbb{N}$. }   
\end{equation}

\begin{lemma}\label{le:frbb}
    For $\beta\in \mathcal{R}^+(\bbw_0)$, one has 
    \[
    \Fri(\underline{B_\beta})=F_\beta^\ell+\text{l.o.t.}.
    \]
\end{lemma}

\begin{proof}
    It follows from the construction that 
    \[
    \text{Fr}^\imath(\un{B_\beta})=\text{Fr}(\un{F_\beta}+\text{l.o.t.})=F_\beta^\ell+\text{l.o.t.},
    \]
    where the last equality follows from Proposition \ref{prop:DCKv} (4) and \eqref{eq:filFr}.
\end{proof}
\begin{lemma}\label{lem:leading1}
Let $m\in \N$ and write $m=k\ell +r$ for $k\in \N,0\le r<\ell$. Let $\beta\in \cR^+(\bbw_0)=\{\beta_1,\ldots,\beta_L\}$, Then we have in $\Uiv$
\begin{align}
B_\beta^m = B_\beta^{[k\ell]} B_{\beta}^r + \text{l.o.t.}.
\end{align}
\end{lemma}

\begin{proof}
By our construction, the leading term of $B_\beta^k$ is $F_\beta^k$ for $k\in\N$.
Thus, it suffices to show that the leading term of $\iota_v(B_\beta^{[\ell]})$ is $F_\beta^\ell$ for any $\beta\in \cR^+(\bbw_0)$, which follows from Lemma \ref{le:frbb}.
\end{proof}
%Let $\beta_k,1\le k\le |\cR^+|$ be the sequence of positive roots associated to this expression. 

\begin{theorem}\label{thm:findim}
The set of monomials 
\begin{align}\label{eq:PBWv}
M_{\mu,\ba,\bc,\bd},
\end{align}
where all $a_r,c_t,d_t$ run over $0,1\ldots,\ell-1$ and $\mu$ runs over a set of representatives of $P^\theta\slash \ell P^\theta$, forms a basis of $\Uiv$ over $\Zi$. In particular, $\Ui_v$ is a free $\Zi$-module of rank  $\ell^{\dim\k}$ where $\k=\g^\theta$.
\end{theorem}

\begin{proof}
Let $\cm$ be a monomial of the form \eqref{eq:intPBW}. By Proposition~\ref{prop:intPBW}, the images of all such monomials under $\pi_v$ form a $\C$-basis of $\Uiv$. By Lemma~\ref{lem:leading1} and an induction on the degree, one can write $\pi_v(\cm)$ as a $\Zi$-linear combination of elements in \eqref{eq:PBWv}. Thus, elements in \eqref{eq:PBWv} form a spanning set of $\Uiv$ over $\Zi$.

It remains to show that elements in \eqref{eq:PBWv} are linearly independent over $\Zi$. Recall from \cite{LYZ25} that the leading term of $\iota(B_\beta)$ is $F_\beta$ and hence the leading term of $\iota_v(B_\beta)$ is $F_\beta$ in $\U_v$. This implies that the leading term of the element \eqref{eq:PBWv} in $\Uiv$ is 
\begin{align}\label{eq:PBWFv}
K_\mu \orProd_{1\leq t\leq M} F_{\gamma_t}^{c_t}  \orProd_{1\leq r\leq L} F_{\beta_r}^{a_r}\orProd_{1\leq t\leq M} E_{\gamma_t}^{d_t}.
\end{align}
By Proposition~\ref{prop:DCKv}, when $a_r,c_t,d_t$ run over $0,1\ldots,\ell-1$ and $\mu$ runs over $P^\theta\slash \ell P^\theta$, these elements in \eqref{eq:PBWFv} are linearly independent over $Z_0$. Noting that $\iota_v(\Zi)\subset Z_0$, we conclude that elements in \eqref{eq:PBWv} are linearly independent over $\Zi$.

The proof is completed.
\end{proof}

\begin{lemma}\label{le:tsl}
    Let $t=\ell^{\dim \g}$ and $s=\ell^{\dim \k}$. There exists a $\Zi$-basis $\{x_1,x_2,\dots,x_s\}$ of $\Uiv$, which extends to a $Z_0$-basis $\{x_1,x_2,\dots, x_t\}$ of $\U_v$.
\end{lemma}

\begin{proof}
    Take a reduced expression $w_0=s_{i_1}\cdots s_{i_L}s_{j_1}\dots s_{j_M}$ and the root vectors $B_{\beta_i}$, $E_{\gamma_j}$, $F_{\gamma_j}$ in $\Ui_v$, for $1\leq i\leq L$ and $1\leq j\leq M$, as in Section \ref{sec:iPBW}. Let $ P^\theta_\ell \subset P^\theta$ be a set of representatives for $P^\theta/\ell P^\theta$. By Theorem~\ref{thm:findim}, the monomials 
    \begin{equation}\label{eq:kbf}
    K_{\mu_1}\orProd_{1\leq j\leq M}
    F_{\gamma_j}^{c_j}\orProd_{1\leq i\leq L} B_{\beta_i}^{a_i}\orProd_{1\leq j\leq M} E_{\gamma_j}^{d_j},
    \end{equation}
    where $\mu\in P^\theta_\ell$, $0\leq a_i,c_j,d_j< \ell$, form a $\Zi$-basis of $\Ui_v$.

    To obtain a $Z_0$-basis of $\U_v$, we let $E_{\beta_t}=T_{s_{i_1}}\cdots T_{s_{i_{t-1}}}(E_{i_t})$ in $\U_v$ for $1\leq t\leq L$. Write $P^{-\theta}=\{\mu\in P\mid \theta\mu=-\mu\}$. Fix a set of representatives $ P^{-\theta}_\ell\subset P^{-\theta}$ for $P^{-\theta}/\ell P^{-\theta}$. Since $\ell$ is odd, we have the isomorphism $P/\ell P\cong P^\theta/\ell P^\theta\oplus P^{-\theta}/\ell P^{-\theta}$ as $\Z/\ell\Z$-modules, and then $P^\theta_\ell+P^{-\theta}_\ell$ form a set of representatives for $P/\ell P$.
    
    We \emph{claim} that the monomials
    \begin{equation}\label{eq:bsbm}
        K_{\mu_1}K_{\mu_2}\orProd_{1\leq j\leq M} F_{\gamma_j}^{c_j}\orProd_{1\leq i\leq L} B_{\beta_i}^{a_i}\orProd_{1\leq j\leq M} E_{\gamma_j}^{d_j}\orProd_{1\leq i\leq L} E_{\beta_i}^{e_i},
    \end{equation}
    where $\mu_1\in P^\theta_\ell, \mu_2\in P^{-\theta}_\ell$, $0\leq a_i,c_j,e_i,d_j<\ell$, form a $Z_0$-basis of $\U_v$. Indeed, under the filtration $\U_{v,\leq n}$ in Section \ref{sec:iPBW}, the monomial in \eqref{eq:bsbm} admits the leading term
    \begin{equation}\label{eq:kff}
    K_{\mu_1}K_{\mu_2}\orProd_{1\leq j\leq M} F_{\gamma_j}^{c_j}\orProd_{1\leq i\leq L} F_{\beta_i}^{a_i}\orProd_{1\leq j\leq M} E_{\gamma_j}^{d_j}\orProd_{1\leq i\leq L} E_{\beta_i}^{e_i}.
    \end{equation}
    By Proposition~\ref{prop:DCKv}, the monomials in \eqref{eq:kff} form a $Z_0$-basis of $\U_v$, and this implies the claim. Hence, there is a $Z_0$-basis of $\U_v$ given by the monomials in \eqref{eq:bsbm}, which contains the subset \eqref{eq:kbf} as a $\Zi$-basis of $\Uiv$. We complete the proof.
\end{proof}

\begin{proposition}
    We have $\Zi=Z_0\cap \Uiv$.
\end{proposition}

\begin{proof}
It is clear that $\Zi\subseteq Z_0\cap \Uiv $. We show the opposite inclusion.

By Lemma \ref{le:tsl}, we can take a $Z_0$-basis $M$ of $\U_v$ which contains a subset $M^\imath$ such that $M^\imath$ is a $\Zi$-basis of $\Ui_v$. It follows from the proof of Lemma \ref{le:tsl} that we may assume that $1\in M^\imath$. 
Let $x\in Z_0\cap \Ui_v$. Since $x$ belongs to $\U^\imath_v$, we can write uniquely
\[
x=\sum_{m\in M^\imath}x_m m,\qquad \text{where $x_m\in Z_0^\imath$}.
\]

Since $M^\imath$ is also $Z_0$-linearly independent and $x\in Z_0$, we conclude that $x_m=0$ unless $m=1$. Therefore $x=x_1\in \Zi$. We complete the proof.
\end{proof}

\subsection{Small quantum symmetric pairs}\label{sec:smQSP}

For $x\in G^*$, we write $\U_{v,x}=\U_v/\mathfrak{m}_x\U_v$ where $\mathfrak{m}_x\subset Z_0$ is the maximal ideal corresponding to $x$. Similarly, for $y\in K^\perp\backslash G^*$ define $\U^\imath_{v,y}=\U^\imath_v/\mathfrak{m}^\imath_y\U^\imath_v$ where $\mathfrak{m}^\imath_y\subset \Zi$ is the maximal ideal corresponding to $y$. For $x\in G^*$, we write $\overline{x}=K^\perp\cdot x\in K^\perp\backslash G^*$.
\begin{proposition}\label{prop:cops}
    For $x\in G^*$, we have $\U_{v,\overline{x}}^\imath\subset \U_{v,x}$. For $x\in G^*$, $y\in K^\perp\backslash G^*$, the coproduct induces an algebra homomorphism
    \[
    \Delta_{x,y}:\U_{v,y\cdot x}^\imath\longrightarrow \U^\imath_{v,y}\otimes \U_{v,x}.
    \]
\end{proposition}

\begin{proof}
    Let $x\in G^*$. To show that the natural embedding $\Uiv\hookrightarrow \U_v$ induces an algebra embedding $\Ui_{v,x}\hookrightarrow\U_{v,\ov{x}}$, we need to show that $\mathfrak{m}_{\overline{x}}^\imath\Uiv=\Ui_v\cap\mathfrak{m}_x\U_v$. Since $\mathfrak{m}_{\overline{x}}^\imath=\mathfrak{m}_x\cap \Zi$, it is clear that $\mathfrak{m}_{\overline{x}}^\imath\Uiv\subset\Ui_v\cap\mathfrak{m}_x\U_v$. To show another inclusion, we take an arbitrary element $X\in \Uiv\cap\mathfrak{m}_x\U_v$. Thanks to Lemma \ref{le:tsl}, there is a $Z_0$-basis $M$ of $\U_v$ which contains a $\Zi$-basis $M'$ of $\Uiv$. We can write the element $X$ in term of this basis
        \[
        X=\sum_{m\in M'}X_mm,\qquad \text{where $X_m\in \Zi$}.
        \]
    Since $X\in \mathfrak{m}_x\U_v$ we have $X_m\in \mathfrak{\mathfrak{m}}_x$, for any $m\in M'$. Hence $X_m\in \mathfrak{m}_x\cap \Zi=\mathfrak{m}_{\overline{x}}^\imath$, which implies that $X\in \mathfrak{m}^\imath_{\overline{x}}\Uiv$.

    Take $x\in G^*$ and $y\in K^\perp\backslash G^*$. The existence of the coproduct follows from the fact that 
        \[
        \Delta(\mathfrak{m}^\imath_{y\cdot x})\subset \mathfrak{m}_y^\imath\otimes Z_v+\Zi\otimes \mathfrak{m}_x,\qquad \text{ and }\qquad \Delta(\Ui_{v})\subset\Uiv\otimes \U_v.  
        \]
    We complete the proof.
\end{proof} 

Let $e\in G^*$ be the identity element. The finite-dimensional Hopf algebra $\su=\U_{v,e}$ is known as the \emph{small quantum group} introduced by Lusztig \cite{Lus90}; see also \cite[(1)]{GK93}.

\begin{definition}\label{def:smQSP}
Given a Satake diagram $(\I=\wI\cup\bI,\tau)$, the algebra $\su^\imath=\U^\imath_{v,\overline{e}}$ is called the associated \emph{small iquantum group}, and $(\su,\su^\imath)$ is called the associated \emph{small quantum symmetric pair}.
\end{definition}

\begin{proposition}\label{prop:coid}
    The algebra $\su^\imath$ is a coideal subalgebra of $\su$, which satisfy the following commutative diagram
    \begin{equation}\label{eq:smallQSP}
    \begin{tikzcd}
            & \Uiv \arrow[r] \arrow[d,hook] & \su^\imath \arrow[d,hook] \\ & \U_v \arrow[r]  & \su
    \end{tikzcd}
   \end{equation}
\end{proposition}

\begin{proof}
It follows from Proposition \ref{prop:cops} and its proof that $\su^\imath$ is a coideal subalgebra of $\su$ and the commutative diagram holds. 
\end{proof}
    
\begin{proposition}\label{prop:sPBW}
The set of ordered monomials 
\begin{align}
K_\mu \orProd_{1\leq t\leq M} F_{\gamma_t}^{c_t} \orProd_{1\leq r\leq L} B_{\beta_r}^{a_r} \orProd_{1\leq t\leq M} E_{\gamma_t}^{d_t},
\end{align}
where all $a_r,c_t,d_t$ run over $0,1\ldots,\ell-1$ and $\mu$ runs over a set of representatives of $P^\theta\slash \ell P^\theta$, forms a $\C$-basis of $\su^\imath$. In particular, $\dim \su^\imath=\ell^{\dim \mathfrak{k}}$.
\end{proposition}

\begin{proof}
   % By Proposition \ref{prop:cops}, $\su^\imath$ is a coideal subalgebra of $\su$. 
This dimension formula and PBW basis follow from Theorem \ref{thm:findim}.
\end{proof}

\begin{remark}
    In \cite{BS21}, Bao and Sale defined the modified version of small quantum symmetric pairs. Our definition is different since we obtain an actual coideal subalgebra of $\su$.
\end{remark}

%\begin{remark}
%Let $\Z_v=\Z[v^{\pm 1/2}]$ and $\F_\ell=\Z\slash \ell \Z$. Since $\ell$ is coprime to $\epsilon_i$, there is a well-define ring homomorphism $\Z_v\rightarrow \F_\ell,v\mapsto 1$. Let $\mathfrak{n}$ be the kernel of this map. 

%Recall from \cite{Lus90} the Lusztig $\Z_v$-form $\su_{\Z_v}$ for $\su$. Using the integral form of $\Ui$ in \cite{SZ25}, one can construct an $\Z_v$-form $\su^\imath_{\Z_v}$ for $\su^\imath$.  By \cite[Theorem 6.8]{Lus90}, $\su_{\Z_v}\slash \mathfrak{n}\su_{\Z_v} \cong \g_\ell$.
%Using similar arguments, one can show that 
%\[\su^\imath_{\Z_v} \slash \mathfrak{n}\su^\imath_{\Z_v} \cong \k_\ell.\].
%\end{remark}

\section{Degenerations of iquantum groups}\label{sec:deg}

In this section, we construct a filtration on $\Uiv$ such that the associated graded algebra $\Gr \Uiv$ is a localized twisted polynomial algebra. We determine the degree and the center of $\Gr \Uiv$.

\subsection{Twisted polynomial algebras}\label{sec:tpa}

Let $\F$ be a ring and $q\in \F$ be an invertible element. Given an $n\times n$ skew-symmetric integer matrix $H=(h_{ij})$, the \emph{twisted polynomial algebra} $\tT_H$ is the associative $\F$-algebra generated by $x_i$, $1\leq i\leq n$, which subjects to the relations
\[
x_ix_j=q^{h_{ij}}x_jx_i,\quad \text{for }1\leq i,j\leq n.
\]
For any subset $J\subset\{1,\cdots,n\}$, the \emph{localized twisted polynomial algebra} $\tT_{H,J}$ is defined to be $\tT_{H,J}=\tT_H[x_j^{-1}\mid j\in J]$. For any $\ba=(a_1,\cdots,a_n)\in\Z^n$, we denote by $x^{\ba}=x_1^{a_1}\cdots x_n^{a_n}$ the element in the field of fractions of $\tT_H$.

We will be mostly interested in the case when $\F=\C$ and $q=v$, a primitive $\ell$th root of 1. Assume that we are in this case in the remaining of this subsection. Let $H_\ell$ be the linear map $H_\ell:\Z^n\rightarrow (\Z/\ell\Z)^n$ which is induced from the matrix $H$. Let $K$ be the kernel of $H_\ell$ and $h$ the cardinality of the image of $H_\ell$.

\begin{proposition}[\text{\cite[Proposition~2.2]{DCKP93a}\cite[Proposition~7.1]{DCP93}}]
\label{prop:twpolycenter}
Let $J\subset \{1,\cdots,n\}$ be any subset.
\begin{itemize}
\item[(1)] The $\C$-algebra $\tT_{H,J}$ is an integrally closed domain.

\item[(2)] The set  
\[
\{x^{\mathbf{a}}\mid \mathbf{a}\in K,\;a_i\geq0\;\text{ for }i\not\in J\}
\]
forms a $\C$-basis of the center of $\tT_{H,J}$.     

\item[(3)] The degree of $\tT_{H,J}$ is $\sqrt{h}$.
\end{itemize}
\end{proposition}

\subsection{An associated graded algebra}\label{sec:asa}

In the remaining part of this section, we retain the notations in Section~\ref{sec:iPBW}. In particular, we fix a reduced expression $w_0=s_{i_1}\cdots s_{i_L}s_{j_1}\cdots s_{j_M}$ such that $w_\bullet=s_{j_1}\cdots s_{j_M}$. Let $\beta_1<\cdots <\beta_L<\gamma_1<\cdots<\gamma_M$ be the associated convex order on $\mathcal{R}$. For $1\leq t\leq M$, write $i_{L+t}=j_t$ and $\beta_{L+t}=\gamma_t$.

% For an element $F^{\ba}E^{\bc}K_\mu$, define 
% \[
% \het_\bullet(F^{\ba}E^{\bc}K_\mu)=\sum_{t=1}^M(a_{L+t}+c_{L+t})\het(\gamma_t).
% \]
% Recall the degree function $h^\imath$ defined in Section \ref{sec:iPBW}. It is clear that
% \[
% h^\imath(F^{\ba}E^{\bc}K_\mu)=\sum_{i=1}^L\het^\imath(\beta_i)
% \]

For $\ba\in \N^N$, define $\ba^\circ\in\N^L$ and $\ba^\bullet\in\N^M$ by $a_k^\circ=a_k$ for $1\leq k\leq L$ and $a_k^\bullet=a_{L+k}$ for $1\leq k\leq M$. Define two functions $\het_\bullet,h^\imath:\N^N\times \N^M\rightarrow \N$ by 
\begin{equation}
    \het_\bullet((\ba,\bd))=\sum_{t=1}^M \big( a^\bullet_t \het(\gamma_t)+ d_t \het(\gamma_t)\big),\quad \text{ and }\quad h^\imath((\ba,\bd))=\sum_{i=1}^L a_i \het^\imath(\beta_i).
\end{equation}
Recall the functions $\het,\het^\imath:Q\rightarrow \Z$ from \eqref{eq:recs} and \eqref{def:hti}.

For $\mu\in P^\theta$, $\ba\in\N^L$, $\bc,\bd\in\N^M$, recall from \eqref{eq:Mabc} the monomial $M_{\mu,\ba,\bc,\bd}$. When $\ba\in\N^N$ we also write $M_{\mu,\ba,\bd}=M_{\mu,\ba^\circ,\ba^\bullet,\bd}$ for simplicity. By our construction in Section~\ref{sec:iPBW}, the degree of the monomial $M_{\mu,\ba,\bd}$ is given by
\[
h^\imath(M_{\mu,\ba,\bd})=h^\imath((\ba,\bd)).
%\sum_{\beta\in \cR^+(\bbw_0)} a_\beta \het(\beta). %\sum_{\gamma\in \cR^+_\bullet}(c_\gamma+d_\gamma)\het(\gamma).
\]
For $\mu\in P^\theta$, $\ba\in \N^N$ and $\bd\in \N^M$, define the \emph{total degree} $d(M_{\mu,\ba,\bd})$ of the monomial $M_{\mu,\ba,\bd}$ to be the following element in $\N^{N+M+2}$
\begin{align}\label{eq:totaldeg}
d(M_{\mu,\ba,\bd})=
\big(a_N,\ldots, a_1, d_M,\ldots, d_1, &\het_\bullet((\ba,\bd)), h^\imath((\ba,\bd)\big).
\end{align}
 
We equip $\N^{N+M+2}$ with the opposite lexicographical order. For $\gamma\in \cR^+_\bullet$, we sometimes write $B_\gamma$ for $F_\gamma$. 

\begin{lemma}\label{lem:Bconvex}
Let $1 \le t< k\le L$. We have
\begin{align}\label{eq:Bconvex}
B_{\beta_k}  B_{\beta_t} - q^{-( \beta_k,\beta_t)}B_{\beta_t} B_{\beta_k} 
=\sum_{\mu,\ba,\bc,\bd} \rho_{\mu,\ba,\bc,\bd} M_{\mu,\ba,\bc,\bd},
\end{align}
where $\rho_{\mu,\ba,\bc,\bd}=0$ unless $d(M_{\mu,\ba,\bc,\bd})<d(B_{\beta_t} B_{\beta_k})$.
\end{lemma}

\begin{proof}
Recall from Section~\ref{sec:iPBW} the filtration $\U=\bigcup_{n\in \N} \U_{\leqslant n}$ and that $B_\beta$ has a leading term $F_\beta$ with respect to this filtration. Hence, under the natural isomorphism $\gr(\Ui_{\A'} )\cong \U_{P,\A'}$, the image of the left-hand side of \eqref{eq:Bconvex} in the associated graded algebra $\gr(\Ui_{\A'} )$ is 
$F_{\beta_k}  F_{\beta_t} - q^{-( \beta_k,\beta_t)}F_{\beta_t} F_{\beta_k}$. 
For $\ba=(a_1,\ldots,a_L)\in \N^L$, we denote 
\[
F^{\ba}=F_{\beta_L}^{a_L}\cdots F_{\beta_1}^{a_1},
\qquad 
B^{\ba}=B_{\beta_L}^{a_L}\cdots B_{\beta_1}^{a_1}.
\]
By \cite[lemma 1.7]{DCK90} and references therein, the subspace spanned by $\{F^{\ba},\ba\in \N^L\}$ is a subalgebra of $\U$ and we have
\begin{align}\label{eq:Fconvex}
F_{\beta_k}  F_{\beta_t} - q^{-( \beta_k,\beta_t)}F_{\beta_t} F_{\beta_k} 
=\sum_{\ba\in \N^L} \rho'_\ba F^{\ba}
\end{align}
for some $\rho'_\ba\in \A'$ such that $\rho'_\ba=0$ unless $d_1(F^{\ba})<d_1(F_{\beta_t} F_{\beta_k} )$. Here we equip $\N^L$ with the opposite lexicographical order and $d_1(F^{\ba})$ is defined by
\begin{align}
d_1(F^{\ba}) = (a_L,\ldots,a_1)\in \N^L;
\end{align}
cf. \cite[Section 1.7]{DCK90}.

Let $\ba\in \N^L$ such that $\rho'_\ba\neq 0$. The corresponding $F^{\ba}$ in the right-hand side of \eqref{eq:Fconvex} must have the same weight as $F_{\beta_k}  F_{\beta_t}$, and hence $h^\imath(B^{\ba})=h^\imath(B_{\beta_t} B_{\beta_k})$. Thus, by \eqref{eq:totaldeg}, $d_1(F^{\ba})<d_1(F_{\beta_t} F_{\beta_k} )$ forces that $d(B^{\ba})<d(B_{\beta_t} B_{\beta_k} )$. i.e.,
\begin{align}\label{eq:lot1}
\rho'_\ba= 0, \qquad \text{ unless } d(B^{\ba})<d(B_{\beta_t} B_{\beta_k} ).
\end{align}

On the other hand, recall that $h^\imath(B_\beta)=\het^\imath(\beta)$. Using the isomorphism $\gr(\Ui_{\A'} )\cong \U_{P,\A'}$, \eqref{eq:Fconvex} implies that the leading term of the following element
\begin{align}\label{eq:Bconvexlow}
B_{\beta_k}  B_{\beta_t} - q^{-(\beta_k,\beta_t)}B_{\beta_t} B_{\beta_k} 
-\sum_{\ba} \rho'_\ba B^{\ba}
\end{align}
has a degree strictly less than $h^\imath(B_{\beta_t} B_{\beta_k})=\het^\imath(\beta_t+\beta_k)$. 
Thus, if we rewrite the element \eqref{eq:Bconvexlow} in terms of the PBW basis $\{M_{\mu,\ba,\bc,\bd}\mid \mu\in P^\theta,\ba\in \N^L,\bc\in \N^M, \bd\in\N^M\}$ of $\Ui$
\begin{align} \label{eq:lot2}
B_{\beta_k}  B_{\beta_t} - q^{-( \beta_k,\beta_t)}B_{\beta_t} B_{\beta_k} 
-\sum_{\ba} \rho'_\ba B^{\ba}=\sum_{\mu,\ba,\bc,\bd} \rho_{\mu,\ba,\bc,\bd}'' M_{\mu,\ba,\bc,\bd},
\end{align}
then $\rho_{\mu,\ba,\bc,\bd}''=0$ unless $h^\imath(M_{\mu,\ba,\bc,\bd})<h^\imath(B_{\beta_t} B_{\beta_k})$. 

Therefore, by \eqref{eq:lot1} and \eqref{eq:lot2}, it is clear that $B_{\beta_k}  B_{\beta_t} - q^{-( \beta_k,\beta_t)}B_{\beta_t} B_{\beta_k}$ is a linear combination of monomials whose total degrees are strictly less that $d(B_{\beta_t} B_{\beta_k})$, as desired.
\end{proof}  

\begin{lemma}\label{lem:BE}
Let $1\leq t \leq L, 1\le k \le M$. We have
\begin{align}\label{eq:BE}
E_{\gamma_k}  B_{\beta_t} -  B_{\beta_t} E_{\gamma_k}
=\sum_{\mu,\ba,\bc,\bd} \rho_{\mu,\ba,\bc,\bd} M_{\mu,\ba,\bc,\bd},
\end{align}
where $\rho_{\mu,\ba,\bc,\bd}=0$ unless $d(M_{\mu,\ba,\bc,\bd})<d(B_{\beta_t} E_{\gamma_k})$.
\end{lemma}

\begin{proof}
Recall from \eqref{def:hi} the filtration and the function $h^\imath$ on $\U$. By \cite{LYZ25}, the leading term of $B_{\beta_t}$ is $F_{\beta_t}$. Hence, we have
\begin{align}
\label{eq:BEleading}
E_{\gamma_k}  B_{\beta_t} -  B_{\beta_t} E_{\gamma_k}=E_{\gamma_k}  F_{\beta_t} -  F_{\beta_t} E_{\gamma_k}+X,
\end{align}
where $h^\imath(X)<h^\imath(B_{\beta_t} E_{\gamma_k})$. Note that $h^\imath(E_{\gamma_k}  F_{\beta_t} -  F_{\beta_t} E_{\gamma_k})=h^\imath( B_{\beta_t} E_{\gamma_k})$. However, using the relation \eqref{Q4}, we have
\[
E_{\gamma_k}  F_{\beta_t} -  F_{\beta_t} E_{\gamma_k}=\sum_{\mu,\ba,\bc,\bd} \rho'_{\mu,\ba,\bc,\bd} M_{\mu,\ba,\bc,\bd},
\]
where $\rho'_{\mu,\ba,\bc,\bd}=0$ unless $\het_\bullet(M_{\mu,\ba,\bc,\bd})< \het_\bullet (B_{\beta_t} E_{\gamma_k}).$ Thus, the total degree for the right-hand side of \eqref{eq:BEleading} is strictly less than $d(B_{\beta_t} E_{\gamma_k})$ as desired.
\end{proof}

For $\bc\in \N^{N+M+2}$, write $\Ui_{\A'}({\leq \bc})$ (resp., $\Ui_{\A'}({< \bc})$) to be the $\A'$-subspace of $\Ui_{\A'}$ spanned by $M_{\mu,\ba,\bd}$, for $d(M_{\mu,\ba,\bd})\leq \bc$ (resp., $d(M_{\mu,\ba,\bd})< \bc$). Then $\Ui_{\A'}$ becomes a $\N^{N+M+2}$-filtered algebra. Let $\Gr \Ui_{\A'}$ be the associated graded algebra, that is, 
\[
\Gr\Ui_{\A'}=\bigoplus_{\bc\in\N^{N+M+2}}\Ui_{\A'}(\leq \bc)/\Ui_{\A'}(<\bc).
\]

\begin{theorem}\label{thm:assoc}
The associated graded algebra $\Gr \Ui_{\A'}$ is an $\A'$-algebra generated by elements 
\[
B_{\beta_k},  E_{\gamma_t}, K_\mu
\]
for $1\le k\le N,1\le t\le M,\mu\in P^\theta$, subject to the following relations:
\begin{align*}
& K_{\mu_1}K_{\mu_2}=K_{\mu_2}K_{\mu_1}=K_{\mu_1+\mu_2}, 
\\
& E_{\gamma_t} B_{\beta_k} = B_{\beta_k} E_{\gamma_t},
\\
& K_\mu B_{\beta_k}=q^{-(\mu,\beta_k)}B_{\beta_k} K_\mu,
\qquad \quad K_\mu E_{\gamma_t}=q^{(\mu,\gamma_t)}E_{\gamma_t} K_\mu,
\\
&  E_{\gamma_{t_1}} E_{\gamma_{t_2}} = q^{-( \gamma_{t_1},\gamma_{t_2} )} E_{\gamma_{t_2}} E_{\gamma_{t_1}},
\qquad
B_{\beta_{k_1}} B_{\beta_{k_2}} = q^{-( \beta_{k_1},\beta_{k_2} )} B_{\beta_{k_2}} B_{\beta_{k_1}}, 
 \qquad
 \text{ if } t_1>t_2,k_1>k_2.
\end{align*}
\end{theorem}

\begin{proof}
Let $\mathscr{A}$ be the algebra defined by this presentation. Then it is clear that $\mathscr{A}$ is a localized twisted polynomial algebra, and the set of monomials 
$\{ K_\mu B^{\ba} F^{\bc} E^{\bd}|\mu\in P, \ba\in \N^{L},\bc,\bd\in \N^{M}\}$ form a basis of $\mathscr{A}$. On the other hand, the PBW basis of $\Ui_{\A'}$ descends to a basis of $\Gr\Ui_{\A'}$. Thus, it suffices to show that the defining relations of $\mathscr{A}$ hold in $\Gr \Ui_{\A'}$.

It is straightforward to check the relation between $B_\beta,K_\mu$ and the relation between $E_\gamma, K_\mu$.

Recall that $\cR^+=\{\beta_1,\cdots, \beta_L,\gamma_1,\cdots,\gamma_M\}$. Note that if $\ba=0\in\N^L$, then $h^\imath(M_{\mu,0,\bc,\bd})=0$ and the total degree defined in \eqref{eq:totaldeg} is compatible with the total degree in \cite{DCK90}. Thus, by \cite[Proposition 1.7]{DCK90}, we have the following relations
in $\Gr \Ui_{\A'}$ for $i>j$
\[
E_{\gamma_i} E_{\gamma_j} = q^{-( \gamma_i,\gamma_j )} E_{\gamma_j} E_{\gamma_i},
\qquad
F_{\gamma_i} F_{\gamma_j} = q^{-( \gamma_i, \gamma_j )} F_{ \gamma_j} F_{ \gamma_i}, 
 \]

By Lemma~\ref{lem:Bconvex}, we have $B_{\beta_i} B_{\beta_j} = q^{-( \beta_i,\beta_j )} B_{\beta_i} B_{\beta_j}$ for $1\leq j < i \leq L$. Using similar arguments as in the proof of Lemma~\ref{lem:Bconvex}, we can also show $F_{\gamma_k} B_{\beta_t} = q^{-(\gamma_k ,\beta_t)} B_{\beta_t} F_{\gamma_k}$ in $\Gr \Ui_{\A'}$.
The remaining relation $E_\gamma B_\beta = B_{\beta} E_\gamma$ follows from Lemma~\ref{lem:BE}.

Therefore, we have showed that all defining relations of $\mathscr{A}$ hold in $\Gr \Ui_{\A'}$ as desired. 
\end{proof}

\begin{remark}\label{rmk:deg}
    Similar to the observations in \cite[Remark 10.1]{DCP93} and \cite[Remark 4.2 (a)]{DCKP93a}, the algebra $\Gr\Ui_{\A'}$ can be viewed as the associated graded algebra with respect to iterated $\N$ filtration. 
\end{remark}

We would like to apply the base change of the filtration on $\Ui_{\A'}$. The following lemma is direct.

\begin{lemma}\label{le:basc}
    Let $A$ be a ring, $R=\bigcup_{\bc}R(\leq \bc)$ be a filtrated $A$-algebra, and $R$ be the associated graded algebra. Suppose that $R$ admits a basis $B$ as an $A$-module, such that each subspace $R(\leq \bc)$ is spanned by a subset of $B$. Then for any $A$-algebra $A'$, the $A'$-algebra $R'=A'\otimes_{A}R=\bigcup_{\bc}R'(\leq \bc)$ is a filtrated $A'$-algebra where $R'(\leq \bc)=R'\otimes_{R}R(\leq \bc)$. Moreover the associated graded algebra $\Gr A'$ is isomorphic to $R'\otimes_{R}\Gr A$.
\end{lemma}

By Lemma \ref{le:basc}, we obtain an $\N^{N+M+2}$-filtration on the $\C$-algebra $\Ui_v$. Let $\Gr \Uiv$ denote the associated graded algebra. The following proposition is immediate from Theorem \ref{thm:assoc} and Lemma \ref{le:basc}

\begin{proposition}\label{prop:assoc}
    The associated graded algebra $\Gr \Uiv$ is a $\C$-algebra generated by elements 
\[
B_{\beta_k},  E_{\gamma_t}, K_\mu
\]
for $1\le k\le N,1\le t\le M,\mu\in P^\theta$, which are subject to the same relations as in Theorem \ref{thm:assoc} with $q$ replaced by $v$.
\end{proposition}
% We set 
%\begin{align}
%d(M_{\mu,\ba,\bc,\bd})=\big( a_{\beta_1},\ldots,a_{\beta_s}, c_{\gamma_1},\ldots,c_{\gamma_r}, d_{\gamma_r},\ldots,d_{\gamma_1}, \het(M_{\mu,\ba,\bc,\bd})\big)\in \N^{s+2r+1}.
%\end{align}

Define the skew symmetric matrices 
\begin{equation}\label{eq:AA'}
    A=(\sgn(j-i)(\beta_i,\beta_j))_{1\leq i,j\leq N}\quad \text{ and }\quad A'=(\sgn(j-i)(\gamma_i,\gamma_j))_{1\leq i,j\leq M},
\end{equation}
where $\sgn(k)=1,-1$ or $0$ when $k>0$, $k<0$ or $k=0$, respectively. 

% Let $'P=\mathbb{Z}[2^{-1}]\otimes_\mathbb{Z} P$. Then $'P={'P}^\theta\oplus{'P}^{-\theta}$, where $'P^\theta$ (resp., $'P^{-\theta}$) is the subgroup of $\theta$-invariants (resp., $\theta$-antiinvariants). Note that $'P^\theta$ is the free $\mathbb{Z}[2^{-1}]$-module with basis $\{\overline{\omega_i}=(\omega_i+\theta(\omega_i))/2\mid i\in \I\backslash \wItau\}$. For simplicity let us assume that $\I\backslash \wItau=\{1,\dots,R\}$ where $|\I\backslash \wItau|=R$. Define 
% \[
% B=(-(\overline{\omega_i},\beta_j))_{1\leq i\leq R,1\leq j\leq N}\quad \text{ and }\quad B'=(-(\overline{\omega_i},\gamma_j))_{1\leq i\leq R,1\leq j\leq M}.
% \]

Let $\cT_{S_+}$ be the twisted polynomial algebra with generators $y_{\beta_1},\cdots, y_{\beta_N},x_{\gamma_1},\cdots,x_{\gamma_M}$ associated to the skew-symmetric matrix
\begin{equation}
    S_+=\begin{pmatrix}
         A & 0 \\
         0 & A' 
    \end{pmatrix}.
\end{equation}

% The algebra $\cP$ is $\N^{L+2M}$-graded by setting 
% \[
% \deg y_{\beta_i}=(\delta_{i1},\ldots,\delta_{iN},0\ldots,0),
% \qquad
% \deg x_{\gamma_j}=(0,\ldots,0, \delta_{j1},\ldots,\delta_{jM}).
% \]
% It is straightforward to show that $\cP$ has no zero divisors. 

Let $\C[K_\mu|\mu\in P^\theta]$ be the commutative subalgebra of $\Uiv$. By Proposition~\ref{prop:assoc}, we have the algebra isomorphism 
\begin{align}
\Gr\Uiv&\cong \cT_{S_+} \otimes_\C \C[K_\mu|\mu\in P^\theta]\label{eq:gru},
\\
B_{\beta_i} &\mapsto y_{\beta_i},E_{\gamma_j}\mapsto x_{\gamma_j}\notag,
\end{align}
where the multiplication on the right-hand side is given by $K_\mu y_{\beta_i} = v^{-(\mu,\beta_i)}y_{\beta_i} K_\mu $ and $K_\mu x_{\gamma_j} = v^{(\mu,\gamma_j)}x_{\gamma_j} K_\mu$. 

\begin{lemma}[\text{cf. \cite[Proposition 1.8]{DCK90}}]
\label{lem:intclose}
Let $A$ be a commutative integrally closed domain, which is generated by $r_1,\cdots,r_s$ as a $\C$-algebra. Let $\tT$ be a twisted polynomial algebra generated by $x_1,\cdots,x_t$. Let $C$ be a $\N^{K}$-filtered algebra such that $C_0=A$, and $\Gr C=\cT\otimes_\C A$ where the multiplication is given by $r_ix_j=\mu_{ij}x_jr_i$, $\mu_{ij}\in \C^\times$. Assume that each generator $x$ of $\cT$ has a preimage $\widetilde{x}$ in $C$ such that 
\begin{equation}\label{eq:intclose}
\widetilde{x}^\ell=z_x+\text{lower degree terms}
\end{equation}
for some $z_x$ in the center $Z_C$ of $C$. Then the algebra $C$ is integrally closed.
\end{lemma}

\begin{proof}
The proof of this lemma is parallel to the proof of \cite[Proposition 1.8]{DCK90} (see also \cite[Theorem 6.5]{DCP93}).
\end{proof}

\begin{theorem}
The algebra $\Uiv$ is an integrally closed domain.
\end{theorem}

\begin{proof}
Recall that $\Uiv$ is naturally a subalgebra of $\U_v$ (cf. Section \ref{sec:qsp}). By \cite[Corollary~1.8]{DCK90}, $\U_v$ has no zero divisors. Hence $\Uiv$ has no zero divisors. We next show $\Uiv$ is integrally closed.

Clearly, $\C[K_\mu|\mu\in P^\theta]$ is integrally closed with no zero divisors. 
Recall that $\Uiv$ is $\N^{L+2M+2}$-filtered by \eqref{eq:totaldeg} and the total degree $0$ part of $\Uiv$ is exactly $\C[K_\mu|\mu\in P^\theta]$. Note that the leading term of $B_i^{[\ell]}$ is $B_i^\ell$ and hence the leading term of $B_{\beta_i}^{[\ell]}$ is $B_{\beta_i}^{\ell}$. By Theorem~\ref{thm:Fri} and its proof, $B_{\beta_i}^{[\ell]}$ are central and hence elements $B_{\beta_i},E_{\gamma_j}$ for $1\le i\le N,1\le j\le M$ satisfy the condition \eqref{eq:intclose}. Now the desired statement follows by \eqref{eq:gru} and Lemma~\ref{lem:intclose}.
\end{proof}

\subsection{The degree of $\Gr\Uiv$}\label{sec:degui}
In this subsection, we determine the degree of $\Gr \Uiv$.

Recall that $\wItau\subset \I_\circ$ is a set of representatives of $\tau$-orbits. The sublattice $P^\theta\subset P$ is a free $\mathbb{Z}$-module of rank $R=|\I\backslash \wItau|$. Take a $\mathbb{Z}$-basis $\{\omega_s^0\mid 1\leq s \leq r\}$ of $P^\theta$. Define 
\[
B=(-(\omega_s^0,\beta_j))_{1\leq s \leq r,1\leq j\leq N}\quad \text{ and }\quad B'=(-(\omega_s^0,\gamma_j))_{1\leq s \leq r,1\leq j\leq M}.
\]
Recall the skew-symmetric matrices $A$ and $A'$ from \eqref{eq:AA'}.

By Proposition~\ref{prop:assoc}, the algebra $\Gr\Uiv$ can be identified with the localized twisted polynomial algebra $\tT_{S,J(S)}=\langle y_{\beta_1},\ldots, y_{\beta_N},x_{\gamma_1},\ldots,x_{\gamma_M}, t_{\omega_1^0},\ldots, t_{\omega_r^0}\rangle$, where $\F=\C,q=v$, $S$ is the following skew-symmetric matrix
\begin{equation}\label{eq:S}
    S=\begin{pmatrix}
         A & 0 & -{^tB} \\
         0 & A' & ^tB' \\
         B & -B' & 0
    \end{pmatrix}
\end{equation}
and $J(S)=\{t_{\omega_s^0}| 1\le s \le r\}$; see Section~\ref{sec:tpa} for the convention.

Let $V$ be a free $\mathbb{Z}[2^{-1}]$-module with basis $\{u_i\mid 1\leq i\leq N\}$, and let $V'$ be a free $\mathbb{Z}[2^{-1}]$-module with basis $\{u_i'\mid 1\leq i\leq M\}$. We identify the matrix $A=(a_{ij})$ with the linear map $V\rightarrow V$ such that $u_i\mapsto \sum_{j=1}^N a_{ji}u_j $, and similarly identify $A'$ with the linear map $V'\rightarrow V'$. 

Let $'P=\mathbb{Z}[2^{-1}]\otimes_\mathbb{Z} P$. Then $'P={'P}^\theta\oplus{'P}^{-\theta}$, where $'P^\theta$ (resp., $'P^{-\theta}$) is the subgroup of $\theta$-invariants (resp., $\theta$-anti-invariants). Let $\widetilde{B}:V\rightarrow {'P}$ be the linear map such that $u_i\mapsto \beta_i$ for $1\leq i\leq N$, and $\widetilde{B}':V'\rightarrow {'P}$ be the linear map such that $u_i'\mapsto \gamma_i$ for $1\leq i\leq M$. Then $^t\widetilde{B}$ is identified with the linear map $'P\rightarrow V$ such that $\mu\mapsto \sum_{t=1}^N(\mu,\beta_t)u_t$, and $^t{\widetilde{B}'}$ is identified with $'P\rightarrow V'$ such that $\mu\mapsto \sum_{t=1}^M(\mu,\gamma_t)u_t'$. Finally, note that $B=pr_0 \circ \widetilde{B}:V\rightarrow {'P}^\theta$ where $pr_0:{'P}\rightarrow {'P}^\theta$ is the projection map, and $^tB={^t}\widetilde{B}\mid_{'P^\theta}:{'P}^\theta\rightarrow V$. Similarly, $B'=pr_0 \circ \widetilde{B}':V'\rightarrow {'P}^\theta$ and $^tB'={^t}\widetilde{B}'\mid_{'P^\theta}:{'P}^\theta\rightarrow V'$.

% Consider the linear maps
% \begin{equation}
%     M=\begin{pmatrix}
%         A & -^t\widetilde{B}
%     \end{pmatrix}:V\oplus {'P}\rightarrow V\quad \text{ and }\quad M'=\begin{pmatrix}
%         A' & -^t\widetilde{B}'
%     \end{pmatrix}:V'\oplus {'P}\rightarrow V'.
% \end{equation}
Recall the reduced expression $w_0=s_{i_1}\cdots s_{i_L}s_{j_1}\cdots s_{j_M}$ where $w_\bullet=s_{j_1}\cdots s_{j_M}$. Write $i_{L+t}=j_t$ for $1\leq t\leq M$. Note that the convex order $\gamma_1<\cdots<\gamma_M$ in $\mathcal{R}^+_\bullet$ is associated with the reduced expression $w_\bullet=s_{j_1'}\cdots s_{j_M'}$, where $s_{j_k'}=w_0w_\bullet s_{j_k}w_\bullet w_0$ for $1\leq k\leq M$. 

For $\nu\in P$, define $\bc(\nu)=(c_k(\nu))\in\Z^N$ and $\bc'(\nu)=c'_k(\nu)\in \Z^M$ by
\begin{equation}\label{eq:ckn}
    c_k(\nu)=\langle \alpha_{i_k}^\vee,\nu\rangle,\qquad \text{for }1\leq k\leq N;
\end{equation}
and 
\begin{equation}\label{eq:ckn'}
    c'_k(\nu)=\langle \alpha_{j_k'}^\vee,\nu\rangle\qquad \text{for }1\leq k\leq M.
\end{equation}
Define 
\begin{equation}\label{eq:uvv}
u(\nu)=\sum_{k=1}^Nc_k(\nu)u_k\in V\qquad \text{and}\qquad u'(\nu)=\sum_{k=1}^Mc'_k(\nu)u_k'\in V'.
\end{equation}

% For $i\in \I$, set $$I_i=\{1\leq t\leq N\mid i_t=i\} \quad \text{ and }\quad I^\bullet_i=\{1\leq t\leq M\mid j'_t=i\}.$$ For $i\in \I$, set $v_{\omega_i}=\sum_{t\in I_i}u_t-(1+w_0)\omega_i$ in $V\oplus {'P}$, and set $v^\bullet_{\omega_i}=\sum_{t\in I^\bullet_i}u'_t+(1+w_\bullet)\omega_i$ in $V'\oplus {'P}$. For any $\mu=\sum_{i\in \I} c_i\omega_i$ in $'P$, set $v_\mu=\sum_{i\in \I}c_iv_{\omega_i}$, $v^\bullet_\mu=\sum_{i\in\I}c_iv^\bullet_{\omega_i}$, $\mu^{(1)}=v_\mu+(1+w_0)\mu$, and $\mu^{(2)}=v_\mu-(1+w_\bullet)\mu$. Note that $\mu^{(1)}\in V$ and $\mu^{(2)}\in V'$.

Consider the matrices
\begin{equation}
    M=\begin{pmatrix}
        A & -^t\widetilde{B}
    \end{pmatrix}\quad \text{ and }\quad M'=\begin{pmatrix}
        A' & ^t\widetilde{B}'
    \end{pmatrix}.
\end{equation}
Denote by $M_\ell:V\oplus {'P}\rightarrow V/\ell V$ and $M'_\ell:V'\oplus {'P}\rightarrow V'/\ell V'$ the linear maps induced by $M$ and $M'$, respectively.

\begin{lemma}[\text{\cite[Lemma~10.4]{DCP93}}]\label{le:lmm}
    Suppose $\ell$ is an odd integer which is coprime to $\epsilon_i$ for $i\in\I$. Then we have:
    
    (1) For any $\mu\in P$, one has $M(u(\mu)+(1+w_0)\mu)=0$. Moreover, the kernel of the linear map $M_\ell$ is spanned by vectors $u(\omega_i)+(1+w_0)\omega_i$ ($i\in \I$) and $\ell (V\oplus {'P})$. 
    
    (2) For any $\mu\in P$, one has $M'(u'(\mu)-(1+w_\bullet)\mu)=0$. Moreover, the kernel of the linear map $M'_\ell$ is spanned by vectors $u'(\omega_i)-(1+w_\bullet)\omega_i$ ($i\in \I$) and $\ell (V'\oplus {'P})$.

    (3) For any $\mu\in P$, we have $\widetilde{B}(u(\mu))=(1-w_0)\mu$ and $\widetilde{B}'(u'(\mu))=(1-w_\bullet)\mu$.
\end{lemma}
Let $X=V\oplus V'\oplus {'P}^\theta$, and write $S_\ell$ to be the linear $X\rightarrow X/\ell X$ induced by $S$. We next determine the kernel of $S_{\ell}$. We need some preparations.

Following \cite{KL08}, define 
\begin{equation}\label{eq:Pi}
P^\imath=\{\mu\in {P}\mid \theta(\mu)=\mu+w_0\mu-w_\bullet\mu\},
\end{equation}
and $'P^\imath=\Z[2^{-1}]\otimes_{\Z}P^\imath$. Recall that $\tau_0:\I\rightarrow\I$ is the diagram involution such that $w_0s_iw_0=s_{\tau_0 i}$ for any $i\in \I$. We collect basic properties of $P^\imath$ following \cite[Proposition~9.2]{KL08} and \cite[Lemma~4.5]{Ko20}.

\begin{lemma}\label{le:pi}
    \begin{itemize}
        \item [(1)] $P^\imath$ is a free abelian group, whose rank equals to $\text{rank }\mathfrak{k}$.
        \item[(2)] One has $P^\imath=\{\mu\in P\mid \tau\tau_0\mu=\mu\}$. 
        \item[(3)] One has
        \[
        P^\imath=\{\nu\in P\mid \theta((1+w_0)\nu)=(1+w_0)\nu, \quad \theta((w_0-w_\bullet)\nu)=(w_\bullet-w_0)\nu\}.
        \]
    \end{itemize}
\end{lemma}

Let $\{i_1,i_2,\dots,i_m\}\subset \I$ be a set of representatives of $\tau\tau_0$-orbits. For $1\leq k\leq m$, take 
\begin{equation}\label{eq:nui}
\nu_k=\begin{cases}
     \omega_{i_k}+\omega_{\tau_0\tau i_k} & \text{if $\tau_0\tau i_k\neq i_k$,}
    \\
     \omega_{i_k} & \text{if $\tau_0\tau i_k=i_k$.}
\end{cases}
\end{equation}
Set $P_+^\imath=P_+\cap P^\imath$. Then, by \cite[Proposition 9.1]{KL08}, we have 
\begin{equation}\label{eq:pip}
    P^\imath=\bigoplus_{i=1}^m\Z\nu_i\qquad \text{ and }\qquad  P^\imath_+=\bigoplus_{i=1}^m\mathbb{N}\nu_i.
\end{equation}
In particular, one has $m=\text{rank }\mathfrak{k}$, thanks to Lemma \ref{le:pi} (1). 

% The following lemma can be proved by direct computation.

% \begin{lemma}
%     For any $\mu\in {'P}^{-\theta}$, one has $w_\bullet\mu=\mu$.
% \end{lemma}

% For any $\mu\in {'P}$, write $\mu=\mu_1+\mu_0$ where $\mu_0\in {'P}^\theta$ and $\mu_1\in{'P}^{-\theta}$. 

Recall $u(\nu),u'(\nu)$ from \eqref{eq:uvv}. For $1\leq i\leq m$, define 
\[
x({\nu_i})=u(\nu_i)+u'(-w_\bullet\nu_i)+(1+w_0)\nu_{i}.
\]
By Lemma \ref{le:pi} (3), we have $(1+w_0)\nu_{i}\in P^\theta$. In particular $x({\nu_i})\in X$.

\begin{lemma}\label{le:ker}

     (1) The kernel of $S_\ell:X\rightarrow X/\ell X$ is spanned by vectors $x({\nu_i})$ ($1\leq i\leq m$) and $\ell X$. 
     
     (2) The image of $S_\ell$ is isomorphic to $(\mathbb{Z}/\ell\mathbb{Z})^{2N_0}$ where $2N_0=\text{dim }\mathfrak{k}-\text{rank }\mathfrak{k}$. 
\end{lemma}

\begin{proof}
We first prove (1). Suppose that $S_\ell(v+v'+\lambda)=0$ where $v\in V$, $v'\in V'$ and $\lambda\in {'P}^\theta$. Then one has $M_\ell(v+\lambda)=0$. By Lemma \ref{le:lmm} (1), we have $\lambda=(1+w_0)\mu+\ell\lambda_1$ and $v=u(\mu)+\ell v_1$, for $\mu,\lambda_1\in {'P}$ and $v_1\in V$. Since $\lambda\in {'P^\theta}$, by replacing $\mu$ with some element in $\mu+\ell {'P}$ we may assume that $(1+w_0)\mu\in {'P^\theta}$. Similarly, by solving $M'_\ell(v'+\lambda)=0$ and applying Lemma \ref{le:lmm} (2), we deduce that $v'=u'(\mu')+\ell v'_1$ and $\lambda=-(1+w_\bullet)\mu'+\ell \lambda_2$, where $(1+w_\bullet)\mu'\in {'P^\theta}$, $v_1'\in V'$ and $\lambda_2\in {'P}$. 

For any $\gamma\in {'P}$, write $\gamma=\gamma_0+\gamma_1$ where $\gamma_0\in {'P^\theta}$ and $\gamma_1\in {'P^{-\theta}}$.

By comparing two expressions for $\lambda$ we conclude that 
\begin{equation}\label{eq:1pw}
    (1+w_0)\mu+(1+w_\bullet)\mu'=(1+w_0)\mu_0+(1+w_\bullet)\mu'_0\in \ell {'P}.
\end{equation}

By Lemma \ref{le:lmm} (3) we have 
\[
Bv-B'v'=(\widetilde{B}v-\widetilde{B}'v')_0=(1-w_0)\mu_0-(1-w_\bullet)\mu'_0.
\]
Then $S_{\ell}(v+v'+\lambda)=0$ implies that 
\begin{equation}\label{eq:1w0}
    (1-w_0)\mu_0-(1-w_\bullet)\mu_0'\in \ell ({'P}^\theta).
\end{equation}

Combing \eqref{eq:1pw} with \eqref{eq:1w0} we have 
\begin{equation}
    \mu_0+w_\bullet\mu_0'\in \ell ({'P^\theta})\quad \text{ and }\quad (w_0-w_\bullet)\mu_0\in \ell('P^\theta).
\end{equation}
Hence $\mu_0\in\ker (w_0-w_\bullet)+\ell ('P^\theta)$. Therefore we may assume that $(w_0-w_\bullet)\mu_0=0$ and $\mu_0'=-w_\bullet\mu_0$. We conclude that $(1+w_0)\mu\in {'P^\theta}$ and $(w_0-w_\bullet)\mu\in {'P^{-\theta}}$, which implies that $\mu\in P^\imath$ thanks to Lemma \ref{le:pi}. Then it is direct to see $v+v'+\lambda=u(\mu)+u'(-w_\bullet\mu)+(1+w_0)\mu+\ell x$ for some $x\in X$, so $v+v'+\lambda$ belongs to the subspace spanned by the vectors in (1). 

Conversely, we show $S_\ell(x(\nu_i))=0$ for $1\le i\le m$. It is clear that 
$$A(u(\nu_i))-{^tB}((1+w_0)\nu_i)=0.$$
By Lemma \ref{le:lmm} (2), we have $A'(u'(-w_\bullet\nu_i))+{^t\widetilde{B}'((1+w_\bullet)\nu_i)}=0$. Note that for any $\mu\in {'P^{-\theta}}$ and $j\in \I_\bullet$, one has $\langle \alpha_j^\vee,\mu\rangle=0$. This implies that $u'(\mu)=0$ by \eqref{eq:uvv}, and $^t\widetilde{B}'(\mu)=0$ by the definition of $^t\widetilde{B}$. In particular, by Lemma \ref{le:pi} (3), one has $(w_0-w_\bullet)\nu_i\in {'P^{-\theta}}$, and hence $^t\widetilde{B}'((w_0-w_\bullet)\nu_i)=0$. 
Therefore, we have 
\[
A'(u'(-w_\bullet\nu_i))-{^tB'}((1+w_0)\nu_i)
=A'(u'(-w_\bullet\nu_i))-{^t\widetilde{B}'}((1+w_\bullet)\nu_i)=0.
\]
Finally by Lemma \ref{le:lmm} (3) and Lemma \ref{le:pi} (3), we have
$$
B(u(\nu_i))-B'(u'(-w_\bullet\nu_i))=pr_0((1-w_0)\nu_i-(1-w_\bullet)\nu_i)=pr_0((w_\bullet-w_0)\nu_i)=0.
$$
We complete the proof of (1). 

% Since $\lambda\in {'P^\theta}$, we have $(1+w_0)\mu_1+\ell\lambda_{1,1}=0$, which implies $\mu_1\in\text{ker }(1+w_0)+\ell {'P}^{-\theta}$. Hence we may assume that $(1+w_0)\mu_1=0$ by replacing $\mu$ with some element in $\mu+\ell {'P}^{-\theta}$. Similarly, by solving $\overline{N}(Y+\lambda)$ and applying Lemma \ref{le:lmm} (2), we deduce that $\lambda= (1+w_\bullet)\mu'+\ell \lambda_2$ and $Y=\mu^{(2)}+\ell Y'$, for $\mu'$, $\lambda_2\in {'P}$ and $Y'\in V'$. Since $\lambda\in {'P}^\theta$, we deduce that $(1+w_\bullet)\mu'_1+\ell\lambda_{2,1}=2\mu'_1+\ell \lambda_{2,1}=0$, thanks to Lemma \ref{le:pi}. Hence by modifying $\mu'$ with certain element in $\ell ({'P}^{-\theta})$, we may assume that $\mu_1'=0$. In conclusion, we have  
% \begin{align}
% &(1+w_0)\mu_1=\mu_1'=0,\label{eq:1w0} \\
%     &(1+w_0)\mu_0+\ell \lambda_{1,0}=(1+w_\bullet)\mu_0'+\ell \lambda_{2,0}\label{eq:1pw}
% \end{align}

% Finally, we have $$BX-B'Y=(\widetilde{B}X-\widetilde{B}'Y)_0=-(1-w_0)\mu_0+(1-w_\bullet)\mu'_0,$$ thanks to Lemma \ref{le:lmm} (3). Combining with $\overline{S}(X+Y+\lambda)$, we have  
% \begin{equation}
%     -(1-w_0)\mu_0+(1-w_\bullet)\mu_0'\in \ell ({'P}^\theta).
% \end{equation}

% Therefore by \eqref{eq:1w0} and Lemma \ref{le:pi} we have
% \[
% (w_0-w_\bullet)\mu=(w_0-w_\bullet)\mu_1=(w_0-1)\mu_1=-2\mu_1=\theta(\mu)-\mu.
% \]
% This implies that $\mu\in P^\imath$. Hence we conclude $X+Y+\lambda\in\mu_0^{(1)}+\mu_1^{(2)}+(1+w_0)\mu_0+\ell W$, which completes the proof of (1). 

The proof of (2) follows by noticing that $P^\imath$ is a direct summand of $P$.
\end{proof}

\begin{theorem}\label{prop:degd}
    Suppose that $\ell$ is an odd positive integer which is coprime to $\epsilon_i$ for $i\in \I$. The algebra $\Gr \Uiv$ has degree $\ell^{N_0}$, where $N_0=(\text{dim }\mathfrak{k}-\text{rank }\mathfrak{k})/{2}$ is the number of positive roots of the reductive Lie algebra $\mathfrak{k}$.
\end{theorem}

\begin{proof}
    It follows from Lemma~\ref{le:ker} (2) and Proposition \ref{prop:twpolycenter}~(3).
\end{proof}

\subsection{The center of $\Gr\Uiv$}\label{sec:Grcenter}
In this subsection, we determine the center of $\Gr \Uiv$.

% Retain the notations as before. Recall that the set $\{i_1,i_2,\dots,i_m\}\subset \I$ is a set of representatives of $\tau\tau_0$-orbits. By \cite[Proposition 9.2]{KL08}, one has $m=\text{rank }\g^\theta$. For $1\leq k\leq m$, take 
% \[
% \nu_k=\begin{cases}
%      \omega_{i_k}+\omega_{\tau_0\tau i_k} & \text{if $\tau_0\tau i_k\neq i_k$,}
%     \\
%      \omega_{i_k} & \text{if $\tau_0\tau i_k=i_k$.}
% \end{cases}
% \]

% Recall the reduced expression $w_0=s_{i_1}\cdots s_{i_L}s_{j_1}\cdots s_{j_M}$ where $w_\bullet=s_{j_1}\cdots s_{j_M}$. Let $\beta_1<\cdots<\beta_L<\gamma_1<\cdots<\gamma_M$ be the convex order on $\mathcal{R}_+$ associated with the reduced expression. Set $i_{L+t}=j_t$ and $\beta_{L+t}=\gamma_t$, for $1\leq t\leq M$ as before.  

% For $\nu\in P^\imath$, set 
% \[
% I_\nu=\{1\leq t\leq N=L+M\mid s_{i_t}(\nu)\neq \nu\},\quad\text{and}\quad I^\bullet_\nu=\{1\leq t\leq M\mid s_{j_t}(\nu)\neq \nu\}.
% \]

% For $\nu\in P^\imath$, set 
% \begin{equation}\label{eq:wn}
% w_\nu=K_{-(1+w_0)\nu}\prod_{t\in I_\nu}F_{\beta_t}\prod_{t\in I^\bullet_\nu}E_{\gamma_t}\qquad \text{in }\Gr \Uiv.
% \end{equation}
Recall elements $\nu_i\in P^\imath_+$, $1\leq i\leq m$, from \eqref{eq:nui}. Recall \eqref{eq:ckn} and \eqref{eq:ckn'}. In particular, one has $c_k(\nu_i)\geq 0$, for $1\leq i\leq m$, $1\leq k\leq N$. For $\nu\in P_+$ and $1\leq k\leq M$, notice that
\begin{equation}\label{eq:ck'}
    c_k'(-w_\bullet\nu)=\langle \alpha_{j'_k}^\vee,-w_\bullet\nu\rangle=\langle \alpha_{j_k}^\vee,-(w_\bullet w_0)w_\bullet\nu\rangle=\langle \alpha_{j_k}^\vee,-w_0\nu\rangle.
    \end{equation}
    In particular, one has $c_k'(-w_\bullet\nu_i)\geq 0$.
\begin{proposition}\label{prop:cendeg}
    The center of $\Gr \Uiv$ is generated by the elements
    \begin{equation}\label{eq:bek}
    K_{(1+w_0)\nu_i}\prod_{t=1}^NB_{\beta_t}^{\langle \alpha_{i_t}^\vee, \nu_i\rangle}\prod_{ t=1}^ME_{\gamma_t}^{\langle \alpha_{j_t}^\vee,-w_0\nu_i\rangle},\quad 1\leq i\leq m=\text{rank }\mathfrak{k};
    \end{equation}
    \begin{equation}\label{eq:bbt}
    B^\ell_{\beta_t},\quad 1\leq t\leq N;\quad E^\ell_{\gamma_t},\quad 1\leq t\leq M; \quad K_{\ell\mu},\quad \mu\in P^\theta;
    \end{equation}
    as a $\mathbb{C}$-algebra.
\end{proposition}

\begin{proof}
It follows from Proposition \ref{prop:twpolycenter} and Lemma \ref{le:ker}.    
\end{proof}

\section{The degree and the full center}\label{sec:inii}

In this section, we first recall Kolb's construction \cite{Ko20} of central elements for iquantum groups, and show in Proposition~\ref{le:intn} that these elements lies in the integral form. We determine in Section~\ref{sec:lead} leading terms of these central elements, which coincide with elements in Proposition~\ref{prop:cendeg}. This allows us to further determine the degree and the full center of $\Uiv$.

\subsection{The Kolb-Letzter center}

The center of the iquantum group $\Ui$ for generic $q$ was studied by Kolb and Letzter \cite{KL08}, where a conceptual construction was given by Kolb \cite{Ko20} in terms of the universal $K$-matrix. In this subsection we recall the construction of Kolb.

Let $\D$ be the minimal positive integer such that $(\mu,\nu)\in \frac{1}{\D}\Z$ for all $\mu,\nu \in P$. To describe the center of $\Ui$, we extend the base field to $\C(q^{1/(2D)})$ and extend $\A$ to $\C[q^{\pm 1/(2D)}]$. By slightly abuse of notation, we still use $\Ui$ to denote $\C(q^{1/(2D)})\otimes _{\C(q^{1/2})}\Ui$ and use $\U_{\A}$ to denote the $\C[q^{\pm1/(2D)}]$-form of $\Ui$ obtained by taking the base change of De Concini-Kac form. Similar setting apply to $\Ui$ and $\Ui_{\A'}$.

Firstly, we recall the construction of the universal $R$-matrix. Let $\cOf$ be the category of finite-dimensional $\U$-modules of type 1, and let $\mathcal{U}\subset \prod _{V\in Ob(\cOf)}\text{End }(V)$ be the completion of $\U$ considered in \cite[\S~3.1]{BK19}. For $n\in\N$, define $\mathcal{U}^{(n)}\subset\prod_{V_1,\dots, V_n\in Ob(\cOf)}\text{End }(V_1\otimes \cdots \otimes V_n)$ to be the completion of the $n$-folded tensor of $\U$ considered in \emph{loc. cit.}. 
 
Let $\mathbf{f}=\oplus_{\mu\in Q^+}\mathbf{f}_\mu$ be the algebra defined in \cite[\S1.2.5]{Lus93} with generators $\theta_i,i\in \I$. Let $r:\mathbf{f}\rightarrow \mathbf{f}\otimes\mathbf{f}$ be the algebra homomorphism such that $r(\theta_i)=1\otimes \theta_i+\theta_i\otimes 1$; here $\mathbf{f}\otimes\mathbf{f}$ is equipped with the twisted product (see \cite[\S 1.2.2]{Lus93}). Let $(\cdot,\cdot)$ be the non-degenerate symmetric bilinear form on $\mathbf{f}$ introduced in \cite[Proposition~1.2.3]{Lus93} such that 
\begin{align*}
(1,1)=1,\qquad (\theta_i,\theta_j)=\frac{\delta_{ij}}{1-q_i^{-2}},
\qquad
(x,y_1y_2)=(r(x),y_1\otimes y_2).%,\qquad (x_1 x_2,y)=(x_1\otimes x_2, r(y)).
\end{align*} 
There are canonical isomorphisms $(\cdot)^{\pm}:\mathbf{f}\rightarrow \U^{\pm}$ such that $E_i=\theta_i^+, F_i=\theta_i^-$. %Let $\widetilde{\rho}:\U\rightarrow \U$ be the anti-automorphism on $\U$ such that
%\[
%\widetilde{\rho}(E_i)=-K_iF_i,\qquad \widetilde{\rho}(F_i)=-E_iK_i^{-1},\qquad \widetilde{\rho}(K_\mu)=K_{-\mu}.
%\]By \cite[Proposition 19.1.2]{Lus93}, there is a unique symmetric bilinear form on $V(\nu)$ such that $(\eta_\nu,\eta_\nu)=1,(ux,y)=(x,\widetilde{\rho}(u)y)$ for $x,y\in V(\nu),u\in \U$.

For $\nu\in P_+$, recall that $V(\nu)$ is the irreducible $\U$-module with the highest weight $\nu$ and a highest weight vector $\eta_\nu$. Let $\bB(\nu)$ be the canonical basis of $V(\nu)$ (cf.\cite[Theorem 14.4.11]{Lus93}).

Let $B=\sqcup_{\mu\in Q^+}B_\mu$ be any linear basis of $\mathbf{f}$, where $B_\mu$ is a linear basis of $\mathbf{f}_\mu$. Let $\widetilde{B}=\{\widetilde{b}\mid b\in B\}$ be the dual basis of $B$ with respect to $(\cdot,\cdot)$; that is, one has $(\widetilde{b},b')=\delta_{b,b'}$ for any $b,b'\in B$. Following \cite[Theorem 4.1.2]{Lus93}, the \emph{quasi $R$-matrix} $R$ for $\U$ admits the form 
\begin{equation}\label{eq:r}
R=\sum_{b\in B}c_bb^-\otimes \widetilde{b}^{+}\in \prod_{\mu\in  Q^+} \U_\mu^-\otimes \U_\mu^+,
\end{equation}
where $c_b\in \pm q^\Z$.

Let $\kappa\in\mathcal{U}^{(2)}$ be an element, such that $\kappa(v_\mu\otimes v_\lambda)=q^{-(\mu,\lambda)}v_\mu\otimes v_\lambda$ for weight vectors $v_\mu$ and $v_\lambda$ in any representations in $\cOf$, with weight $\mu$ and $\lambda$, respectively. Then the element 
\begin{equation}\label{eq:ur}
\mathcal{R}=R_{21}\cdot \kappa\in \mathcal{U}^{(2)}
\end{equation}
is called the \emph{universal $R$-matrix} for $\U$.

Next we recall the universal $K$-matrices and central elements of $\Ui$ constructed in \cite{Ko20}. Let $\Upsilon=\sum_{\mu\in Q^+}\Upsilon_\mu\in\prod_{\mu\in Q^+} \U^+_\mu$ be the \emph{quasi $K$-matrix} associated with the iquantum group $\Ui$ (cf. \cite{BW18a,BK19,WZ23}). 

Following \cite{BW18b, BK19}, there is a function $\xi:P\rightarrow \C(q^{1/(2\D)})^\times$ such that 
\begin{align}
\label{def:xi}
\xi(\mu+\nu) = \xi(\mu) \xi(\nu) q^{-(\mu+\theta\mu,\nu)},\qquad \forall \mu,\nu \in P.
\end{align}
Let us regard the function $\xi$ as an element in $\mathcal{U}$ such that $\xi(v_\mu)=\xi(\mu) v_\mu$ for any weight vector $v_\mu$ with weight $\mu\in P$.

%cf. \cite[Example 3.3]{BK19}.
Recall $P^\imath, P^\imath_+$ from \eqref{eq:Pi}-\eqref{eq:pip}. For any $\mu\in P^\imath_+$, there is a compatible involution $\tau_0\tau$ on $V(\mu)$ such that $\tau_0\tau(x\eta_\mu)=\tau_0\tau(x)\eta_\mu$ for all $x\in\U$. Following \cite[(3.22)]{Ko20} (cf. \cite[Corollary 7.7]{BK19}), the \emph{1-tensor universal $K$-matrix} for $\Ui$ is defined to be
\begin{align}
\mathcal{K}=\Upsilon\xi T_{w_\bullet}^{-1}T_{w_0}^{-1}\tau_0\tau.%\in\cU.
\end{align} 
The element $\mathcal{K}$ induces a well-defined linear operator on any $V(\mu)$ with $\mu\in P_+^\imath$.

Following \cite[(3.36)]{Ko20}, the \emph{2-tensor universal $K$-matrix} is defined to be 
\[
\mathscr{K}=\mathcal{R}_{21}\cdot (1\otimes\mathcal{K})\cdot \cR.%\in\mathcal{U}^{(2)}.
\]
The element $\mathscr{K}$ induces a well-defined linear operator on any $V(\lambda)\otimes V(\mu)$ with $\lambda\in P_+, \mu\in P_+^\imath$.

The \emph{quantum trace} $\text{tr}_\mu:\U\rightarrow \mathbb{C}(q)$ is defined as $\text{tr}_\mu(u)=\text{tr}_{V(\mu)}(uK_{-\rho})$ where $\text{tr}_{V(\lambda)}:\U\rightarrow \mathbb{C}(q)$ is the usual trace map and $\U$ is viewed as operators on $V(\mu)$. Note that $\text{tr}_\mu$ naturally extends to $\mathcal{U}$. 

Let $\mu\in P_+^\imath$. By \cite[Theorem~4.10]{Ko20}, the following expression 
\begin{equation}\label{eq:centeri}
    d_\mu=(\Id\otimes \text{tr}_\mu)(\mathscr{K})%\in\cU
\end{equation}
defines a central element $d_\mu$ of $\Ui$. We remark that, in Kolb's notations, the element $d_\mu$ was denoted by $q^{-(\mu,\mu+2\rho)}\tilde{l}_{1\otimes v\cdot\mathscr{K}}(\text{tr}_{V(\mu)_+,\sigma,q})$ \emph{loc. cit.}.

In the next subsections, we give a more concrete expression of $d_\mu$ in terms of PBW basis.

\subsection{Integrality of the central elements}

Starting from this subsection, we retain the notations in Section~\ref{sec:iPBW}. In particular, let us fix a reduced expression $w_0=s_{i_1}\cdots s_{i_L}s_{j_1}\cdots s_{j_M}$ where $w_\bullet=s_{j_1}\cdots s_{j_M}$. Let $\beta_1<\cdots <\beta_L<\gamma_1<\cdots<\gamma_M$ be the associated convex order on $\mathcal{R}^+$. Set $i_{L+t}=j_t$ and $\beta_{L+t}=\gamma_{t}$, for $1\leq t\leq M$. Set $N=M+L$.

For any $\mathbf{c}=(c_i)\in \N^N$, define $L(\mathbf{c})\in \mathbf{f}$ by
\begin{equation}\label{eq:lc+}
L(\bc)^+=T_{s_{i_1}\cdots s_{i_{N-1}}}(\mathbf{E}_{i_N}^{(c_N)})\cdots T_{i_1}(\mathbf{E}_{i_2}^{(c_2)})\mathbf{E}_{i_1}^{(c_1)}.
\end{equation}
Thanks to \cite[37.2.4]{Lus93} we have
\[
L(\mathbf{c})^-=\lambda_{\bc} T_{s_{i_1}\cdots s_{i_{N-1}}}(\mathbf{F}_{i_N}^{(c_N)})\cdots T_{i_1}(\mathbf{F}_{i_2}^{(c_2)})\mathbf{F}_{i_1}^{(c_1)} ,\qquad \text{where $\lambda_{\bc}\in\pm q^{\Z}$}.
\]

Here the elements $\mathbf{E}_i^{(c)}$, $\mathbf{F}_i^{(c)}$ are divided powers $\frac{\bE_i^c}{[c]_i!}, \frac{\bF_i^c}{[c]_i!}$, respectively. Then $\{L(\mathbf{c})\mid \mathbf{c}\in\N^N\}$ form a basis of $\mathbf{f}$ (cf. \cite[Proposition~40.2.1]{Lus93}). 
Recall that $\{\widetilde{L(\bc)}\mid \bc\in \N^N\}$ denotes the dual basis.
By \cite[Proposition 38.2.3]{Lus93}, one has 
\[
\widetilde{L(\mathbf{c}})=\lambda'_{\bc} \prod_{k=1}^N(q_{i_k}-q_{i_k}^{-1})^{c_k}[c_k]_{i_k}!L(\mathbf{c}),\qquad \text{for some } \lambda'_{\bc}\in\pm q^\Z.
\]

Combining with \eqref{eq:recs}, we have
\begin{equation}\label{eq:lc}
\begin{split}
    \widetilde{L(\mathbf{c})}^+=\lambda^+_\bc 
     T_{s_{i_1}\cdots s_{i_{N-1}}}(E_{i_N}^{c_N})\cdots T_{i_1}(E_{i_2}^{c_2}) E_{i_1}^{c_1} ,
    \\
    \widetilde{L(\mathbf{c})}^-=\lambda^-_{\bc}
     T_{s_{i_1}\cdots s_{i_{N-1}}}(F_{i_N}^{c_N})\cdots T_{i_1}(F_{i_2}^{c_2}) F_{i_1}^{c_1},
\end{split}
\end{equation}
for $\bc\in\N^N$ and $\lambda^+_{\bc},\lambda^-_{\bc}\in \pm q^{\Z}$.

On the other hand, by \eqref{eq:r} one has 
\begin{equation}\label{eq:rcc}
    R=\sum_{\mathbf{c}\in \N^N}c_{\bc} L(\bc)^-\otimes \widetilde{L(\bc)}^+=\sum_{\bc\in\N^N} c_{\bc}\widetilde{L(\bc)}^-\otimes L(\bc)^+,
\end{equation}
where $c_\bc\in \pm q^\Z$. Therefore, by \eqref{eq:ur} one has 
\begin{equation}\label{eq:unir}
    \mathcal{R}=\big(\sum_{\bc\in \N^N}c_{\bc}\widetilde{L(\bc)}^+\otimes L(\bc)^-\big)\cdot \kappa,\quad\text{ and }\quad \mathcal{R}_{21}=\big(\sum_{\bc\in\N^N}c_{\bc}\widetilde{L(\bc)}^-\otimes L(\bc)^+\big)\cdot \kappa.
\end{equation}
Here the first equality follows from the first equality in \eqref{eq:rcc}, and the second equality follows from the second equality in \eqref{eq:rcc}.

Let $\nu\in P^\imath_+$. For $\bc=(c_1,\ldots,c_N)\in \N^N$, set $|\bc|=\sum_{k=1}^N c_k \beta_k\in Q^+$.
By \eqref{eq:unir}, \eqref{eq:centeri} and direct computations, we have 
\begin{equation}\label{eq:dn}
    d_\nu=\sum_{\substack{b\in\mathbf{B}(\nu),\; \mu\in Q^+\\ \mathbf{c},\mathbf{c}'\in\N^N}}c_{b,\mathbf{c},\mathbf{c}',\mu}b^*(\mathscr{K}_{\mathbf{c},\mathbf{c}',\mu}\cdot b)K_{\nu_{b,\mathbf{c}',\mu}}\widetilde{L(\mathbf{c})}^-\widetilde{L(\mathbf{c}')}^+,
\end{equation}
where $b^*\in V(\nu)^*$ satisfies $b^*(b')=\delta_{b,b'}$ for $b'\in \mathbf{B}(\nu)$, $c_{b,\mathbf{c},\mathbf{c}'}\in \pm q^{\Z/(2D)}$, 
\begin{equation}\label{eq:nub}
    \nu_{b,\bc',\mu}=-\wt(b)-w_\bullet w_0\tau_0\tau(\wt(b)-|\bc'|)-\mu\in P,
\end{equation}
 and 
\begin{equation}\label{eq:kc}
\mathscr{K}_{\mathbf{c},\mathbf{c}',\mu}=L(\mathbf{c})^+\Upsilon_\mu T_{w_\bullet}^{-1}T_{w_0}^{-1}\tau_0\tau L(\mathbf{c}')^-
\end{equation}
is a linear operator on $V(\nu)$. The summation in \eqref{eq:dn} is finite by the weight consideration.

% Recall that $T_i=T_{i,+1}''$ in Lusztig notation for braid group symmetries. Write $T_i'=T_{i,+1}'$, cf. \cite[37.1.3]{Lus93}. We also use $T_i$, $T_i'$ to denote the braid group symmetry on any irreducible $\U$-modules. Then for $\mathbf{c}\in\N^N$, we have
% \[
% L'(\mathbf{c})^-=\mathbf{F}_{i_1}^{(c_1)}T'_{i_1}(\mathbf{F}_{i_2}^{(c_2)})\cdots T'_{i_1}\cdots T'_{i_{N-1}}(\mathbf{F}_{i_N}^{(c_N)}),
% \]
% \begin{equation}\label{eq:lcp}
% L(\mathbf{c})^+=\mathbf{E}_{i_1}^{(c_1)}T'_{i_1}(\mathbf{E}_{i_2}^{(c_2)})\cdots T'_{i_1}\cdots T'_{i_{N-1}}(\mathbf{E}_{i_N}^{(c_N)}).
% \end{equation}

\begin{proposition}\label{le:intn}
    For any $\nu\in P^\imath_+$, the element $d_\nu$ belongs to $\U_{\A}$. Hence $d_\nu\in \Ui_{\A'}$.
\end{proposition}

\begin{proof}

Consider the equation \eqref{eq:dn}. Thanks to \eqref{eq:lc}, $\widetilde{L(\mathbf{c})}^-$, $\widetilde{L(\mathbf{c}')}^+$, for $\mathbf{c},\mathbf{c}'\in \N^N$, belong to $\U_{\A}$. It suffices to show that $b^*(\mathscr{K}_{\mathbf{c},\mathbf{c}',\mu}\cdot b)\in \A$, for $b\in \mathbf{B}(\nu)$, $\mathbf{c},\mathbf{c'}\in \N^N$, $\mu\in Q^+$. Let $V(\nu)_{\A}\subset V(\nu)$ be the $\A$-submodule spanned by $\mathbf{B}(\nu)$. Then it suffices to show that the linear operators $\mathscr{K}_{\mathbf{c},\mathbf{c}',\mu}$ preserve $V(\nu)_{\A}$. 

% Theorem 14.4.13 \& Proposition 21.1.2 \& 
Recall from \eqref{eq:kc} the formula of $\mathscr{K}_{\mathbf{c},\mathbf{c}',\mu}$. By \cite[Theorem 41.1.3]{Lus93}, the actions of $L(\mathbf{c})^+$, $L(\mathbf{c}')^-$ preserve $V(\lambda)_\A$. The involution $\tau_0\tau$ clearly preserves $V(\lambda)_\A$. Thanks to \cite[Theorem~5.3]{BW18b}, the action of $\Upsilon_\mu$ preserves $V(\lambda)_{\A}$. By \cite[\S~5.2.1]{Lus93}, the braid group symmetries $T_{w_\bullet}^{-1}$ and $T_{w_0}^{-1}$ preserve $V(\lambda)_{\A}$. We conclude that $\mathscr{K}_{\mathbf{c},\mathbf{c}',\mu}$ preserves $V(\lambda)_{\A}$. We complete the proof.
\end{proof}

%%%%%%%%%%%%%%%%%%%%%%%%%%%%%%%%%%
\subsection{The leading terms of central elements}\label{sec:lead}
Recall from Section \ref{sec:asa} the $\N^{N+M+2}$-filtration of $\Ui_{\A'}$ and the associated graded algebra $\Gr \Ui_{\A'}$. We next determine the image $\overline{d_\nu}$ in $\Gr \Ui_{\A'}$ of the central element $d_\nu$ for $\nu\in P^\imath_+$. 

For $w\in W$, write $\eta_{w\nu}$ to be the unique element in $\mathbf{B}(\nu)$ of weight $w\nu$. For $\nu\in P_+$, recall that the element $\mathbf{c}(\nu)=(c_k(\nu))\in\N^N$ is defined by
\begin{equation}\label{eq:cnu}
    c_k(\nu)=\langle \alpha_{i_k}^\vee,\nu\rangle,\qquad \text{ for }1\leq k\leq N.
\end{equation} 
For $\mu\in P_+$, recall that $\tau_0\mu=-w_0\mu\in P_+$.

\begin{lemma}\label{le:nm}
    For $\nu\in P_+$, the element $\bc(\tau_0\nu)$ is the maximal element in the set
    \[
    \{\bc\in\mathbb{N}^N\mid \eta_\nu^*(L(\mathbf{c})^+\eta_{w_0\nu})\neq 0\},
    \]
    where $\N^N$ is endowed with the lexicographical order. Moreover $\eta_\nu^*(L(\mathbf{c}(\tau_0\nu))^+\eta_{w_0\nu})\in\pm q^\Z$.
\end{lemma}

\begin{proof}
We first show that $\eta_\nu^*(L(\mathbf{c}(\tau_0\nu))^+\eta_{w_0\nu})\in\pm q^\Z$.

For $1\leq k \leq N$, we have
\begin{equation}\label{eq:ti1}
\begin{split}
    T_{i_1}\cdots T_{i_{k-1}}\Big(\mathbf{E}_{i_k}^{(c_k(\tau_0\nu))}\Big)\cdot T_{i_1}\cdots T_{i_{k-1}}(\eta_{w_0\nu})&=T_{i_1}\cdots T_{i_{k-1}}(\mathbf{E}_{i_k}^{(c_k(\tau_0\nu))}\cdot\eta_{w_0\nu})\\&\in \pm q^\Z  T_{i_1}\cdots T_{i_{k-1}}T_{i_k}(\eta_{w_0\nu}).
\end{split}
\end{equation}

% By \cite[Lemma~39.1.2]{Lus93}, one has 
% \[
% \eta_{w_0\nu}\in \pm q^{\Z}\cdot T_{i_1}\cdots T_{i_N}(\eta_\nu).
% \]
Recall $L(\bc(\tau_0\nu))^+$ from \eqref{eq:lc+}. By applying \eqref{eq:ti1} repeatedly, one has
\[
L(\bc(\tau_0\nu))^+\eta_{w_0\nu}\in \pm q^{\Z}T_{i_1}\cdots T_{i_N}(\eta_{w_0\nu})\in \pm q^{\Z}\eta_\nu.
\]
Hence we have $\eta_\nu^*(L(\bc(\tau_0\nu))^+\eta_{w_0\nu})\in\pm q^\Z$.

We next show that for $\bc'\in\N^N$ where $\mathbf{c}'>\mathbf{c}(\tau_0\nu)$ in the lexicographical order, one has $\eta_\nu^*(L(\mathbf{c}')^+\eta_{w_0\nu})= 0$. Since $\mathbf{c}'>\mathbf{c}(\nu)$, there exists $1\leq k'\leq N$, such that $c'_{k'}>c_{k'}$ and $c'_t=c_t$ for $1\leq t<k'$. By applying \eqref{eq:ti1} we have
\[
L(\mathbf{c'})^+\eta_{w_0\nu}\in \pm q^\Z\cdot T_{i_1}\cdots T_{i_{N-1}}(\mathbf{E}_{i_N}^{(c'_N)})\cdots T_{i_1}\cdots T_{i_{k'-1}}(\mathbf{E}_{i_{k'}}^{(c'_{k'})})\cdot T_{i_1}\cdots T_{i_{k'-1}}(\eta_{w_0\nu}).
\]
% Note that
% \begin{align*}
% T_{i_1}\cdots T_{i_{k'-1}}(\mathbf{E}_{i_{k'}}^{(c'_{k'})})\cdot T_{i_1}\cdots T_{i_{k'}}(\eta_\nu)&=T_{i_1}\cdots T_{i_{k'-1}}\big( \mathbf{E}_{i_{k'}}^{(c'_{k'})}\cdot T_{i_{k'}}(\eta_{\nu})\big)\\&\in \pm q^\Z \cdot T_{i_1}\cdots T_{i_{k'-1}}(\mathbf{E}_{i_{k'}}^{(c_k')}\mathbf{F}_{i_{k'}}^{(c_{k'}(\nu))}\cdot \eta _{\nu}).
% \end{align*}
Since $\mathbf{E}_{i_{k'}}^{(c_{k'}')}\cdot \eta _{w_0\nu}=0$ by the weight consideration, we conclude that $L(\bc')^+\eta_{w_0\nu}=0$. We complete the proof. 
\end{proof}

For $\mu\in P_+$, define the element $\bc^\bullet(\nu)=(c_k^\bullet(\nu))\in \N^M$ by 
\begin{equation}\label{eq:cga}
    c^\bullet_k(\nu)=\langle\alpha_{j'_k}^\vee,-w_0\nu\rangle
    =\langle \alpha_{j_k}^\vee, -w_\bullet\nu \rangle 
    \quad \text{ for }1\leq k\leq M.
\end{equation}
We define $\widetilde{\bc}^\bullet(\nu)=(\widetilde{c}_k^\bullet(\nu))\in \N^N$ by $\tilde{c}_k^\bullet(\nu)=0$ for $1\leq k\leq L$ and $\tilde{c}_{L+k}^\bullet(\nu)=c_k^\bullet(\nu)$ for $1\leq k\leq M$.

\begin{lemma}\label{le:nmb}
    For $\nu\in P_+$, the element $\tilde{\mathbf{c}}^\bullet(\tau_0\nu)$ is the maximal element in the set
    \[
    \{\bc\in\N^N\mid \eta_{w_\bullet\nu}^*(L(\mathbf{c})^-\eta_\nu)\neq 0\},
    \]
    where $\N^N$ is endowed with the lexicographical order. Moreover $\eta_{w_\bullet\nu}^*(L(\tilde{\mathbf{c}}^\bullet(\tau_0\nu))^-\eta_\nu)\in\pm q^\Z$.
\end{lemma}

\begin{proof}
    Take $\bc=(c_t)\in \N^N$, such that $\eta_{w_\bullet\nu}^*(L(\mathbf{c})^-\eta_\nu)\neq 0$. By the weight consideration, we have 
    \begin{align}\label{eq:bwnu}
    \sum_{1\leq t\leq N}c_t\beta_t=\nu-w_\bullet\nu,
    \end{align}
    which is a nonnegative linear combination of positive roots in $\mathcal{R}_\bullet$. For $1\leq t\leq L$, $\beta_t\in \mathcal{R}^+\backslash\mathcal{R}_\bullet^+$, and hence \eqref{eq:bwnu} forces that $c_t=0$. Note that $T_{w_0w_\bullet}T_j=T_{j'}T_{w_0w_\bullet}$ and $T_{w_0w_\bullet}(F_j)=F_{j'}$ for $j\in \bI$. Thus, we have 
    \[
    L(\bc)^-=T_{w_0w_\bullet}\big(T_{j_1}\cdots T_{j_{M-1}}(\mathbf{F}_{j_M}^{(c_N)})\cdots T_{j_1}(\mathbf{F}_{j_2}^{(c_{L+2})})\mathbf{F}_{j_1}^{(c_{L+1})}\big),
    \]
    and 
    \[
    \eta_{w_\bullet\nu}^*(L(\mathbf{c})^-\eta_\nu)\in \pm q^{\Z}\eta_{w_0\nu}^*
    \Big(T_{j_1}\cdots T_{j_{M-1}}(\mathbf{F}_{j_M}^{(c_N)})\cdots T_{j_1}(\mathbf{F}_{j_2}^{(c_{L+2})})\mathbf{F}_{j_1}^{(c_{L+1})}\eta_{w_\bullet w_0\nu}\Big).
    \]
    Then the remaining proof follows from the similar arguments as in the proof of Lemma~\ref{le:nm}.
\end{proof}

    Recall the total degree $d$ from \eqref{eq:totaldeg}.

    \begin{lemma}\label{le:pbwt}
        For any $x\in\Ui$, let us write $x$ in terms of the PBW bases of $\Ui$ and $\U$ respectively, that is,
        \begin{align*}
            x&=\sum_{\mu\in P^\theta,\bc\in \N^N,\bc'\in\N^M} c_{\mu,\bc,\bc'}K_\mu\orProd_{1\leq t\leq N}B_{\beta_t}^{c_t}\orProd_{1\leq t\leq M}E_{\gamma_t}^{c'_t}=\sum_{\mu\in P,\bd,\bd'\in\N^N}d_{\mu,\bd,\bd'}K_\mu\orProd_{1\leq t\leq N}F_{\beta_t}^{d_t}\orProd_{1\leq t\leq N}E_{\beta_t}^{d'_t}.
        \end{align*}
        Then for any $(\mu_1,\bc_1,\bc'_1)\in P^\theta\times\N^N\times \N^M$ such that $c_{\mu_1,\bc_1,\bc'_1}\neq 0$ and $d(M_{\mu_1,\bc_1,\bc'_1})$ is maximal among all $d(M_{\mu,\bc,\bc'})$ for $(\mu,\bc,\bc')\in P^\theta\times \N^N\times \N^M$ with $c_{\mu,\bc,\bc'}\neq 0$, one has 
        $$
        c_{\mu_1,\bc_1,\bc_1'}=d_{\mu_1,\bc_1,\tilde{\bc}_1'}.
        $$ 
        Here $\tilde{\bc}_1'=(\tilde{c}'_{1,k})\in\N^N$ is defined by $\tilde{c}_{1,k}'=0$ for $1\leq k\leq L$ and $\tilde{c}'_{1,L+k}=c_{1,k}'$ for $1\leq k\leq M$.
    \end{lemma}

    \begin{proof}
        Write $\beta_{L+t}=\gamma_t$ for $1\le t\le M$. Given $\mu\in P^\theta,\bc\in \N^N,\bc'\in\N^M$, by \cite[Corollary 4.3]{LYZ25}, we can write
        \[
        K_\mu\orProd_{1\leq t\leq N}B_{\beta_t}^{c_t}\orProd_{1\leq t\leq M}E_{\beta_{L+t}}^{c'_t}=K_\mu\orProd_{1\leq t\leq N}F_{\beta_t}^{c_t}\orProd_{1\leq t\leq M}E_{\beta_{L+t}}^{c'_t}+\widetilde{R},
        \]
        where $\widetilde{R}$ is a linear combination of elements  
        \begin{equation}\label{eq:kfe}
        K_{\mu_1}\orProd_{1\leq t\leq N}F_{\beta_t}^{d_t}\orProd_{1\leq t\leq N}E_{\beta_t}^{d'_t}
        \end{equation}
        for $\mu_1 \in P$, $\bd,\bd'\in\N^N$ with
        \[
        \sum_{t=1}^L(d_t-d_t')\het^\imath(\beta_t)<\sum_{t=1}^Lc_t\het^\imath(\beta_t)=h^\imath(M_{\mu,\bc,\bc'}).
        \]
        Hence we can write $\widetilde{R}=\widetilde{R}_1+\widetilde{R}_2$, such that $\widetilde{R}_1$ is a linear combination of elements in the form of \eqref{eq:kfe} with $d_t'=0$ for $1\leq t\leq L$ and $\sum_{t=1}^Ld_t\het^\imath(\beta_t)<\sum_{t=1}^Lc_t\het^\imath(\beta_t)$, and $\widetilde{R}_2$ is a linear combination of elements in the form of \eqref{eq:kfe} with $d_t'\neq 0$ for some $1\leq t\leq L$. The lemma follows immediately.
    \end{proof}

We are now ready to compute the image $\overline{d_\nu}$ in $\Gr \Ui_{\A'}$.

\begin{proposition}\label{prop:dv}
    Take $\nu\in P_+^\imath$. The image of $d_{\nu}$ in the associated graded algebra $\Gr \Ui_{\A'}$ is 
    \[
    \overline{d_{\nu}}=c_\nu K_{-(1+w_0)\nu}\prod_{k=1}^NB_{\beta_k}^{\langle \alpha_{i_k}^\vee,-w_0\nu\rangle }\prod_{k=1}^ME_{\gamma_k}^{\langle \alpha_{j_k}^\vee,\nu\rangle},
    \]
    where $c_\nu\in \pm q^{\Z/D}$. 
\end{proposition}

\begin{proof}
    For $\mathbf{c}\in\N^N$, write $|\mathbf{c}|=\sum_{k=1}^N c_k\beta_k\in Q^+$. Then $\wt\; L(\mathbf{c})^+=\wt\;\widetilde{L(\mathbf{c})}^+=|\mathbf{c}|$ and $\wt\; L'(\mathbf{c})^-=\wt\;\widetilde{L(\mathbf{c})}^-=-|\mathbf{c}|$. Recall \eqref{eq:kc}. Firstly, we have the following claim.

    \textbf{Claim:} \emph{ $b^*(\mathscr{K}_{\mathbf{c},\mathbf{c}',\mu}\cdot b)=0$ for any $b\in \mathbf{B}(\nu),\mu\in Q^+$, unless either $|\mathbf{c}|<\nu-w_0\nu$, or $|\mathbf{c}|=\nu-w_0\nu$ and $|\mathbf{c}'|=\nu-w_\bullet \nu$. In the second case, one has $b^*(\mathscr{K}_{\mathbf{c},\mathbf{c}',\mu}\cdot b)=0$ if $(\mu,b)\neq (0,\eta_\nu)$. }

    We prove this claim. If $|\mathbf{c}|>\nu-w_0\nu$, then by the weight consideration one has $L(\mathbf{c})^+\mid_{V(\nu)}\equiv 0$, so $\mathscr{K}_{\mathbf{c},\mathbf{c}',\mu}\mid _{V(\nu)}\equiv 0$. 
    Hence, it suffices to assume $|\mathbf{c}|=\nu-w_0\nu$. Then $L(\mathbf{c})^+x\mid _{V(\nu)}\equiv 0$ for any $x\in \U$ with $\wt \;x\in Q^+$. Hence $L(\mathbf{c})^+\Upsilon_\mu\mid _{V(\nu)}\equiv 0$ if $\mu\neq 0$. Also note that $L(\mathbf{c})^+\cdot v=0$ unless $\wt\; v=w_0\nu$ and $\wt\;L(\mathbf{c})^+\cdot v=\nu$. Hence $b^*(\mathscr{K}_{\mathbf{c},\mathbf{c}',\mu})=0$ unless $(\mu,b)=(0,\eta_\nu)$ and 
    \begin{equation}\label{eq:tt}
        T^{-1}_{w_\bullet}T^{-1}_{w_0}\tau_0\tau L'(\mathbf{c}')^-\eta_\nu=c\eta_{w_0\nu},\quad \text{for }c\in \C(q).
    \end{equation}
    Note that $\tau\tau_0w_\bullet\nu=w_\bullet\tau\tau_0\nu=w_\bullet\nu$. Hence \eqref{eq:tt} is equivalent to $L(\mathbf{c}')^-\eta_\nu=c'\eta_{w_\bullet\nu}$ for $c'\in \C(q)$, which implies that $|\mathbf{c}'|=\nu-w_\bullet \nu$. We complete the proof of the claim.

    Now suppose that $|\mathbf{c}|=\nu-w_0\nu$, $|\mathbf{c}'|=\nu-w_\bullet\nu$, $\mu=0$ and $b=\eta_\nu$. The coefficient $c$ in \eqref{eq:tt} actually belongs to $\pm q^\Z$ thanks to (a) in the proof of Lemma \ref{le:nm}. By the above arguments we have 
    \begin{equation}\label{eq:bkc}
        b^*(\mathscr{K}_{\mathbf{c},\mathbf{c}',\mu}\cdot b)=\eta_\nu^*(L(\mathbf{c})^+T_{w_\bullet}^{-1}T_{w_0}^{-1}\tau_0\tau L'(\mathbf{c})^-\eta_\nu)=c\eta_\nu^*(L(\mathbf{c})^+\eta_{w_0\nu})\eta_{w_\bullet\nu}^*(L'(\mathbf{c}')^-\eta_\nu),
    \end{equation}
    where $c\in \pm q^\Z$. Recall $\nu_{b,\bc',\mu}$ from \eqref{eq:nub}. It follows from the direct computation that under our assumption one has $\nu_{b,\bc',\mu}=-(1+w_0)\nu$. 

    % Let us write $$\bd_\nu=d\big(K_{-(1+w_0)\nu}\prod_{t=1}^NB_{\beta_t}^{c_t(\nu)}\prod_{t=1}^ME_{\gamma_t}^{c^\bullet_{t}(\nu)}\big)\in \N^{N+M+2}.$$ By Lemma \ref{le:nm}, Lemma \ref{le:nmb} and the above arguments, we have $b^*(\mathscr{K}_{\mathbf{c},\mathbf{c}',\mu}\cdot b)\in \pm q^\Z$ for $\bc=\bc(\nu)$, $\bc'=\tilde{\bc}^\bullet(\nu)$ and $(\mu,b)=(0,\eta_\nu)$. It suffices to show that 
    % \begin{equation}\label{eq:dnv}
    % d_\nu-c'_{\eta_\nu,\mathbf{c}(\nu),\tilde{\mathbf{c}}^\bullet(\nu),0}\eta_\nu^*(\mathscr{K}_{\mathbf{c}(\nu),\tilde{\mathbf{c}}^\bullet(\nu),0}\cdot \eta_\nu)K_{-(1+w_0)\nu}\prod_{t=1}^N B_{\beta_t}^{c_t(\nu)}\prod_{t=1}^ME_{\gamma_t}^{c^\bullet_t(\nu)}\quad \text{belongs to } \Ui_{<\bd_\nu}.
    % \end{equation}

    Let us write $d_\nu$ in terms of PBW basis of $\Ui$, that is,
    \begin{equation}
        d_\nu=\sum_{\mu\in P^\theta,\bc\in\N^N,\bc'\in\N^M}c_{\mu,\bc,\bc'}K_\mu\orProd_{1\leq t\leq N}B_{\beta_t}^{c_t}\orProd_{1\leq t\leq M}E_{\gamma_t}^{c'_t},
    \end{equation}
    where $c_{\mu,\bc,\bc'}\in\A'$. 
    
   By \eqref{eq:lc} and \eqref{eq:dn}, we can write
    \begin{equation}
        d_{\nu}=\sum_{\substack{b\in\mathbf{B}(\nu),\; \mu\in Q^+\\ \mathbf{c},\mathbf{c}'\in\N^N}}c'_{b,\mathbf{c},\mathbf{c}',\mu}b^*(\mathscr{K}_{\mathbf{c},\mathbf{c}',\mu}\cdot b)K_{\nu_{b,\mathbf{c}',\mu}}\orProd_{1\leq t\leq N} F_{\beta_t}^{c_t}\orProd_{1\leq t\leq N}E_{\beta_t}^{c'_t},
    \end{equation}
    where $c'_{b,\mathbf{c},\mathbf{c}',\mu}\in \pm q^{\Z/(2D)}$. Moreover, combining the claim, \eqref{eq:bkc}, Lemma \ref{le:nm} and Lemma \ref{le:nmb}, we conclude that 
    \begin{equation}\label{eq:dbu}
    d_\nu=c_\nu'K_{-(1+w_0)\nu}\orProd_{1\leq t\leq N}F_{\beta_t}^{c_t(\tau_0\nu)}\orProd_{1\leq t\leq M}E_{\gamma_t}^{c^\bullet_t(\tau_0\nu)}+\text{l.o.t.},
    \end{equation}
    where $c_\nu'\in \pm q^{\Z/D}$ and l.o.t. is a linear combination of elements $$K_{\mu}\orProd_{1\leq t\leq N} F_{\beta_t}^{c_t}\orProd_{1\leq t\leq N}E_{\beta_t}^{c'_t}$$ for $\mu\in P$, $c_t,c'_t\in \N$, and 
    \[
    (c_N,\cdots,c_1,c'_{N},\cdots c'_{L+1}, \het_\bullet(\bc,\bc'),h^\imath(\bc,\bc'))<d(M_{-(1+w_0)\nu,\bc(\tau_0\nu),\bc^\bullet(\tau_0\nu)})=\bd_\nu
    \]
    in $\N^{N+M+2}$.

    By Lemma \ref{le:pbwt} and \eqref{eq:dbu}, we conclude that 
    \[
    d_\nu\in \Ui({\leq \bd_\nu})\quad \text{ and }\quad d_\nu-c_\nu'K_{-(1+w_0)\nu}\prod_{t=1}^NB_{\beta_t}^{c_t(\tau_0\nu)}\prod_{t=1}^ME_{\gamma_t}^{c^\bullet_t(\tau_0\nu)}\in \Ui{(<\bd_\nu)}.
    \]
    By definitions we have $c_t(\tau_0\nu)=\langle \alpha_{i_t}^\vee,\nu\rangle$ for $1\leq t\leq N$, and $c_t^\bullet(\tau_0\nu)=\langle \alpha_{j_t'}^\vee,\nu\rangle=\langle \alpha_{j_t}^\vee,w_0w_\bullet\nu\rangle$ for $1\leq t\leq M$. Since $\nu\in P^\imath$, by Lemma \ref{le:pi} (3), we deduce that $\nu-w_0w_\bullet\nu$ is $\theta$-antiinvariant. Combining with $\theta(\alpha_{j_t}^\vee)=\alpha_{j_t}^\vee$, we conclude that $\langle \alpha^\vee_{j_t},w_0w_\bullet\nu\rangle=\langle \alpha^\vee_{j_t},\nu\rangle$, for $1\leq t\leq M$.
    
    We complete the proof.
\end{proof}

The following lemma was obtained from \cite[Theorem~8.3]{KL08}. It is now a direct consequence of Proposition \ref{prop:dv}.

\begin{lemma}\label{lem:dmuleading}
    For $\nu\in P^\imath_+$, the element $d_{\nu}\in\Ui$ can be written uniquely as
    \[
    d_{\nu}=d_{\nu,\overline{\nu}}+\text{l.o.t.},
    \]
    where $0\neq d_{\nu,\overline{\nu}}\in \U_{\overline{\nu}}$, $\overline{\nu}=w_0\nu-w_\bullet \nu$, and $\text{l.o.t.}\in\sum_{\text{ht}^\imath(\mu)<\text{ht}^\imath(\overline{\nu})}\U_{-\mu}$.
\end{lemma}

\begin{proof}
    It suffices to show that 
    \[
    K_{-(1+w_0)\nu}\prod_{k=1}^NF_{\beta_k}^{\langle \alpha_{i_k}^\vee,-w_0\nu\rangle }\prod_{k=1}^ME_{\gamma_k}^{\langle \alpha_{j_k}^\vee,\nu\rangle}\in \U_{\overline{\nu}}.
    \]

    By Lemma \ref{le:nm} and Lemma \ref{le:nmb} we have 
    \[
    \deg \prod_{k=1}^NF_{\beta_k}^{\langle \alpha_{i_k}^\vee,-w_0\nu\rangle }=w_0\nu-\nu,\quad \text{and}\quad \deg \prod_{k=1}^ME_{\gamma_k}^{\langle \alpha_{j_k}^\vee,\nu\rangle}=\nu-w_\bullet\nu.
    \]
    Hence 
    \[
    \deg \big(K_{-(1+w_0)\nu}\prod_{k=1}^NF_{\beta_k}^{\langle \alpha_{i_k}^\vee,-w_0\nu\rangle }\prod_{k=1}^ME_{\gamma_k}^{\langle \alpha_{j_k}^\vee,\nu\rangle}\big)=w_0\nu-w_\bullet\nu=\overline{\nu},
    \]
    which completes the proof.
\end{proof}

\subsection{The degree of $\Uiv$}

Recall from Section~\ref{sec:iPBW} that $\Zi$ is a (localized) polynomial algebra and $\Uiv$ is a free module over $\Zi$. Thus, the submodule $\Ziv\subset \Uiv$ is free over $\Zi$.

Thanks to Proposition \ref{le:intn}, the central element $d_\nu$ belongs to $\Ui_{\A'}$ for any $\nu\in P^\imath_+$. 
For $\nu\in P^\imath_+$, we still use the notation $d_\nu$ to denote the image of $d_\nu$ in $\Uiv$. In particular, $d_{\nu_i}$, $1\leq i\leq m$, are central elements in $\Uiv$. 

\begin{proposition}\label{prop:dzl}
We have $\dim_{\Zi} \Ziv \geqslant \ell^{m}$ where $m=\rank \mathfrak{k}$.
\end{proposition}

\begin{proof}
Note that $P^\imath_{\ell}=\{\mu= \sum_{i=1}^m k_i\nu_i| 0\leq k_i < \ell \}$ forms a set of representatives for $P^\imath\slash \ell P^\imath$. For $\mu=\sum_{i=1}^m k_i\nu_i$, denote $d^{\mu}=d_{\nu_1}^{k_1}\cdots d_{\nu_m}^{k_m}$.

It suffices to show that the set $\{d^\mu | \mu \in P^\imath_{\ell}\}$ is linearly independent over $\Zi$. Write $\ov{\mu}= \theta\mu-\mu=w_0\mu-w_\bullet\mu$ for $\mu\in P^\imath$. Note that by our construction of $P^\imath_{\ell}$, we have $\ov{\mu_1}-\ov{\mu_2}\not\in \ell P$ for any $\mu_1\neq\mu_2,\mu_1,\mu_2\in P^\imath_{\ell}$.

Suppose otherwise that there exist $c_\mu\in \Zi$ not all equal $0$ such that
\begin{align}\label{eq:sumdv}
\sum_{\mu\in P^\imath_{\ell}} c_\mu d^\mu=0.
\end{align}
Let $\widetilde{c}_\mu$ be a lift of $c_\mu$ in $\Ui_{\A'}$. Recall that $\Ui_{\A'}$ is a free $\A'$-module. Then \eqref{eq:sumdv} implies that in $\Ui_{\A'}$
\begin{align}\label{eq:sumdmu}
\sum_{\mu\in P^\imath_{\ell}} \widetilde{c}_\mu d^\mu\in(q-v)\Ui_{\A'}.
\end{align}
Without loss of generality, we assume that when writing in terms of the PBW basis of $\Ui_{\A'}$, $\widetilde{c}_\mu$ does not have any nonzero summand lying in $(q-v)\Ui_{\A'}$. In particular, due to the definition of $\Zi$, any weight component of $\widetilde{c}_\mu$ has a weight in $\ell P$.

We consider the set $\mathscr{P}$ of weights $\lambda$ such that the weight $\lambda$ component of $\widetilde{c}_\mu d^\mu$ is nonzero for some $\mu\in P^\imath_{\ell}$.  Let $\lambda_0\in \mathscr{P}$ such that its relative height $\het^\imath(\lambda_0)$ is maximal. By the assumption and Lemma~\ref{lem:dmuleading}, $\lambda_0$ must have the form $\lambda_0=\ov{\mu_0} + \ell \nu_0$ for some $\mu_0\in P^\imath_{\ell},\nu_0\in P$.

Write $\mu_0=p_1 \nu_1+\cdots+p_m \nu_m$. By Lemma~\ref{lem:dmuleading}, we have 
\[
d^{\mu_0}=d_{\nu_1,\ov{\nu_1}}^{p_1}\cdots d_{\nu_m,\ov{\nu_m}}^{p_m} + o_{\mu_0}, \qquad \text{ in } \U_{\A'}
\]
where the leading terms $d_{\nu_1,\ov{\nu_1}}^{p_1}\cdots d_{\nu_m,\ov{\nu_m}}^{p_m}$ is nonzero when descending to $\Uiv$ and the relative height of weights of $o_{\mu_0}$ are strictly smaller than $\het^\imath(\ov{\mu_0})$. 

Due to \eqref{eq:sumdmu}, the remaining summands $\sum\limits_{\mu\in P^\imath_{\ell},\mu\neq \mu_0} \widetilde{c}_\mu d^\mu$ must have a nonzero component of weight $\lambda_0$. Since $\ov{\mu_0}-\ov{\mu}\not\in  \ell P$ for any $\mu\in P^\imath_{\ell}, \mu\neq \mu_0$, the weights of the leading terms of $\widetilde{c}_\mu d^\mu$ cannot be $\lambda_0$ for any $\mu\neq \mu_0$. Therefore, there exists $\mu\in P^\imath_{\ell}$ such that the lower term $\widetilde{c}_\mu o_\mu$ of $\widetilde{c}_\mu d^\mu$ has a nonzero component of weight $\lambda_0$. However, by Lemma~\ref{lem:dmuleading}, this forces that $\widetilde{c}_\mu d^\mu$ has a weight component of weight $\lambda_1$ such that $\het^\imath(\lambda_1)>\het^\imath(\lambda_0)$. This contradicts the choice of $\lambda_0$.

Therefore, we proved that the set $\{d^\mu | \mu \in P^\imath_{\ell}\}$ is linearly independent over $\Zi$. This implies that $\dim_{\Zi} \Ziv \geqslant \ell^{m}$.
\end{proof}

\begin{theorem}\label{thm:deg}
  The degree of $\Ui_v$ is $\ell^{N_0}$, where $N_0=(\text{dim }\mathfrak{k}-\text{rank }\mathfrak{k})/2$ is the number of positive roots of $\mathfrak{k}$. Therefore $\dim_{\Zi} \Ziv = \ell^{m}$ where $m=\rank \k$.
\end{theorem}

\begin{proof}
    Write $Q(\Uiv)=Q(\Ziv)\otimes _{\Ziv}\Uiv$. By Theorem \ref{thm:findim} and Proposition \ref{prop:dzl}, we have
    \[
    \dim _{Q(\Ziv)}Q(\Uiv)=\dim_{Q(\Zi)}Q(\Uiv)/\dim _{Q(\Zi)}Q(\Ziv)\leq \ell ^{2N_0}.
    \]
    Hence $\deg \Uiv\leq \ell ^{N_0}$.

    On the other hand, by \cite[(1.3.2)]{DCKP93a} and Theorem~\ref{prop:degd}, we have 
    \[
    \deg \Uiv\geq \deg \Gr \Uiv=\ell ^{N_0}.
    \]
    Hence we conclude that $\deg \Uiv=\ell^{N_0}$. In particular, the inequality in Proposition \ref{prop:dzl} is actually an equality.
\end{proof}

Combining with Proposition \ref{prop:repalg}, we immediately obtain the following upper-bound on the dimensions of irreducible representations of $\Uiv$.

\begin{corollary}\label{cor:dimv}
    Set $N_0$ to be the number of positive roots for $\k$ as before. One has the follows.

    \begin{itemize}
        \item [(1)]For any $V\in\Irr\Uiv $, one has $\dim V\leq \ell^{N_0}$.
        \item [(2)] The set
        \[
        \X^0=\{x\in\X\mid \dim V=\ell^{N_0} \text{ for any }V\in \Irr \Ui_{v,x}\}
        \]
        is a non-empty Zariski open subset of $\X$.
    \end{itemize}
\end{corollary}

\begin{remark}
When the underlying symmetric pair is of type AI, a class of finite-dimensional modules with dimension $\ell^{N_0}$ for $\Uiv$ were constructed in \cite[Theorem 6.2]{Wen20}; these modules are analogs of baby Verma modules for quantum groups. It is interesting to study when such modules are irreducible.
\end{remark}

\subsection{The full center of $\Uiv$}
Recall that $d_\nu\in \Ui_{\A'}$ from Proposition \ref{le:intn}. We still write $d_\nu$ to denote its image in $\Uiv$. Let $Z^\imath_1$ denote the subspace of $\Ziv$ spanned by $d_{\nu}$ for $\nu\in P^\imath_+$. %The following lemma is a direct consequence of \cite[Remark 4.12]{Ko20}.

\begin{lemma}\label{le:poly}
    The subspace $Z_1^\imath$ is a polynomial algebra generated by $d_{\nu_i}$ for $1\leq i\leq m$.
\end{lemma}

\begin{proof}
By \cite[Corollary 4.11 and Remark 4.12]{Ko20}, $Z_1^\imath$ is a subalgebra generated by $d_{\nu_i}$ for $1\leq i\leq m$. By \cite[Theorem 9.3]{KL08} and \cite[Proposition 4.9]{Ko20}, the center of $\Ui$ at generic $q$ is generated by $d_{\nu_i},1\leq i\leq m$ as a polynomial algebra. Hence, by Proposition~\ref{le:intn}, $d_{\nu_i},1\leq i\leq m$ generate a polynomial central subalgebra of $\Ui_{\A'}$ and this implies that $Z_1^\imath$ is a polynomial algebra.
\end{proof}

%     By definition, $\{d_\nu\mid \nu\in P^\imath_+\}$ forms a basis of $Z_1^\imath$. For $\mu_1,\mu_2\in P_+^\imath$, suppose that 
%     \begin{equation}\label{eq:dd}
%     d_{\mu_1}d_{\mu_2}=\sum_{\mu\in P_+^\imath}m_{\mu_1,\mu_2}^\mu d_\mu\qquad \text{in $\Ui$},
%     \end{equation}
%     where $m_{\mu_1,\mu_2}^{\mu}\in \C(q^{1/(2D)})$.
%     Suppose that $m_{\mu_1,\mu_2}^\mu$ has a pole at $q-v$ for some $\mu_1,\mu_2,\mu\in P_+^\imath$. By multiplying certain power of $(q-v)$ to \eqref{eq:dd} and applying the base change to $\Uiv$, we obtain a linear combination of $d_\mu$ which equals to zero. This is a contradiction. 
    
%     Hence, $m_{\mu_1,\mu_2}^\mu$ does not have any pole at $v$, and we still write $m_{\mu_1,\mu_2}^\mu$ to denote its image in $\C$ by abuse of notations. In particular, \eqref{eq:dd} holds in $\Uiv$ and $Z_1^\imath$ is a subalgebra. To show that it is generated by $d_{\nu_i}$, for $1\leq i\leq m$, note that $m_{\mu_1,\mu_2}^\mu=0$ unless the total degree $d(d_\mu)\leq d(d_{\mu_1 }d_{\mu_2})$, thanks to Proposition \ref{prop:dv}. Moreover we have $m_{\mu_1,\mu_2}^{\mu_1+\mu_2}\in \pm v^{\Z/D}$. Since $P_+^\imath$ is the monoid spanned by $d_{\nu_i}$, we conclude the proof. 

We call $Z_1^\imath$ the \emph{Kolb-Letzter center} of $\Uiv$. The main theorem in this subsection is the follows. 

\begin{theorem}\label{thm:icenter}
    The center $\Ziv$ of $\Uiv$ is generated by $\Zi$ and $Z^\imath_1$.
\end{theorem}

In order to prove this theorem, we need the following lemma.

\begin{lemma}\label{le:degc}
    Let $R=\cup_{m\in \N}R_m$ be a filtered $\mathbb{C}$-algebra with no zero divisors. Let $\Gr R$ be the associated graded algebra. Suppose that there are central elements $z_i\in Z(R)$, $1\leq i\leq n$ such that $\overline{z_i}$  generate the center of $\Gr R$ as a $\C$-algebra. Then $z_i$, $1\leq i\leq n$, generate the center $Z(R)$ as a $\C$-algebra. 
\end{lemma}

\begin{proof}
For $f\in R$, recall that we write $\deg f\in\mathbb{N}$ for the minimal $m$ such that $f\in R_m$. Let $Z'$ be the $\C$-algebra generated by $z_i$, $1\leq i\leq n$. 

Take $f\in Z(R)$. We show that $f\in Z'$. Note that $\overline{f}\in Z(\Gr R)$. By our assumption, we have 
\[
\overline{f}=\sum_{\mathbf{m}=(m_1,\cdots,m_n)\in \mathbb{N}^n}c_{\mathbf{m}}\overline{z_1}^{m_1}\cdots \overline{z_n}^{m_n},\qquad c_{\mathbf{m}}\in\mathbb{C}.
\]
Since $\overline{f}$ is homogeneous in $\Gr R$ of degree $\deg f$. We may assume that $c_{\mathbf{m}}=0$ unless $\deg (z_1^{m_1}\cdots z_m^{m_n})=\deg f$.

Let $$f'=\sum_{\mathbf{m}=(m_1,\cdots,m_n)\in \mathbb{N}^n}c_{\mathbf{m}}{z_1}^{m_1}\cdots {z_n}^{m_n}\qquad \text{in }R.$$
Then it is clear that $f'\in Z'$ and $\overline{f'}=\overline{f}$. Hence $\deg (f-f')<\deg f $, and $f-f' \in Z(R)$. Then $f\in Z'$ by induction on $\deg f$.
\end{proof}

\begin{proof}[Proof of Theorem \ref{thm:icenter}]
    Fix a reduced expression $w_0=s_{i_1}\cdots s_{i_L}s_{j_1}\cdots s_{j_M}$ such that $w_\bullet=s_{j_1}\cdots s_{j_M}$. Let $\beta_1<\cdots<\beta_L<\gamma_1<\cdots<\gamma_M$ be the associated convex ordering on $\mathcal{R}^+$. Recall from Section~\ref{sec:iPBW} the root vectors $B_{\beta_k}$, $E_{\gamma_t}$, $F_{\gamma_t}$ in $\Ui_{\A'}$, for $1\leq k\leq L$ and $1\leq t\leq M$. Let $\{\omega_{1}^0,\cdots,\omega_{R}^0\}$ be a $\mathbb{Z}$-basis of $P^\theta$. By Proposition~\ref{prop:PBWZi}, $\Zi$ is generated by $\Fri(\un{B_{\beta_k}})$, $\Fri(\un{E_{\gamma_t}})$,  $\Fri(\un{F_{\gamma_t}})$, and $K_{\ell\omega_s^0}$, for $1\leq k\leq L$, $1\leq t\leq M$, and $1\leq s \leq r$, as a $\mathbb{C}$-algebra. 
    
    Since $\Gr \Uiv$ can be obtained from $\Uiv$ by $N+M+2$-steps of taking associated graded algebras, where in each step one has a $\mathbb{N}$-filtered algebra, by Lemma \ref{le:poly} and Lemma \ref{le:degc}, it suffices to show that elements 
    \begin{equation}\label{eq:odnu}
    \qquad \ov{d_{\nu_i}},\quad 1\leq i\leq m;
    \end{equation}
    
    \begin{equation}\label{eq:ofri}
    \overline{\Fri(\un{B_{\beta_k}})}, \overline{\Fri(\un{E_{\gamma_t}})}, \overline{\Fri(\un{F_{\gamma_t}})},K_{\ell\mu},\quad 1\leq k\leq L,1\leq t\leq M,\mu\in P^\theta ;
    \end{equation}
    generate the center of $\Gr\Uiv$ as a $\C$-algebra. By Lemma \ref{le:frbb}, elements in \eqref{eq:ofri} coincide with elements in \eqref{eq:bbt}. Note that $\tau_0$ permutes the set $\{\nu_1,\cdots \nu_m\}$. By Proposition \ref{prop:dv}, $\ov{d_{\tau_0\nu_i}}$ coincides with with the element in \eqref{eq:bek} up to nonzero scalars. Thus, by Proposition \ref{prop:cendeg}, elements in \eqref{eq:odnu}-\eqref{eq:ofri} generate the center of $\Gr\Uiv$ as desired. %(with $\nu_i$ replaced by $\tau_0\nu_i=-w_0\nu_i$)
    
    The proof is completed.
\end{proof}

\section{Parametrization of simple modules}\label{sec:par}

In this section, we relate simple modules of $\Uiv$ with the twisted conjugacy classes of $G$.

\subsection{Poisson orders}\label{sec:po}

In the work of De Concini-Kac-Procesi \cite{DCKP92}, the Frobenius center $Z_0$ of $\U_v$ plays a vital role on the parametrization of the simple modules of $\U_v$. The key structure in their approach was later axiomatized in \cite{BG03} under the name of the \emph{Poisson orders}, which we recall. 

A Poisson order is a triple $(R,C ,\partial)$, where $R$ is an associative $\C$-algebra, $C$ is a central subalgebra of $R$ such that $R$ is a finitely generated $C$-module, and $\partial: C\rightarrow \text{Der}_{\mathbb{C}}(R)$ is a $\mathbb{C}$-linear map such that the map
\[
\{z_1,z_2\}=\partial_{z_1}(z_2),\quad \text{for }z_1,z_2\in C,
\]
defines a Poisson bracket on $C$. Let $\mathcal{V}={\text{MaxSpec }}C$ be the Poisson variety associated with $C$.

The main theorem in \cite{BG03} states that irreducible $R$-modules are parametrized by \emph{symplectic cores} of $\mathcal{V}$. For any ideal $I$ of $C$, the \emph{Poisson core} $\cP(I)$ of $I$ is defined to be the maximal Poisson ideal of $C$ which is contained in $I$. Define an equivalence relation $\sim$ on $\mathcal{V}$ by setting $\mathfrak{m}\sim\mathfrak{n}$ if $\cP(\mathfrak{m})=\cP(\mathfrak{n})$, for $\mathfrak{m},\mathfrak{n}\in\mathcal{V}$. The equivalence classes of $\sim$ are called symplectic cores of $\mathcal{V}$. In the case when $\mathcal{V}$ is a smooth Poisson variety with Zariski locally closed symplectic leaves, \cite[Proposition~3.5]{BG03} shows that symplectic leaves coincide with symplectic cores. For $\mathfrak{m}\in\mathcal{V}$, set $R_\mathfrak{m}=R/\mathfrak{m}R$. Then \cite[Theorem~4.2]{BG03} shows that 
$$
R_{\mathfrak{m}}\cong R_{\mathfrak{n}}$$ as finite-dimensional algebras, when $\mathfrak{m}$ and $\mathfrak{n}$ belong to the same symplectic cores of $\mathcal{V}$. As a consequence, the surjective map 
\[
\text{Irr} (R)\longrightarrow \mathcal{V},
\]
sending an irreducible $R$-module, $M$, to its $Z$-character $\text{Ann}_Z(M)$, has isomorphic fibres over the symplectic cores of $\mathcal{V}$.

The following proposition follows from the general example considered in \cite[Section 2.2]{BG03}. We record the proof for readers' convenience.

\begin{proposition}\label{prop:Porder}
$(\Uiv, \Zi, \partial^\imath)$ is a Poisson order where the derivation $\partial^{\imath}$ is defined by
\begin{align}
\label{eq:der}
\partial^\imath:\Zi\rightarrow \mathrm{Der}_\C(\Uiv),\quad z\mapsto \partial^{\imath}_{z},\qquad \partial^{\imath}_{z}(x)=\pi_v^{\imath}\Big(\frac{[\tz,\tx]}{\ell^2 (qv^{-1}-1)}\Big),
\end{align}
where $\tz,\tx$ are preimages of $z,x$ in $\Ui_{\A'}$ under the map $\pi_v^{\imath}$.
\end{proposition}

\begin{proof}
By Theorems~\ref{thm:Fri} and \ref{thm:findim}, $\Zi$ is central subalgebra of $\Uiv$ and $\Uiv$ is finite-dimensional over $\Zi$. Note that $\Ui_{\A'}$ is free over $\A'$, the kernel of $\pi_v^\imath$ is $(q-v)\Ui_{\A'}$. Since $z\in \Zi$ is central in $\Uiv$, $[z,x]=0$ and hence $[\tz,\tx]\in (q-v)\Ui_{\A'}$. This shows that the derivation in \eqref{eq:der} is well-defined. Comparing \eqref{eq:der} with \eqref{eq:iPoissonv}, it is clear that $ \partial^{\imath}_{z_1}(z_2)=\{z_1,z_2\}$ for $z_1,z_2\in \Zi$.
\end{proof}   

Recall from Section~\ref{sec:smQSP} the fibre algebras $\Ui_{v,y}$ for $y\in \X$.

\begin{corollary}\label{cor:sympl}
 Let $x,x'\in \X$. If $x,x'$ lie in the same symplectic core of $\X$, then one has an isomorphism $\Ui_{v,x}\cong\Ui_{v,x'}$ as $\C$-algebras.
\end{corollary}

\subsection{Symplectic leaves}

Recall that $\X=K^\perp\backslash G^*$. In this subsection, we follow \cite{LY08} to study the Poisson structure of $\X$ and determine its symplectic leaves. 

Let $D=G\times G$ and $\mathfrak{d}=\g\oplus \g$ be the associated Lie algebra. Let $\mathfrak{d}$ equip with the nondegenerate symmetric bilinear form
\[
\langle (x_1,x_2),(y_1,y_2)\rangle=\langle\langle x_1,y_1\rangle\rangle-\langle\langle x_2,y_2\rangle\rangle,\qquad \text{for }x_1x_2,y_1,y_2\in\mathfrak{g},
\]
where $\langle\langle \cdot,\cdot\rangle\rangle$ is the Killing form on $\mathfrak{g}$. Let $\g_{\Delta}\subset \mathfrak{d}$ be the diagonal subalgebra, and $\g^*\subset \mathfrak{b}^+\oplus\mathfrak{b}^-$ be the Lie algebra of the dual Poisson Lie group $G^*$. Then $\mathfrak{d}=\mathfrak{g}_\Delta+\mathfrak{g}^*$ is the \emph{standard Lagrangian splitting}. Let 
\[
R=\frac{1}{2}\sum_{j=1}^n\xi_j\wedge x_j\in\wedge^2\mathfrak{d}
\]
be the associated $r$-matrix, where $\{x_1,\cdots,x_n\}$ is a basis of $\mathfrak{g}_\Delta$ and $\{\xi_1,\cdots,\xi_n\}$ is a dual basis of $\g^*$.

Let $G_\theta=\{(g,\theta(g))\mid g\in G\}$ be the closed subgroup of $D$. Its Lie algebra $\g_\theta=\{(X,\theta(X))\mid X\in\mathfrak{g}\}$ is clearly a Lagrangian subalgebra of $\mathfrak{d}$, that is, $\mathfrak{g}_\theta^\perp=\mathfrak{g}_\theta$ with respect to $\langle\langle\cdot,\cdot\rangle\rangle$. The right action of $D$ on $G_\theta\backslash D$ induces a linear map $\kappa:\mathfrak{d}\rightarrow \chi^1( G_\theta\backslash D)$ from $\mathfrak{d}$ to the space $\chi^1( G_\theta\backslash D)$ of vector fields on $G_\theta\backslash D$.

Let $\Pi_{G_\theta\backslash D}=(\wedge^2\kappa)(R)$ be the bivector field on $G_\theta\backslash D$. By \cite[Theorem~2.3]{LY08}, $\Pi_{G_\theta\backslash D}$ is a Poisson bivector field on $G_\theta\backslash D$. One has an isomorphism 
\begin{equation}\label{eq:gtd}
    G_\theta\backslash D\overset{\sim}{\longrightarrow} G;\qquad G_\theta(g_1,g_2)\mapsto \theta(g_1)^{-1}g_2,\quad\text{for $g_1,g_2\in G$}.
\end{equation}
Let $\pi_\theta$ be the Poisson structure on $G$ which is the pushforward of $\Pi_{G_\theta\backslash D}$ under the isomorphism.

Recall that $K^\perp\subset D$ is the identity component of $G^*\cap G_\theta$. Then the map 
\begin{equation}\label{eq:xg}
    \varphi: \X=K^\perp\backslash G^*\longrightarrow G;\qquad K^\perp(b_1,b_2)\mapsto \theta(b_1)^{-1}b_2,
\end{equation}
is well-defined. For $g\in G$, let $\mathcal{C}_\theta(g)=\{\theta(h)^{-1}gh\mid h\in G\}\subseteq G$ be the $\theta$-twisted conjugacy class. Recall the Poisson structure $\Pi_\X$ on $\X$ from Section \ref{sec:limit}. We write $g\sim_{\theta}g'$ if $g'\in\mathcal{C}_\theta(g)$. Let $H^{\theta0}=\{t\theta(t)\mid t\in H\}$ be the identity component of $H^\theta=\{t\in H\mid \theta t=t\}$. 

We are now ready to state the main result of this subsection. 

\begin{proposition}\label{prop:syml}
    The map $\varphi:(\X,\Pi_\X)\rightarrow (G,\pi_\theta)$ is Poisson, and is a $2^{|\I_\bullet|}$ to $1$ covering map onto $\theta(U^+) H^{\theta0} U^-$. The symplectic leaves of $\X$ are connected components of $\varphi^{-1}(\mathcal{C}_\theta(g))$ for $g\in G$.
\end{proposition}

\begin{proof}
    Under the isomorphism \eqref{eq:gtd}, $G$ is a $D$-homogeneous space, where the right $D$-action is given by $g\cdot_\theta(g_1,g_2)=\theta(g_1)^{-1}gg_2$ for $g,g_1,g_2\in G$. Note that $G^*$ is a subgroup of $D$. The map $\varphi$ is clearly $G^*$-equivariant. Equip $D$ with the Poisson structure associated with the standard splitting (cf. \cite[(2.7)]{LY08}). By definition $G^*$ is a Poisson subgroup of $D$. The image of $\varphi$ is the $G^*$-orbit of the identity element in $G$, which is known to be a Poisson submanifold of $(G,\pi_\theta)$ and a $G^*$-Poisson homogeneous space, thanks to \cite[\S~2.2]{LY08}. It follows from the definition that $\Phi_{\X}(K^\perp\cdot e)=0$ and $\pi_\theta(e)=0$. Therefore we conclude that $\varphi$ is Poisson.

    It is clear from \eqref{eq:Gdual} that the image of $\varphi$ is $\theta(U^+) H^{\theta0} U^-$. To show that $\varphi$ is a $2^{|\I_\bullet|}$ to $1$ covering map onto the image, we note that $
    \Im \varphi\cong (G^*\cap G_\theta)\backslash G^*.
    $
    Therefore the fibre of $\varphi$ is isomorphic to $(G^*\cap G_\theta)/(G^*\cap G_\theta)^\circ$, the component group of $G^*\cap G_\theta$. By direct computation, one has 
    \[
    G^*\cap G_\theta\cong U_{{\I_\bullet}}\times H_\theta
    \]
    as varieties, where $U_{{\I_\bullet}}$ is the unipotent radical of the parabolic subgroup associated with $\I_\bullet$, and $H_\theta=\{t\in H\mid \theta(t)=t^{-1}\}$. We next count the cardinality of $H_\theta/H_\theta^\circ$, where $H_\theta^\circ$ is the identity component of $H_\theta$. 
    
    Let $H_\bullet=\{\alpha_j^\vee(a)\mid j\in \I_\bullet,a\in \C^\times\}$, and let $E=\{h\in H_\bullet\mid h^2=1\}$. It is clear that $|E|=2^{|\I_\bullet|}$ and $E\subseteq H_\theta$. By \cite[Prop~2.2]{Spr83}, we have $H_\theta^\circ=\{t\theta(t)^{-1}\mid t\in H\}$, which is generated by $\alpha_i^\vee(\lambda_i)\theta(\alpha_i^\vee(\lambda_i^{-1}))$, for $\lambda_i\in\C^\times$ and $i\in\I_{\circ,\tau}$. 
    
    Take two elements $h_1,h_2$ in $E$, such that $h_1\neq h_2$. We deduce that $h_1$ and $h_2$ are not in the same connected component. Suppose not, then $h_1h_2^{-1}\in H_\theta^\circ$. This is a contradiction since $H_\theta^\circ\cap H_\bullet=\{1\}$, due to $\theta(\alpha_i^\vee(\lambda_i^{-1}))=\alpha_{\tau i}^\vee(\lambda_i)t'$ for some $t'\in H_\bullet$. 
    
    On the other hand, we \emph{claim} that any element $t$ in $H_\theta$ can be written as $t=t_0e$ for some $t_0\in H_\theta^\circ$ and $e\in E$, and hence $t$ lies in the same connected component as $e$. Let us write $t=\prod_{i\in\I}\alpha_i^\vee(\lambda_i)$. Then $\theta(t)=t^{-1}$ implies that $\lambda_i=\lambda_{\tau i}$ for $i\in \I_\circ$. Set $t_0=\prod_{i\in\I_{\circ,\tau}}\alpha_i^\vee(\lambda'_i)\theta(\alpha_i^\vee({\lambda'_i}^{-1}))\in H_\theta^\circ$, where $\lambda_i'=\lambda_i$ if $\tau i\neq i$ and $\lambda_i'$ is a square root of $\lambda_i$ if $\tau i=i$. Then it is clear that $tt_0^{-1}\in H_\theta\cap H_\bullet=E$. We complete the proof for the claim.

    Therefore, the cardinality of $H_\theta/H_\theta^\circ$ equals $|E|=2^{|\bI|}$ as desired.

    The last statement on symplectic leaves of $\X$ follows from \cite[Proposition 2.1 (3)]{Lu14}.
\end{proof} 

In particular, symplectic leaves of $\X$ are locally closed in the Zariski topology. Hence on $\X$ we will not distinguish symplectic leaves and symplectic cores.

Let $H_\theta^\circ=\{t\theta(t)^{-1}\mid t\in H\}$ be the identity component of $\{t\in H\mid \theta(t)=t^{-1}\}$.

\begin{lemma}\label{le:gg'}
    Let $g,g'\in \Im \varphi$, and $g\sim_{\theta} g'$. Then there exists $t\in H_\theta^\circ$, such that $t^2=1$, and $g$, $tg't$ are in the same symplectic leaf of $(G,\pi_\theta)$.
\end{lemma}

\begin{proof}
 Take $w=w_\bullet w_0$ and $\theta'=\text{Ad}_{\dot{w}^{-1}}\circ\theta$. By the definition of $\theta$, $\theta'$ stabilizes $B^{\pm}$ and $H$. By definition, $H=H^{\theta0}H_\theta^\circ$ and $H^{\theta0}$ can be identified with $T_{w\theta'}$ in \cite{Lu14}. Note that the Poisson isomorphism 
    \[
    (G,\pi_\theta)\rightarrow (G,\pi_{\theta'}),\qquad g\mapsto \dot{w}^{-1}g,
    \]
    intertwines the $\theta$-twisted conjugation and the $\theta'$-twisted conjugation. By \eqref{eq:xg}, $\Im \varphi\subset \dot{w}B^+\dot{w}^{-1} B^-$. Applying this isomorphism (cf. \cite[Section~6]{Lu14}) and \cite[Proposition~3.5]{Lu14}, there exists $t\in H_\theta^\circ$ such that $g$, $tg'\theta(t)^{-1}=tg't$ are in the same symplectic leaf of $(G,\pi_\theta)$. 
    
    It remains to show that $t^2=1$. Suppose that 
    $$
    g'=\varphi((u'_+h',h'^{-1}u'_-))=\theta(u'_+h')^{-1}h'^{-1}u'_-.%\subset B^+B^-.
    $$
    Since $tg't$ is in the same symplectic leaf of $g\in\Im\varphi$ and $\Im \varphi$ is a Poisson submanifold of $(G,\pi_\theta)$, we conclude that $tg't\in \Im\varphi$. By \cite[Proposition 3.2]{So24}, $\Im\varphi\subset B^+B^-$, and under the canonical projection $B^+B^-\rightarrow H$, $\Im\varphi$ is sent to $H^{\theta0}$. Hence, we have $t^2\theta(h')^{-1}h'^{-1}=\theta(h)^{-1}h^{-1}$ for some $h\in H$. This implies that $t^2\in H^{\theta0}$. On the other hand, $t^2\in H_\theta^\circ$, and hence we conclude that $t^2=1$. 
\end{proof}

The preimage $\varphi^{-1}(\mathcal{C}_\theta(g))$ of the twisted conjugacy class is not necessarily connected. However, as shown in the proofs of Proposition~\ref{prop:syml} and Lemma~\ref{le:gg'}, different connected components can be obtained from each other by a discrete group action.

\begin{proposition}\label{prop:dimsym}
    For $x\in \X$, the dimension of the symplectic leaf of $\X$ containing $x$ equals to 
    \[
    \dim \mathcal{C}_\theta(\varphi(x))-\dim \mathfrak{g}+\dim \mathfrak{k}.
    \]
\end{proposition}

\begin{proof}
    By Proposition \ref{prop:syml}, the desired dimension equals the dimension of the symplectic leaf of $(G,\pi_\theta)$ containing $\varphi(x)$. By \cite[Theorem~1.1]{Lu14} and applying the isomorphism in \cite[Section~6]{Lu14} by setting $w=w_0w_\bullet$, the desired dimension equals
    \[
    \dim \mathcal{C}_\theta(\varphi(x))-l(w_0w_\bullet)-\dim \ker(1+\theta).
    \]
    Here we view $\theta$ as a linear map on $\mathfrak{h}=\text{Lie }H$. Note that 
    \[
    \dim \mathfrak{k}=l(w_0)+l(w_\bullet)+\text{dim rank}(1+\theta).
    \]
    Hence one has 
    \[
    \dim \g-\dim \mathfrak{k}=l(w_0w_\bullet)+\dim \ker(1+\theta).
    \]
    We complete the proof.
\end{proof}

For $x\in \X$, write $\cO_x\subset \X$ for the symplectic leaf containing $x$.

\begin{corollary}\label{cor:maxdim}
    For $x\in\X$, one has
    \[
    \max_{x\in \X}\dim \cO_x ={\dim\mathfrak{k}-\rank \mathfrak{k}}.
    \]
\end{corollary}

\begin{proof}
    Thanks to Proposition \ref{prop:dimsym}, it suffices to show that 
    $$
    \max_{g\in G} \dim \mathcal{C}_\theta(g)=\dim \mathfrak{g}-\rank \mathfrak{k}.
    $$
    Let $B_0$ be a $\theta$-stable maximal Borel subgroup of $G$, which contains a $\theta$-stable maximal torus $H_0$. Let $U_0\subset B_0$ be the unipotent radical of $B_0$. Then $\theta(U_0)=U_0$. By \cite[Lemma 5 (i)]{Spr06}, any element $g\in G$ is $\theta$-conjugate to an element in $B_0$. We may assume that $g=ut\in B_0$ where $u\in U_0$ and $t\in H_0$. For any $u't'\in B_0$ where $u'\in U_0$ and $t'\in H_0$, it is clear that 
    \[
    (u't')(ut)\theta(u't')^{-1}\in t H_0^{-\theta}U_0,
    \]
    where $H_0^{-\theta}=\{h\theta(h)^{-1}\mid h\in H_0\}$ is the identity component of the subgroup $\{h\in H_0\mid \theta(h)=h^{-1}\}$. Let $\cdot_\theta$ denote the $\theta$-twisted conjugate cation. Then one has $B_0\cdot_\theta g\subset tH_0^{-\theta}U_0$. In particular the dimension of $B_0\cdot_\theta g$ is at most $\dim H_0^{-\theta}+\dim U_0$. Since the subgroup $H_0^\theta=\{h\in H_0\mid \theta(h)=h\}$ is a maximal torus of $K$, we have 
    \[
    \dim H_0^{-\theta}=\rank \g-\rank \k.
    \]
    Let $\text{Stab}_{B_0}(g)$ (resp., $\text{Stab}_{G}(g)$) be the stablizer of $g$ under the $\theta$-twisted conjugate action of $B_0$ (resp., $G$). Then one has
    \[
    \dim \text{Stab}_{G}(g)\geq \dim \text{Stab}_{B_0}(g)\geq \dim B_0-\dim H_0^{-\theta}-\dim U_0=\rank \k.
    \]
    Hence, we have 
    \[
    \dim \mathcal{C}_\theta(g)\leq \dim\g-\rank\k.
    \]
    
    We next show that the equality can be reached. Let $\mathcal{R}_0\subset \text{Hom}(H_0,\C^*)$ be a root system associated with $H_0$. Take $\gamma\in \text{Hom}(\C^*,H_0)$, such that the pairing $\langle \gamma,(1+\theta)\alpha\rangle$ is nonzero for any $\alpha\in \mathcal{R}_0$. Let $p$ be a sufficiently large odd integer, and let $h_0=(\gamma\theta(\gamma))(e^{\frac{2\pi\sqrt{-1}}{p}})$. Then $h_0 \in H_0^\theta$ is a regular element in $H_0^\theta$ such that $h_0^p=1$. Let $\mathfrak{h}_0^\theta$ be the Lie algebra of $H_0^\theta$.

% (Since $p$ is sufficiently large, there are sufficiently many distinct $p$-th roots of unity and hence such regular element exists.)

    Then $\text{Stab}_G(h_0)=\{g\in G| \theta(g)h_0= h_0 g \}$ and its Lie algebra is given by $\mathfrak{stab}_\g(h_0)=\{x\in \g| \theta(x) = \text{Ad}_{h_0}(x)\}$. For any $x\in \mathfrak{stab}_\g(h_0)$, write $x=x_0+x_1$ where $x_0\in \mathfrak{k}, x_1\in \mathfrak{p}=\{X\in \g\mid \theta(X)=-X\}$. Since $h_0\in H_0^\theta$, $\text{Ad}_{h_0}$ preserves both $\mathfrak{k},\mathfrak{p}$; in particular, $\text{Ad}_{h_0}(x_0)\in \mathfrak{k}$ and $\text{Ad}_{h_0}(x_1)\in \mathfrak{p}$. Now, $x\in \mathfrak{stab}_\g(h_0)$ forces that
    \[
    x_0=\text{Ad}_{h_0} (x_0), \qquad - x_1 = \text{Ad}_{h_0} (x_1),
    \]
    Since $h_0$ is a regular element in the maximal torus $H_0^\theta$ of $K$ and $x_0\in \mathfrak{k}$, the first identity above implies that $x_0\in \mathfrak{h}_0^\theta$. Since $h_0^p=1$, the operator $\text{Ad}_{h_0}$ on $\mathfrak{p}$ cannot have any nonzero eigenvector with the eigenvalue $-1$ and hence $x_1=0$. Thus, we have showed that 
    \[
    \mathfrak{stab}_\g(h_0)=\mathfrak{h}_0^\theta.
    \]
    Therefore, $\dim \mathfrak{stab}_\g(h_0) = \dim \mathfrak{h}_0^\theta =\rank \k$ as desired. 
\end{proof}

\subsection{Equivalence of blocks}

Recall the map $\varphi:\X\rightarrow G$ in \eqref{eq:xg}. For $g,g'\in G$, we write $g\sim_\theta g'$ if $g$ and $g'$ are in the same $\theta$-twisted conjugacy class. In this subsection, we establish the following refined version of Corollary~\ref{cor:sympl}.

\begin{theorem}\label{thm:thetaconjugacy}
   Let $x,x'\in \X$. If $\varphi(x)\sim_\theta\varphi(x')$, then one has an isomorphism $\Ui_{v,x}\cong\Ui_{v, x'}$ as $\C$-algebras.
\end{theorem}

\begin{proof}
Thanks to Lemma \ref{le:gg'}, there exists $t\in H_{\theta}^\circ$ with $t^2=1$ such that $\varphi(x)$, $t\varphi(x')t$ are in the same symplectic leaf of $(G,\pi_\theta)$. In particular, one has $t\varphi(x')t=\varphi(x'')$ for some $x''\in\X$. Firstly, we show that $\Ui_{v,x'}\cong \Ui_{v,x''}$ as $\C$-algebra.   

Define a $\mathbb{C}(q^{1/2})$-algebra isomorphism $c_t:\U\rightarrow \U$ by $c_t(E_i)=\alpha_i(t)E_i$, $c_t(F_i)=\alpha_i(t)F_i$, and $c_t(K_\mu)=K_\mu$, for $i\in \I$ and $\mu\in P$. Note that $c_t$ is an algebra isomorphism since $\alpha_i(t)^2=\alpha_i(t^2)=1$ for $i\in \I$. 

We \emph{claim} that $c_t(\Ui)=\Ui$. It suffices to show that $c_t(x)\in \Ui$ for $x$ being the generators of $\Ui$. For $x\in\{E_i,F_i,K_\mu\mid i\in \I_\bullet,\mu\in P^\theta\}$, it is clear that that $c_t(x)\in\{x,-x\}\subset \Ui$. For $i\in \I_\circ$, since $\text{wt }(T_{w_\bullet}(E_{\tau i}))=\theta(\alpha_i)$, we have 
\[
c_t(T_{w_\bullet}(E_{\tau i}))=\theta(\alpha_i)(t)^{-1}T_{w_\bullet}(E_{\tau i})=\alpha_i(\theta(t)^{-1})T_{w_\bullet}(E_{\tau i})=\alpha_i(t)T_{w_\bullet}(E_{\tau i}).
\]
Hence, we conclude that $c_t(B_i)=\alpha_i(t)B_i\in \Ui$. The claim is proved.

It is clear that $c_t$ preserves the integral form $\U_{\A}$, and hence it induces a $\mathbb{C}$-algebra isomorphism $c_t:\Ui_v\rightarrow \Ui_v$. Moreover, it is direct to verify that $c_t(\Zi)=\Zi$, and hence it induces a morphism $\widetilde{c_t}:\X\rightarrow \X$. Then $c_t$ induces an isomorphism $\Ui_{v,x'}\cong \Ui_{v,\widetilde{c_t}(x')}$.

On the other hand, denote $\widetilde{t}=(t,t)\in G^*$. Then $\widetilde{c_t}(K^\perp x)=K^\perp \text{Ad}_{\widetilde{t}}(x)$ for any $x\in G^*$. In particular, one has $\varphi(\widetilde{c_t}(x))=t\varphi(x)t$, for any $x\in \X$. Thus, in order to show $\Ui_{v,x'}\cong\Ui_{v,x''}$, it remains to prove the following statement.

(a) For $y,y'\in\X$ with $\varphi(y)=\varphi(y')$, we have $\Ui_{v,y}\cong\Ui_{v,y'}$ as $\mathbb{C}$-algebras.

By the proof of Proposition \ref{prop:syml}, we can write $y=K^\perp(hu_+,h^{-1}u_-)$ and $y'=K^\perp(\delta hu_+,\delta h^{-1}u_-)$, where $\delta,h\in H, u_{\pm}\in U^{\pm}$, $\theta(\delta)=\delta$, and $\delta^2=1$. Let us define a $\C(q^{1/2})$-algebra isomorphism $r_\delta:\U\rightarrow \U$ by $r_\delta(E_i)=\alpha_i(\delta)E_i$, $r_\delta(F_i)=F_i$, and $r_\delta(K_\mu)=\mu(\delta)K_\mu$, for $i\in \I$ and $\mu\in P$. By the similar consideration as the case of $c_t$, we conclude that $r_\delta$ induces a $\C$-algebra isomorphism $r_\delta:\Ui_v\rightarrow \Uiv$ and $r_\delta(\mathfrak{m}_y)=\mathfrak{m}_{y'}$. Hence we obtain the $\C$-algebra isomorphism $\Ui_{v,y}\cong \Ui_{v,y'}$. This completes the proof of (a).

Therefore, we have showed that $\Ui_{v,x'}\cong \Ui_{v,x''}$, where $\varphi(x)$, $\varphi(x'')$ are in the same symplectic leaf of $(G,\pi_\theta)$. Thanks to Proposition \ref{prop:syml}, we can take $x'''\in \X$ such that $\varphi(x''')=\varphi(x'')$, and $x'''$, $x$ are in the same symplectic leaf of $\X$. By (a), we have $\Ui_{v,x'''}\cong \Ui_{v,x''}$. 
By Proposition~\ref{prop:Porder} we have $\Ui_{v,x'''}\cong\Ui_{v,x}$. Finally, combining all of the above arguments, it is clear that $\Ui_{v,x'}\cong \Ui_{v,x}$ as $\C$-algebras.
\end{proof}

\section{Dimensions of simple modules}\label{sec:dim}

Recall from Corollary~\ref{cor:dimv} that $\dim V\leq \ell^{(\dim \mathfrak{k}-\rank \mathfrak{k})/2}$ for any $V\in \Irr \Uiv$. The goal of this section is to prove the following theorem.

\begin{theorem}\label{thm:genrep}
    Let $x\in \X$ and $V\in\Irr \Ui_{v,x}$. If the $\theta$-twisted conjugacy class $\mathcal{C}_\theta(\varphi(x))$ has the maximal dimension ($=\dim \mathfrak{g}-\text{rank }\mathfrak{k}$), then $\dim V=\ell^{(\dim \mathfrak{k}-\rank \mathfrak{k})/2}$.
\end{theorem}

%We need some preparations.

\subsection{Poisson centers}

For a commutative Poisson algebra $R$ over $\C$, we write 
$$
\PZ(R)=\big\{f\in R\mid \{f,g\}=0,\forall g\in R\big\}
$$ 
to denote the subalgebra of Casimir elements. The subalgebra $\PZ(R)$ is called the \emph{Poisson center} of $R$. 

Given an $n\times n$ skew-symmetric integer matrix $H=(h_{ij})$, the (log-canonical) Poisson algebra $\P_H=\C[x_i\mid 1\leq i\leq n\}$ is defined to be the polynomial algebra of variables $x_i$, $1\leq i\leq n$, as a commutative algebra, with the Poisson bracket defined by
\[
\{x_i,x_j\}=h_{ij}x_ix_j,\quad \text{for }1\leq i,j\leq n.
\]
For any subset $J\subset\{1,\cdots,n\}$, we write $\P_{H,J}=\P_H[x_j^{-1}\mid j\in J]$. 

The Poisson algebra $\P_{H,J}$ is the semi-classical limit of the localized twisted polynomial algebra $\tT_{H,J}$ in Section \ref{sec:tpa} with $\F=\A$. Then one has the canonical isomorphism 
\[
\C\otimes _{\A}\tT_{H,J}\cong \P_{H,J}
\]
as Poisson algebras. Here $\C$ is viewed as an $\A$-module by $q^{1/2}\mapsto 1$, and the Poisson bracket on the left hand side is given by \eqref{def:SCPoisson}.

For any $\ba=(a_1,\cdots,a_n)\in\Z^n$, we denote by $x^{\ba}=x_1^{a_1}\cdots x_n^{a_n}$ the element in $\C[x_i^{\pm1}\mid 1\leq i\leq n]$. Consider $H$ as a linear function $H:\Z^n\rightarrow\Z^n$. Let $K$ be the kernel of $H$. The following lemma is straightforward to prove.

\begin{lemma}\label{le:pcp}
Let $J\subset \{1,\cdots,n\}$ be any subset. The set  
\[
\{x^{\mathbf{a}}\mid \mathbf{a}\in K,\;a_i\geq0\;\text{ for }i\not\in J\}
\]
forms a $\C$-basis of the $\PZ(\P_{H,J})$. \end{lemma}

One can define filtrations on Poisson algebras. An ($\N$-)filtered Poisson algebra $R=\cup_{i\in \mathbb{N}}R_i$ is a commutative $\C$-Poisson algebra $R$ with subspaces $R_i$, $i\in \mathbb{N}$, such that $\{R_i,R_j\}\subset R_{i+j}$ and $R_iR_j\subset R_{i+j}$, for $i,j\in \mathbb{N}$. For a filtered Poisson algebra $R=\cup_{i\in \mathbb{N}}R_i$, its associated graded algebra $\Gr R=\oplus_{i\in\N}R_{i}/R_{i-1}$ carries a natural Poisson structure. For $r\in R$, write $\deg r$ to be the minimal number $i$ such that $r\in R_i$, and write $\overline{r}\in R_i/R_{i-1}$ to denote its image in $\Gr R$, as usual. The following theorem is the Poisson analogue of Lemma \ref{le:degc}, whose proof will be omitted.

\begin{lemma}\label{le:fpa}
    Let $R=\cup_{i\in\N}R_i$ be a filtered Poisson algebra and an integral domain. Let $z_i\in \PZ(R)$, for $1\leq i\leq n$. Suppose that $\overline{z_i}$, $1\leq i\leq n$, generate $\PZ(\Gr R)$ as a $\C$-algebra. Then $z_i$, $1\leq i\leq n$, generate $\PZ(R)$ as a $\C$-algebra.
\end{lemma}

\subsection{Dimensions associated with regular classes}
In this subsection, we retain the notations in Section~\ref{sec:iPBW}.
Recall elements $d_{\nu_i}\in Z(\Ui_{\A'})$ for $1\leq i\leq m=\rank \mathfrak{k}$ from Section~\ref{sec:inii}, and the commutative Poisson algebra $\Ui_1$ from Section~\ref{sec:limit}. Write $d_i=\underline{d_{\nu_i}}$ to denote the image of $d_{\nu_i}$ in $\Ui_1$. Then it is clear that $d_i\in\PZ(\Ui_1)$ for $1\leq i\leq m$.

\begin{proposition}\label{prop:Pcenter}
    The algebra $\PZ(\Ui_1)$ is generated by $d_i$, $1\leq i\leq m$, as a $\C$-algebra.
\end{proposition}

\begin{proof}
    Recall the $\N^{N+M+2}$-filtration on $\Ui_{\A'}$ from Section \ref{sec:asa}. It induces a $\N^{N+M+2}$-filtration on the Poisson algrba $\Ui_1$ by Lemma \ref{le:basc} and \eqref{def:SCPoisson}. Then the associated graded algebra $\Gr \Ui_1$ is the Poisson algebra $\P_{S,J(S)}$ where $S$ is the skew-symmetric matrix defined in \eqref{eq:S} and $J(S)=\{t_{\omega_s^0}| 1\le s \le r\}$. Thanks to Remark \ref{rmk:deg} and  Lemma \ref{le:fpa}, it suffices to show that $\ov{d_i}$, $1\leq i\leq m$ generate the Poisson center $\PZ(\Gr\Ui_1)$. Indeed, this follows from Lemma \ref{le:pcp}, Proposition \ref{prop:dv}, and the proof of Lemma \ref{le:ker}.
\end{proof}

Recall from Section~\ref{sec:limit} that $\Ui_1$ is identified with the coordinate algebra $\mathbb{C}[\X]$ of the Poisson variety $\X=K^\perp\backslash G^*$. %For $x\in \X$, denote by $\mathcal{O}_x\subset \X$ the symplectic leaf containing $x$.  

\begin{proposition}\label{prop:defidel}
     Let $x\in \X$ and $\cO_x\subset \X$ be the symplectic leaf containing $x$. Suppose that $\dim\cO_x=\dim \k-\text{rank }\k$. Then $\overline{\cO_x}\subset \X$ is defined by the equations
     \begin{equation}\label{eq:defeq}
     d_i=d_i(x)\in\C, \quad \text{ for }1\leq i\leq m=\text{rank }\k.
     \end{equation}
\end{proposition}

\begin{proof}
    Let $V_x\subset \X$ be the subvariety defined by \eqref{eq:defeq}. Since $d_i$ lie in $\PZ(\Ui_1)$, it is constant on each symplectic leaves. Hence $\overline{\cO_x}\subset V_x$. On the other hand, consider elements $\overline{d_i}$ for $1\leq i\leq m$ in $\Gr \Ui_1$. By Proposition~\ref{prop:dv}, elements $\overline{{d_i}}$, $1\leq i\leq m$ are monomials over $\C[P^\theta]$ in disjoint sets of indeterminates. In particular, they form a regular sequence in $\Gr \Ui_1$. By \cite[Proposition~1.4 (a)]{DCKP93a}, elements $d_i$, $1\leq i\leq m$ form a regular sequence in $\Ui_1$. Hence $V_x\subset \X$ is an irreducible closed subvariety of codimension $m=\text{rank }\mathfrak{k}$. Hence $\dim \cO_x=\dim V_x$. We conclude that $\overline{\cO_x}=V_x$.
\end{proof}

We are now ready to prove Theorem \ref{thm:genrep}.

\begin{proof}[Proof of Theorem \ref{thm:genrep}]
    Let us write $N_0=(\dim \mathfrak{k}-\text{rank }\mathfrak{k})/2$. Recall from Corollary \ref{cor:dimv} that 
    $$
    \X^0=\{x\in\X\mid \dim V=\ell^{N_0},\text{ for any }V\in \text{Irr }\Ui_{v,x}\}.
    $$
    is an non-empty Zariski open subset of $\X$. Fix an element $x\in \X$ such that $\dim \mathcal{C}_\theta(\varphi(x))$ is maximal. Then the dimension of the symplectic leaf $\cO_x$ is maximal among all the symplectic leaves, thanks to Proposition \ref{prop:dimsym}. We need to show that $x\in\X^0$. By Corollary \ref{cor:sympl}, $\X^0$ is a union of symplectic leaves. Hence it suffices to show that 
\begin{equation}\label{eq:nonempt}
\overline{\cO_x}\cap {\X}^0\neq \emptyset.
\end{equation}

Let $\Gr\Ui_1$ be the associated graded Poisson algebra as in the proof of Proposition~\ref{prop:Pcenter}.
Recall from \eqref{eq:filFr} that $\Fri:\Ui_1\rightarrow \Ui_v$ is compatible with the $\N^{N+M+2}$-filtrations, and hence it induces an algebra embedding $\overline{\Fri}:\Gr \Ui_1\rightarrow \Gr \Ui_v$, which identifies $\Gr\Ui_1$ with the central subalgebra $\overline{Z_0^\imath}=\{\ov{z}\mid z\in \Zi\}$ of $\Gr \Ui_v$. Let $\overline{\X}=\MaxSpec \Gr\Ui_1$, and let
\[
\overline{\X}^0=\{x\in\overline{\X}\mid \dim V=\ell^{N_0},\text{for any }V\in \text{Irr }(\Gr\Ui_v)_{x}\}.
\]

Let $I\subset \Ui_1$ be the defining ideal of $\overline{\cO_x}$, and let $\overline{I}\subset \Gr\Ui_1$ be the associated graded ideal. Let $\cO_1\subset \overline{\X}$ be the set of zeros of $\overline{I}$. By Theorem \ref{thm:deg} and Theorem \ref{prop:degd} we have $\deg \Ui_v=\deg (\Gr \Ui_v)$. Thanks to \cite[Lemma~1.5]{DCKP93a}, it suffices to show that 
\begin{equation}
    \cO_1\cap {\overline{\X}}^0\neq \emptyset.
\end{equation}

By Proposition \ref{prop:defidel}, the ideal $I$ is generated by elements $d_i-c_i$, for some $c_i\in\C$, $1\leq i\leq m$. As observed in the proof of Proposition \ref{prop:defidel}, the elements $\overline{d_i}$, for $1\leq i\leq m$, form a regular sequence in $\Gr \Ui_1$. By \cite[Proposition~1.4 (b)]{DCKP93a}, the ideal $\overline{I}$ is generated by $\overline{{d_i}}$, for $1\leq i\leq m$. 

Recall from \cite[Proposition 4.1]{SZ25} that $\Ui_1$ is a polynomial algebra with generators 
\[
\un{B_{\beta_{k}}},\un{E_{\gamma_t}},\un{K_{\omega_s^0}^{\pm1}},\qquad 1\leq k\leq N,1\leq t\leq M,1\le s \le r.
\]
For $1\leq i\leq m$, by Proposition \ref{prop:dv}, there exists $1\leq {k_i}\leq N$ such that $\un{B_{\beta_{k_i}}}$ divides the monomial $\overline{{d_i}}$ in $\Gr \Ui_1$. Let $\cO_1'\subset \overline{\X}$ be the set of zeros of $\un{B_{\beta_{k_i}}}$ for $1\leq i\leq m$. Then $\cO_1'\subset \cO_1$ is an irreducible component. It suffices to show that 
\begin{equation}
    \cO_1'\cap {\overline{\X}}^0\neq \emptyset.
\end{equation}

Let $P=\Gr\Ui_v/ J$ where $J\subset \Gr\Ui_v$ is the two sided ideal generated by $B_{\beta_{k_i}}$, for $1\leq i\leq m$. Then $P$ is again a localized twisted polynomial, containing a central subalgebra which is isomorphic to $\C[\cO_1']$. Moreover one has $P_x\cong (\Gr \Ui_1)_x$ for any $x\in \cO_1'$. Hence 
\[
\cO_1'\cap {\overline{\X}}^0=\{x\in \cO_1'\mid \text{dim }V=\ell^{N_0},\text{for any }V\in \text{Irr }  P_x\}.
\]
By Proposition \ref{prop:twpolycenter} it suffices to show that $\deg P=\ell^{N_0}$.

Let $S_P$ be the skew-symmetric matrix associated with the the localized twisted polynomial algebra $P$. Recall from Section \ref{sec:degui} the skew-symmetric matrix $S$ associated with $\Gr\Uiv$, and recall the $\Z[2^{-1}]$-module $W=V\oplus V'\oplus {'P}^\theta$. We view $S_{\ell}$ as a skew-symmetric $\Z/\ell\Z$-valued bilinear form on $W$ induced by $S$. Let $V_P$ be the $\Z[2^{-1}]$-submodule of $V$, spanned by $u_t$, for $1\leq t\leq N$ and $t\notin \{k_1,\cdots ,k_m\}$. Write $W_P=V_P\oplus V'\oplus {'P^\theta}\subset W$. We view $S_{P,\ell}$ as the $\Z/\ell\Z$-valued skew-symmetric bilinear form on $W_P$ induced by $S_P$. Then $S_{P,\ell}=S_\ell\mid_ {W_P}$. It is then straightforward to see that $$\text{ker }S_{P,\ell}=\text{ker }S_\ell\cap W_P=\ell W_P.$$
The second equality follows from Lemma \ref{le:ker} and our choice of $W_P$. Hence when viewing $S_{P,\ell}$ as a linear operator $S_{P,\ell}:W_P\rightarrow W_p/\ell W_P$, we have $\text{im }S_{P,\ell}\cong W_P/\ell W_P$. Since $W_P$ has rank $2N_0$ we conclude that $\deg P=\ell ^{N_0}$ by Proposition \ref{prop:twpolycenter}. We complete the proof.
\end{proof}

For general $x\in \X$ and $V\in \Irr\Ui_{v,x}$, we conjecture that the dimension of $\ell$ is controlled by the dimension of the symplectic leaf containing $x$. Recall from Proposition \ref{prop:dimsym} that the dimension formula of symplectic leaves of $\X$. 

\begin{conjecture}\label{conj:dim}
    Let $x\in \X$ and $V\in \Irr \Ui_{v,x}$. Let $\mathcal{C}=\mathcal{C}_\theta(\varphi(x))$ be the $\theta$-twisted conjugacy class containing $\varphi(x)$. Then 
    \[
    \ell^{(\dim \mathcal{C}-\dim \g+\dim \k)/2} \mid \text{dim }V.
    \]
\end{conjecture}

\begin{remark} The following special cases of this conjecture is known.
\begin{itemize}
    \item[(1)] Theorem \ref{thm:genrep} implies that the conjecture holds when the dimension of $\mathcal{C}_\theta(\varphi(x))$ is maximal.

    \item[(2)] When $\theta$ is the trivial involution, Conjecture \ref{conj:dim} coincides with the De Concini-Kac-Procesi conjecture in \cite[Conjecture~6.8]{DCKP92}. Their conjecture was proved in \cite{Kre06,Sev21}.
\end{itemize}
\end{remark}

\section{Branching laws}\label{sec:bran}

In this section, we study the branching problem for representations of $\U_v$ restricting to $\Uiv$. In particular, we show that generic simple $\Uiv$-modules are direct summands of simple $\U_v$-modules.

\subsection{Compatible pairs}
We first recall the notion of reduced traces and compatible pairs from \cite{DCPRR05}. Let $R$ be a finitely generated $\C$-algebra. Assume that $R$ is an integrally closed domain, and is a finite module over its center $Z$. Let $A\subset Z$ be a central subalgebra. Assume that $A$ is finitely generated and integrally closed. There is a finite map $t:\MaxSpec Z\rightarrow \MaxSpec A$. For $x\in \MaxSpec A$, suppose that $t^{-1}(x)=\sum_{i=1}^s h_i y_i$ is the fibre of $x$, counted with multiplicities. Define the semisimple $R$-module $V(x)$ by 
\[
V(x)=\bigoplus_{i=1}^sV(y_i)^{\oplus h_i},
\]
where $V(y_i)$ is the semisimple $R$-module defined in Proposition \ref{prop:repalg}. 

We further assume that $Z$ is a projective $A$-module. Define the trace map $\tr _{Z/A}:Z\rightarrow A$ by setting $\tr_{Z/A}(z)$ to be the trace of the $A$-linear map $Z\rightarrow Z$, $x\mapsto zx$. Recall the reduced trace $\tr_{R/Z}:R\rightarrow Z$ from Section \ref{sec:icd}. Following \cite[Section 5.1]{DCPRR05}, define the \emph{reduced trace of $R$ over $A$} by $$\tr_{R/A}=\tr _{Z/A}\circ \tr_{R/Z}:R\rightarrow A,$$ and the \emph{degree of $R$ over $A$} by $$[R:A]=\deg R\cdot \dim_{Q(A)}Q(Z).$$
Note that $[R:Z]=\deg R$. By \cite[Section 5]{DCPRR05}, for $x\in \Spec A$, $V(x)$ is the unique $[R:A]$-dimensional semisimple representation of $R$ which is compatible with the trace $\tr_{R/A}$.

Let $R_1\subset R_2$ be two finitely generated $\C$-algebras which are integrally closed domains, and let $A_i\subset R_i$ be central subalgebras such that $A_i$ is finitely generated and $R_i$ is a finite $A_i$-module, for $i=1,2$. Suppose that $A_1$, $A_2$ are integrally closed, and the centers $Z(R_1)$, $Z(R_2)$ are projective modules over $A_1$, $A_2$, respectively. Following \cite[Definition 5.8 and Remark~5.9]{DCPRR05} we call that $R_1\subset R_2$ is \emph{compatible} with $A_1\subset A_2$ if the obvious map $R_1\otimes _{A_1}Q(A_2)\rightarrow R_2\otimes_{A_2}Q(A_2)$ is injective, where $Q(A_i)$ is the field of fractions of $A_i$. Let $\pi:\MaxSpec A_2\rightarrow \Spec A_1$ be the morphism induced by the embedding $A_1\subset A_2$.

\begin{proposition}[\text{\cite[Theorem 5.11 \& Theorem~5.12]{DCPRR05}}]\label{prop:brach}
    Retain the notations as above. Suppose that $R_1\subset R_2$ is compatible with $A_1\subset A_2$. Then there is positive integer $r$ and a nonempty Zariski open subset $\Omega\subset\MaxSpec A_2$, such that the following properties holds.

    (1) One has \[
    r[R_1:A_1]=[R_2:A_2],\qquad r\tr_{R_1/A_1}=\tr_{R_2/A_2}\quad \text{on }R_1.
    \]

    (2) When restricting to $R_1$, one has $V(x)\cong V(\pi(x))^{\oplus r}$, for any $x\in \Omega$.
\end{proposition}

\subsection{Branching laws for quantum symmetric pairs}
Let us come back to the setting of quantum symmetric pairs.

\begin{proposition}\label{prop:comp}
    The algebras $\Uiv\subset \U_v$ is compatible with $\Zi\subset Z_0$. Moreover one has
    \begin{equation}\label{eq:trce}
\tr_{\U_v\slash Z_0} = \ell^{N-N_0+\rank \g-\rank \k} \tr_{\Uiv\slash \Zi}.
\end{equation}
\end{proposition}

\begin{proof}
Let $Q(Z_0)$ and $Q(Z_0^\imath)$ denote the the field of fractions of $Z_0$ and $Z_0^\imath$, respectively. By Lemma~\ref{le:tsl}, $Q(Z_0)\otimes _{Z_0}\U_v$ admits a $Q(Z_0)$-basis which contains a $Q(\Zi)$-basis of $Q(\Zi)\otimes_ {\Zi}\Uiv$. Hence the map $Q(\Zi)\otimes_ {\Zi}\Uiv\rightarrow Q(Z_0)\otimes _{Z_0}\U_v$ is injective, and the first statement is proved. To prove the second statement, it is well-known that $[\U_v:Z_0]=\ell ^{N+\rank \g}$ (cf. \cite{DCK90}). By Theorem \ref{thm:deg}, one has $[\Uiv:\Zi]=\ell^{N_0+\rank \k}$. Then \eqref{eq:trce} follows from Proposition \ref{prop:brach} (1).
\end{proof}
% There is a unique trace map $\text{tr}_{R/A}:R\rightarrow R$ with the image in $A$, making $R$ into Cayley-Hamilton algebra of degree $n$, such that, for any trace map $\text{tr}:R\rightarrow R$ with the image in $A$ for which $R$ is a Cayley-Hamilton algebra of degree $m$, there is a positive integer $t$ for which $\text{tr}=t\text{tr}_{R/A}$. In particular, one has $m=tn$. The trace map $\text{tr}_{R/A}$ is called the \emph{reduced trace} of $R$ over $A$. Let $\mathcal{V}(A)$ be the affine variety associated with $A$. The \emph{unramified locus} of $R$ over $A$ is the non-empty Zariski oopen subset of $\mathcal{V}(A)$ consisting of points $v$ such that the finite-dimensional algebra $R/\mathfrak{m}_vR$ is semisimple where $\mathfrak{m}_v\subset A$ is the maximal ideal corresponding to $v$.

% Now suppose $R_1\subset R_2$ is compatible with $A_1\subset A_2$. Let $\mathcal{V}(A_i)$ be the affine variety associated with the algebra $A_i$, for $i=1,2$. Then the inclusion between algebras induces a map $\pi:\mathcal{V}(A_2)\rightarrow \mathcal{V}(A_1)$. For any $v\in \mathcal{V}(A_2)$, let $M_v$ be the $[R_2:A_2]$-dimensional semisimple representation associated with $v$ (see Proposition \ref{prop:CH} (2)) with respect to the reduced trace $\text{tr}_{R_2/A_2}$. Then by \cite[Theorem~5.12]{DCPRR05}, if $\pi(v)$ belongs to the unramified locus of $R_1$ over $A_1$, then the restriction of $M_v$ to $R_1$ is $M_{\pi(v)}^{\oplus r}$ where $r$ is a positive integer, and $M_{\pi(v)}$ is the $[R_1:A_1]$-dimensional representation associated with $\pi(v)\in\mathcal{V}(A_1)$.

\begin{theorem}\label{thm:dec}
    There is a non-empty Zariski open subset $\Omega$ of $G^*$, such that for any $x\in\Omega$ and $V\in \Irr \U_{v,x}$, when restricting to $\Uiv$ one has
    \[
    V\cong\bigoplus_{W\in\Irr\Ui_{v,\overline{x}}}W^{\oplus h_{V,W}},
    \] 
    for some $h_{V,W}\in \N$ such that
    \begin{equation}\label{eq:hvw}
    \sum_{W\in\Irr\Ui_{v,\overline{x}}}h_{V,W}=\ell ^{N-N_0},
    \quad \quad \sum_{V\in \Irr\U_{v,x}}h_{V,W}=\ell ^{N-N_0+\rank \g-\rank \k},
    \end{equation}
    where $N$ (resp., $N_0$) is the number of positive roots of $\g$ (resp., $\k$).
\end{theorem}

\begin{proof}
Let $Z\subset \U_v$ be the center of $\U_v$. Let $\Lambda:\MaxSpec Z\rightarrow \MaxSpec Z_0$ be the natural map induced by the algebra embedding. By Theorem~\ref{thm:deg}, $\Lambda$ is a finite map of degree $\ell^{\rank \g}$, and there is a non-empty Zariski open subset $\Omega_1$ of $G^*$, such that for $x\in \Omega_1$ the fibre $\Lambda^{-1}(x)$ consists of $\ell^{\rank \g}$ many distinct points. By Proposition \ref{prop:repalg} (3), we can choose $\Omega_1$ such that $\dim V=\ell^{N}$ for any $x\in \Omega_1$ and $V\in \Irr\U_{v,x}$. Then for any $x\in \Omega_1$, one has $V(x)=\oplus_{i=1}^{t}V(y_i) $ where $t=\ell^{\rank \g}$, $\{y_1,\cdots,y_{t}\}=\Lambda^{-1}(x)$, and $V(y_i)\in \Irr \U_{v,x}$ has dimension $\ell^{N}$. Similarly, there is a nonempty Zariski open subset $\Omega_2\subset \X$, such that for any $y\in \Omega_2$ one has $V(y)=\oplus_{i=1}^s V(z_i)$ where $s=\ell^{\rank \k}$ and $V(z_i)\in \Irr \Ui_{v,y}$ are distinct irreducible $\Uiv$-modules of dimension $\ell^{N_0}$. 

Let $\pi:G^*\rightarrow \X$ be the natural projection. Thanks to Proposition~\ref{prop:comp} and Proposition~\ref{prop:brach}, there is a nonempty Zariski open subset $\Omega_0\subset \MaxSpec Z_0\cong G^*$ such that 
$
V(x)\cong V(\pi(x))^{\oplus r}
$
as $\Uiv$-modules for any $x\in \Omega_0$, where $r=\ell^{N+\rank \g-(N_0+\rank \k)}$.

Let $\Omega:=\Omega_0\cap\Omega_1\cap \pi^{-1}(\Omega_2)$. By our construction above, for any $x\in \Omega$, we have
\[
\bigoplus_{V\in \Irr \U_{v,x}}V\cong\bigoplus_{W\in\Irr \Ui_{v,\overline{x}}}W^{\oplus r},
\]
as $\Uiv$-modules. Now the proposition follows from the Krull–Schmidt theorem and dimension counting.
\end{proof}
% \[
% r=\ell^{N+\rank \g-(N_0+\rank \k)};
% \]
% see also \cite[Theorem 5.11]{DCPRR05}.

\begin{remark}
    When the symmetric pair is of diagonal type, that is, $G=G'\times G'$ and $\theta:(g_1,g_2)\mapsto (g_2,g_1)$, \cite[Theorem~7.7]{DCPRR05} proved a stronger result that the multiplicities $h_{V,W}$ can taken to be equal. This is false for general quantum symmetric pairs. For example, when $\theta=id$, then only $h_{V,V}$ equals to 1, and other multiplicities equal to 0.
\end{remark}

\begin{corollary}\label{cor:dsum}
    There is a non-empty Zariski open subset $\Omega^\imath\subset \X$, such that for any $x\in\Omega^\imath$ and $W\in \Irr\Ui_{v,x}$ such that $W$ is a direct summand of some $V\in \Irr \U_v$ as $\Uiv$-modules.
\end{corollary}

\begin{proof}
    Take $\Omega\subset G^*$ to be the open subset as in the proof of Theorem \ref{thm:dec}. Recall $\pi:G^*\rightarrow\X$ is the projection map. Take $\Omega^\imath=\pi(\Omega)$. Then  $\Omega^\imath$ is open since $\pi$ is a quotient map. For any $x\in \X$ and $W\in \Irr \Ui_{v,x}$, take $x'\in\pi^{-1}(x)\cap \Omega$. By the second equality of \eqref{eq:hvw}, there exists $V\in \Irr\U_{v,x'}$ such that $h_{V,W}>0$, and hence $W$ is a direct summand of $V$. The proof is completed.
\end{proof}

%\begin{corollary} Reduced traces $\tr_{\Uiv\slash \Zi},\tr_{\U_v\slash Z_0} $ satisfy the following identity
%\[\tr_{\U_v\slash Z_0} = \ell^{N-N_0+\rank \g-\rank \k} \tr_{\Uiv\slash \Zi}.\]
%\end{corollary}

%\subsection{Reduced traces}

% sure. Definitions in \cite[\S 4]{DCPRR05} is quite general and technical. But according to \cite[5.1 \& Remark 5.2]{DCPRR05}, in practical cases, this is just the composition of the reduced trace (onto the full center) composing with usual trace map, maybe we can just use that as the definition : added in the first section, and added in Proposition \ref{prop:comp} Yes, this is great.

\section{Relative braid group symmetries}\label{sec:comp}

Only in this section, we assume that the Satake diagram $(\I=\bI\cup \wI,\tau)$ is quasi-split; i.e., $\bI=\emptyset$. By \cite[Theorem 3.7]{SZ25}, the relative braid group symmetry $\TT_i$ preserves $\Ui_{\A'}$ and hence it induces algebraic automorphisms on both $\Ui_1$ and $\Ui_v$; we also denote both induced maps by $\TT_i$. The goal of this section is to show the compatibility between the quantum Frobenius map $\Fri:\Ui_1\rightarrow \Uiv$ and $\TT_i$.

\subsection{The compatibility theorem}
\begin{theorem}\label{thm:tfcom}
Let $(\I=\bI\cup \wI,\tau)$ be a quasi-split Satake diagram. We have 
\[
\TT_i \circ \Fri=\Fri\circ \TT_i.
\]
In particular, $\Zi$ is closed under the action of $\TT_i$.
\end{theorem}

\begin{proof} 
By \cite{SZ25}, the induced map $\TT_i$ on $\Ui_1$ is a Poisson automorphism. On the other hand, since $\TT_i\circ \pi_v^\imath=\pi_v^\imath \circ \TT_i$, $\TT_i(\ker \pi_v^\imath)\subset \ker \pi_v^\imath$ and we have
\[
\TT_i\{x,y\}=\TT_i\pi_v^\imath\left(\frac{[\tx,\ty]}{\ell^2 (qv^{-1}-1)} \right)
=\pi_v^\imath\left(\frac{[\TT_i\tx,\TT_i\ty]}{\ell^2 (qv^{-1}-1)} \right)=\{\TT_i x, \TT_iy\},
\quad \forall x,y\in \Zi,
\]
where $\tx,\ty$ are preimages of $x,y$ under $\pi_v^\imath:\Ui_{\A'}\rightarrow \Uiv$.

Thus, it suffices to verify 
\begin{align}
\label{eq:TTFri}
\TT_i \circ \Fri(x)=\Fri\circ \TT_i(x),
\end{align}
when $x$ lies in a set of Poisson generators of $\Ui_1$. Recall from \cite[Lemma 3.4]{So24} that $\Ui_1$ is generated by $\un{B_j},\un{k_j}^{\pm1},j\in \wI=\I$ as a Poisson algebra. 

By \cite[Proposition 6.3]{WZ23} (see also \cite[Proposition 2.13]{WZ25}), we have in $\Ui_{\A'}$
\[
\TT_i(B_i)=q_i^{a_{i,\tau i}/2-1} B_{\tau_i \tau (i)} k_{\tau_i (i)}^{-1},
\qquad
\TT_i(B_{\tau i})=q_i^{a_{i,\tau i}/2-1} B_{\tau_i (i)} k_{\tau_i \tau (i)}^{-1}.
 \]
Hence, in $\Uiv$ we have
\[
\TT_i(B_i^{[\ell]})=B_{\tau_i \tau (i)}^{[\ell]} k_{\tau_i (i)}^{-\ell},
\qquad 
\TT_i(B_i^{[\ell]})=B_{\tau_i (i)}^{[\ell]} k_{\tau_i\tau (i)}^{-\ell};
\]
while in $\Ui_1$ we have $ \TT_i(\un{B_i})= \un{B_{\tau_i \tau (i)}} \cdot \un{k_{\tau_i (i)}^{-1}}$ and $\TT_i(\un{B_{\tau i}})= \un{B_{\tau_i (i)}} \cdot \un{k_{\tau_i \tau (i)}^{-1}}$. Hence, \eqref{eq:TTFri} holds for $x=B_i,B_{\tau i}$. It is also straightforward to verify that \eqref{eq:TTFri} holds for $x=\un{k_j}^{\pm1}$.

It remains to verify \eqref{eq:TTFri} for $x=B_j,j\neq i,\tau i$. The proof for this part is separated into 3 cases depending on $a_{i,\tau i}$ and will be given in Propositions~\ref{prop:vsplit}, \ref{prop:vdiag}, \ref{prop:vqsplit} in the next three subsections.
\end{proof}

\subsection{Case $a_{i,\tau i}=2$} In this subsection, fix $i\in \wI$ such that $a_{i,\tau i}=2$; i.e., $\tau i=i$.

We define for $m\in \N$
\begin{align}
[m]!!=\prod_{r=0}^{\lceil m/2\rceil-1} [m-2r].
%\Qbinom{m}{2k}=\frac{[m][m-2]\cdots [m-2k+2]}{[2][4]\cdots [2k]},\qquad \Qbinom{m}{0}=1.
\end{align}

\begin{lemma}\label{lem:vbinom}
We have
\[
\sum_{r=1}^{\ell-1} (-1)^r \frac{v_i^{r}}{(v_i-v_i^{-1})^\ell [r]_i![\ell-r]_i!}=\frac{1-\ell}{2\ell}
\]
\end{lemma}

\begin{proof}
By \cite[\S 1.3.4]{Lus93}, for generic $q$, we have
\[
\sum_{r=0}^\ell (-1)^r q_i^{r(1-\ell)} \qbinom{\ell}{r}_i =0,
\]
which implies that 
\begin{align*}
\frac{1}{(q_i-q_i^{-1})^\ell}\sum_{r=1}^{\ell-1} (-1)^r q_i^{r(1-\ell)} \qbinom{\ell}{r}_i 
=
\frac{q_i^{\ell-\ell^2}-1}{(1+q_i^{-\ell})(q_i^{\ell}-1)} \frac{1}{(q_i-q_i^{-1})^{\ell-1}[\ell-1]_i!}
\end{align*}
Specialize $q=v$ in the above identity and note that $\frac{q_i^{\ell-\ell^2}-1}{q_i^{\ell}-1}$ specialize to $1-\ell$. By \cite{DCK90}, $(v_i-v_i^{-1})^{\ell-1}[\ell-1]_i!=\ell$. Thus, the desired identity follows.
\end{proof}

Recall the element $B_i^{[m]},m\in \N$ from Section~\ref{sec:iDP}.

\begin{lemma}\label{lem:prodBB}
For any $a\ge 0$, we have in $\Uiv$
\begin{align*}
B_i^{[a]} B_i^{[\ell-a]} &= B_i^{[\ell]},
\qquad \qquad
B_i^{[a]} B_i^{[2\ell-a]} = B_i^{[2\ell]}.
\end{align*}
\end{lemma}

\begin{proof}
By \cite[Lemma 4.4]{BS22} and Section~\ref{sec:iDP}, for generic $q$ and $a,k\geq0$, we have in $\Ui$
\begin{align}\label{eq:prodBB}
B_i^{[a]} B_i^{[k]} = \sum_{t\geq 0} (-1)^t\frac{[a+k]_i!}{[a+k-2t]_i!} 
&\left(\prod_{m=1}^t \frac{[a-2m+2]_i[k-2m+2]_i}{[a+k-2m+1]_i [2m]_i} \right)
 (q_i-q_i^{-1})^{2t} B_i^{[a+k-2t]}.
%\prod_{m=1}^t \frac{[a-2m+2]_i[k-2m+2]_i}{[a+k-2m+1]_i [2m]_i}
\end{align}
Now we set $q=v$ to be a primitive $\ell$th root of unity. Note that $\frac{[\ell]_i!}{[\ell-2t]_i!}=0$ unless $t=0$ and $ \prod_{m=1}^t 
\frac{[a-2m+2]_i[\ell-a-2m+2]_i}{[\ell-2m+1]_i [2m]_i}$ does not have any pole at $q=v$. Hence, we have in $\Uiv$
\[
B_i^{[a]} B_i^{[\ell-a]} = B_i^{[\ell]}.
\]
We show the second identity. We consider the coefficient in the right-hand side of \eqref{eq:prodBB} for $t>0$. Note that $\frac{1}{[2\ell-2t]_i!\prod_{m=1}^t [2\ell-2m+1]_i}$ always has a single pole at $q=v$ for any $t>0$. 
Moreover, for any $0\le a \le 2\ell,t\ge 0$, $\prod_{m=1}^t\frac{[a-2m+2]_i[2\ell-a-2m+2]_i}{[2m]_i}$ does not have any pole at $q=v$. Hence, the coefficient in the right-hand side of \eqref{eq:prodBB} is $0$ unless $t=0$, which implies $B_i^{[a]} B_i^{[2\ell-a]} = B_i^{[2\ell]}.$
\end{proof}

\begin{proposition}\label{prop:vsplit}
Let $j\in \wI$ such that $j\neq i$. Then we have $\TT_i(B_j^{[\ell]})=\Fri(\TT_i(\un{B}_j))$.
\end{proposition}

\begin{proof}
By \cite[Theorem 3.11]{WZ25} and \cite[Proposition 4.6]{BS22}, we have in $\Ui_{\A'}$
\begin{align}
\label{eq:TTell}
\TT_i(B_j^{[\ell]})=\sum_{r+s=-\ell a_{ij}} (-1)^r q_i^r 
\frac{q_i^{\ell a_{ij}/2}}{(q_i-q_i^{-1})^{-\ell a_{ij}}}\frac{B_i^{[r]} B_j^{[\ell]} B_i^{[s]}}{[r]_i![s]_i!}.
\end{align}

We prove this proposition case by case depending on $a_{ij}$. The case $a_{ij}=0$ is trivial. 

Suppose that $a_{ij}=-1$. By \cite[Lemma 3.6]{SZ25}, we have $\TT_i(\un{B_j})=\frac{1}{2\epsilon_i}\{\un{B_j},\un{B_i}\}-\frac{1}{2} \un{B_j} \cdot \un{B_i}  $. On the other hand, recall from \cite{DCK90} that $(v_i-v_i^{-1})^{\ell-1}[\ell-1]_i!=\ell$. Note also that for generic $q$
\[
(q_i-q_i^{-1})[\ell]_i=(q_i^{-\ell}+1)(q_i^\ell-1)=(q_i^{-\ell}+1)(q_i-v_i)(q_i^{\ell-1}+q_i^{\ell-2}v_i+\cdots +v_i^{\ell-1}).
\]
Thus, by \eqref{eq:TTell} and \eqref{eq:Poissonv}, we have in $\Uiv$
\begin{align*}
\TT_i(B_j^{[\ell]})=&\pi_v^\imath\left(\frac{1}{(q_i-q_i^{-1})^{\ell-1}[\ell-1]_i!(q_i^{-\ell}+1)}
\bigg(\frac{\big[B_j^{[\ell]}, B_i^{[\ell]}\big]}{\ell(q_i v_i^{-1}-1)}-B_i^{[\ell]} B_j^{[\ell]}\bigg)\right)
\\
&+\sum_{r=1}^{\ell-1} (-1)^r v_i^r 
\frac{B_j^{[\ell]}}{(v_i-v_i^{-1})^{\ell}}\frac{B_i^{[r]}  B_i^{[\ell-r]}}{[r]_i![\ell-r]_i!}
\\
=& \frac{1}{2\epsilon_i} \big\{B_j^{[\ell]}, B_i^{[\ell]}\big\}-\frac{1}{2\ell} B_j^{[\ell]} B_i^{[\ell]}
+ \frac{B_j^{[\ell]}B_i^{[\ell]}}{(v_i-v_i^{-1})^{\ell}}\sum_{r=1}^{\ell-1} (-1)^r v_i^r 
\frac{1}{[r]_i![\ell-r]_i!}
\\
=& \frac{1}{2\epsilon_i} \big\{B_j^{[\ell]}, B_i^{[\ell]}\big\}-\frac{1}{2\ell} B_j^{[\ell]} B_i^{[\ell]}+\frac{1-\ell}{2\ell} B_j^{[\ell]} B_i^{[\ell]}
\\
=& \frac{1}{2\epsilon_i} \big\{B_j^{[\ell]}, B_i^{[\ell]}\big\}-\frac{1}{2} B_j^{[\ell]} B_i^{[\ell]}=\Fri\big(\TT_i(\un{B_j})\big),
\end{align*}
where the third equality follows from Lemma~\ref{lem:vbinom}. The desired statement is verified for $a_{ij}=-1$.

Suppose that $a_{ij}=-2$. In this case, $q_i=q,q_j=q^2$. By \cite[Lemma 3.6]{SZ25}, we have $\TT_i(\un{B_j})=\frac{1}{8}\big\{\{\un{B_j},\un{B_i}\},\un{B_i}\big\}-\frac{1}{4} \{\un{B_i} \cdot\un{B_j},\un{B_i}\}+ \un{B_j}$. On the other hand, note that $\frac{[2\ell]!}{[(\ell]!)^2}$ specializes to $2$ at $q=v$; cf. \cite[Lemma 34.1.2]{Lus93}. By \eqref{eq:prodBB} and similar arguments as in the proof of Lemma~\ref{lem:prodBB}, we have
\begin{align}
\pi_v^\imath\left(\frac{\big(B_i^{[\ell]}\big)^2-B_i^{[2\ell]}}{(q-q^{-1})^{2\ell}[2\ell]!}\right)
=\pi_v^\imath\left((-1)^\ell\frac{\big([\ell]!! \cdot [\ell-2]!!\big)^2}{[2\ell-1]!!\cdot [2\ell]!!}\right)=-\frac{1}{2}.
\end{align}
Moreover, by \eqref{eq:prodBB}, we have $B_i^{[\ell]}B_i^{[r]}=B_i^{[\ell+r]}$ for $0\le r\le \ell$ in $\Uiv$. Thus, we have in $\Uiv$
\begin{align*}
\TT_i(B_j^{[\ell]})=&\pi_v^\imath\sum_{r+s=2\ell } (-1)^r q_i^r 
\frac{q_i^{\ell}}{(q_i-q_i^{-1})^{2\ell }}\frac{B_i^{[r]} B_j^{[\ell]} B_i^{[s]}}{[r]_i![s]_i!}.
\\
=&\pi_v^\imath\left(\frac{\big[[B_j^{[\ell]},B_i^{[\ell]}],B_i^{[\ell]}\big]}{2(q-q^{-1})^{2\ell }([\ell]!)^2}\right)
+\pi_v^\imath\left(B_j^{[\ell]} \frac{-\big(B_i^{[\ell]}\big)^2+B_i^{[2\ell]}}{(q-q^{-1})^{2\ell}[2\ell]!}+q^{2\ell}\frac{-\big(B_i^{[\ell]}\big)^2+B_i^{[2\ell]}}{(q-q^{-1})^{2\ell}[2\ell]!}B_j^{[\ell]}\right)
\\
&+\pi_v^\imath\left(\frac{(q^{2\ell}-1) \big(B_i^{[\ell]}\big)^2 B_j^{[\ell]}-2(q^\ell-1)  B_i^{[\ell]} B_j^{[\ell]} B_i^{[\ell]}}{2(q-q^{-1})^\ell([\ell]!)^2}\right)
\\
&+ \sum_{r=1 }^{\ell-1} (-1)^r \frac{v_i^r}{(v_i-v_i^{-1})^{\ell }[r]_i![\ell-r]_i!} \pi_v^\imath \left(\frac{B_i^{[r]} B_j^{[\ell]} B_i^{[2\ell -r]}-B_i^{[\ell+r]} B_j^{[\ell]} B_i^{[\ell-r]}}{(q_i-q_i^{-1})^{\ell }[\ell]!}\right)
\\
=&\frac{1}{8}\Big\{\big\{B_j^{[\ell]},B_j^{[\ell]}\big\},B_j^{[\ell]}\Big\} + B_j^{[\ell]}
-\frac{1}{4\ell}\big\{B_i^{[\ell]} B_j^{[\ell]},B_i^{[\ell]}\big\}
+\frac{1-\ell}{4\ell}\big\{B_i^{[\ell]} B_j^{[\ell]},B_i^{[\ell]}\big\}
\\
=&\Fri \big(\TT_i(\un{B_j})\big)
\end{align*}
where the third equality follows from Lemma~\ref{lem:vbinom}. The desired statement is verified for $a_{ij}=-2$.

The case $a_{ij}=-3$ can be verified similarly. We omit the detail.
\end{proof}

\subsection{Case $a_{i,\tau i}=0$} In this subsection, fix $i\in \wI$ such that $a_{i,\tau i}=0$. In this case, $q_i=q$.

\begin{proposition}\label{prop:vdiag}
Let $j\in \wI$ such that $j\neq i,\tau i$. Then we have $\TT_i(B_j^{[\ell]})=\Fri(\TT_i(\un{B}_j))$.
\end{proposition}

\begin{proof}
If $c_{ij}=c_{\tau i, j}=0$, then the statement is trivial. If $c_{\tau i,j}=0,c_{ij}=-1$, then $\TT_i(\un{B}_j)=\frac{1}{2}\{\un{B}_j, \un{B}_i\}-\frac{1}{2} \un{B}_j \cdot \un{B}_i$ and one can use the same arguments as in the proof of Proposition~\ref{prop:vsplit} to show that $\TT_i(B_j^{[\ell]})=\Fri(\TT_i(\un{B}_j))$. The case for $c_{ij}=0,c_{\tau i,j}=-1$ is similar.

It remains to consider the nontrivial case $c_{ij}=c_{\tau i,j}=-1$.  By \cite[Lemma 3.6]{SZ25}, we have 
\[
\TT_i(\un{B}_j)=\frac{1}{4}\big\{\{\un{B_j},\un{B_{\tau i}}\},\un{B_i}\big\}
-\frac{1}{4} \un{B_i}\{ \un{B_j},\un{B_{\tau i}}\}
-\frac{1}{4} \{\un{B_{\tau i}} \cdot \un{B_j},\un{B_i}\}
+ \frac{1}{4}\un{B_i}\cdot \un{B_{\tau i}} \cdot \un{B_j} + \un{k_i}^{-1}\cdot \un{B_j}.
\]
On the other hand, by \cite[Theorem 4.13]{WZ25}, we have for generic $q$
\begin{align}
\TT_i(B_j^{[\ell]})=\sum_{u=0}^\ell \sum_{r=0}^{\ell-u} \sum_{s=0}^{\ell-u}
(-1)^{r+s} q^{r(-u+1)+s(u+1)} k_i^{-u}\frac{q^{-\ell+u} B_i^{[r]}B_{\tau i}^{[s]}B_j^{[\ell]}B_{\tau i}^{[\ell-s-u]}B_i^{[\ell-r-u]}}{(q-q^{-1})^{2\ell-2u}[r]![s]![\ell-r-u]![\ell-s-u]!}.
\end{align}
Now we specialize $q$ to $v$. Note that for any $0<u<\ell$, we have in $\Uiv$
\begin{align*}
&\quad \sum_{r=0}^{\ell-u} \sum_{s=0}^{\ell-u}
(-1)^{r+s} v^{r(-u+1)+s(u+1)} \frac{  B_i^{[r]}B_{\tau i}^{[s]}B_j^{[\ell]}B_{\tau i}^{[\ell-s-u]}B_i^{[\ell-r-u]}}{(v-v^{-1})^{2\ell-2u}[r]![s]![\ell-r-u]![\ell-s-u]!}
\\
&=\frac{1}{\big([\ell-u]!\big)^2}\sum_{r=0}^{\ell-u}(-1)^{r} v^{r(-u+1)} 
\Big( \sum_{s=0}^{\ell-u}(-1)^s v^{s(u+1)}\qbinom{\ell-u}{s}\Big)
\qbinom{\ell-u}{r}\frac{ B_j^{[\ell]} B_i^{[r]} B_{\tau i}^{[\ell-u]} B_i^{[\ell-r-u]}}{(v-v^{-1})^{2\ell-2u} }
\\
&=0,
\end{align*}
where the last equality follows from the $q$-binomial identity in \cite[\S 1.3.4]{Lus93}. Thus, we have in $\Uiv$
\begin{align*}
\TT_i(B_j^{[\ell]})&=  \pi_v^\imath \left(\sum_{r=0}^{\ell} \sum_{s=0}^{\ell}(-1)^{r+s} q^{r +s} 
\frac{q^{-\ell} B_i^{[r]}B_{\tau i}^{[s]}B_j^{[\ell]}B_{\tau i}^{[\ell-s]}B_i^{[\ell-r]}}{(q-q^{-1})^{2\ell}[r]![s]![\ell-r]![\ell-s]!}\right)+k_i^{-\ell} B_j^{[\ell]}
\\
&=\pi_v^\imath\left(\frac{B_j^{[\ell]}B_{\tau i}^{[\ell]}B_{i}^{[\ell]}- q^{\ell} B_i^{[\ell]} B_j^{[\ell]}B_{\tau i}^{[\ell]}
-q^{\ell}B_{\tau i}^{[\ell]} B_j^{[\ell]}B_i^{[\ell]}
+q^{2\ell}B_i^{[\ell]} B_{\tau i}^{[\ell]}B_j^{[\ell]}}{(q-q^{-1})^{2\ell} \big([\ell]!\big)^2}\right)
\\
&\quad +\pi_v^\imath\left(\sum_{r=1}^{\ell-1} (-1)^{r} q^{r} 
\frac{B_i^{[r]}B_j^{[\ell]}B_{\tau i}^{[\ell]}B_i^{[\ell-r]}
-q^{\ell}B_i^{[r]}B_{\tau i}^{[\ell]} B_j^{[\ell]}B_i^{[\ell-r]}}{(q-q^{-1})^{2\ell}[r]![\ell-r]![\ell]!}\right)
\\
&\quad +\pi_v^\imath \left(  \sum_{s=0}^{\ell}(-1)^{s} q^{s} 
\frac{B_{\tau i}^{[s]} B_j^{[\ell]}B_{\tau i}^{[\ell-s]}B_i^{[\ell]}
-q^{\ell}B_i^{[\ell]} B_{\tau i}^{[s]} B_j^{[\ell]}B_{\tau i}^{[\ell-s]} }{(q-q^{-1})^{2\ell}[\ell]![s]![\ell-s]!}\right)
\\
&\quad +\sum_{r=1}^{\ell-1} \sum_{s=1}^{\ell-1}(-1)^{r+s} v^{r +s} 
\frac{  B_i^{[r]}B_{\tau i}^{[s]}B_j^{[\ell]}B_{\tau i}^{[\ell-s]}B_i^{[\ell-r]}}{(v-v^{-1})^{2\ell}[r]![s]![\ell-r]![\ell-s]!}+k_i^{-\ell} B_j^{[\ell]}
\\
&=\frac{1}{4}\big\{\{B_j^{[\ell]},B_{\tau i}^{[\ell]}\},B_i^{[\ell]}\big\}
-\frac{1}{4\ell} B_{i}^{[\ell]}\{ B_j^{[\ell]},B_{\tau i}^{[\ell]}\}
-\frac{1}{4\ell} \{B_{\tau i}^{[\ell]} B_{j}^{[\ell]},B_{i}^{[\ell]}\}
+ \frac{1}{4\ell^2} B_{i}^{[\ell]} B_{\tau i}^{[\ell]} B_{j}^{[\ell]}
\\
&\quad +\frac{1-\ell}{4\ell} B_{i}^{[\ell]}\{ B_j^{[\ell]},B_{\tau i}^{[\ell]}\}
-\frac{1-\ell}{4\ell^2} B_{i}^{[\ell]} B_{\tau i}^{[\ell]} B_{j}^{[\ell]}
+\frac{1-\ell}{4\ell} \{B_{\tau i}^{[\ell]} B_{j}^{[\ell]},B_{i}^{[\ell]}\}
-\frac{1-\ell}{4\ell^2} B_{i}^{[\ell]} B_{\tau i}^{[\ell]} B_{j}^{[\ell]}
\\
&\quad+\frac{(1-\ell)^2}{4\ell^2} B_{i}^{[\ell]} B_{\tau i}^{[\ell]} B_{j}^{[\ell]}
+k_i^{-\ell} B_j^{[\ell]}
\\
&=\Fri(\TT_i(\un{B}_j)),
\end{align*}
where we used Lemma~\ref{lem:vbinom} in the third equality. The proof is completed.
\end{proof}

\subsection{Case $a_{i,\tau i}=-1$} In this subsection, fix $i\in \wI$ such that $a_{i,\tau i}=-1$. In this case, $q_i=q$. 

\begin{proposition}\label{prop:vqsplit}
Let $j\in \wI$ such that $j\neq i,\tau i$. Then we have $\TT_i(B_j^{[\ell]})=\Fri(\TT_i(\un{B}_j))$.
\end{proposition}

\begin{proof}
When $a_{i,\tau i}=-1$, the only nontrivial rank 2 cases involving vertices $i,j$ are $a_{ij}=-1,a_{\tau i,j}=0$ and $a_{\tau i,j}=-1,a_{i,j}=0$. It suffices to consider the first case $a_{ij}=-1,a_{\tau i,j}=0$.

By \cite[Lemma 3.6]{SZ25}, we have
\[
\TT_i(\un{B}_j)=\frac{1}{4}\big\{\{\un{B_j},\un{B_i}\},\un{B_{\tau i}}\big\}
-\frac{1}{4}\{ \un{B_i} \cdot \un{B_j},\un{B_{\tau i}}\}
-\frac{1}{4} \un{B_{\tau i}}\{ \un{B_j},\un{B_i}\}
+ \frac{1}{4}\un{B_i}\cdot \un{B_{\tau i}} \cdot \un{B_j} +\un{k_i}\cdot\un{B_j}.
\]
On the other hand, by \cite[Theorem 6.17]{WZ25}, we have
\begin{align}
\TT_i(B_j^{[\ell]})=\sum_{u=0}^\ell \sum_{s=0}^{\ell-u}\sum_{r=0}^{\ell-u} (-1)^{r+s} q^{r(u+1)+s(-2u+1)-\frac{u^2}{2}} k_i^u 
\frac{q^{-\ell+u}B_{\tau i}^{[s]}B_i^{[r]} B_{j}^{[\ell]}B_{i}^{[\ell-u-r]}B_{\tau i}^{[\ell-u-s]} }{(q-q^{-1})^{2\ell-2u}[s]![r]![\ell-u-r]![\ell-u-s]!}.
\end{align}
By similar arguments as in the proof of Proposition~\ref{prop:vdiag}, for any fixed $0<u<\ell$, the summand on the right-hand side is $0$ in $\Uiv$. Hence, we have in $\Uiv$,
\begin{align*}
\TT_i(B_j^{[\ell]})&= \pi_v^\imath \left(\sum_{s=0}^{\ell}\sum_{r=0}^{\ell} (-1)^{r+s} q^{r+s} 
\frac{B_{\tau i}^{[s]}B_i^{[r]} B_{j}^{[\ell]}B_{i}^{[\ell-r]}B_{\tau i}^{[\ell-s]} }{(q-q^{-1})^{2\ell}[s]![r]![\ell-r]![\ell-s]!} \right) + k_i^\ell B_j^{[\ell]}
\\
&=\frac{1}{4}\big\{\{B_j^{[\ell]},B_{i}^{[\ell]}\},B_{\tau i}^{[\ell]}\big\}
-\frac{1}{4\ell} \{B_{i}^{[\ell]} B_j^{[\ell]},B_{\tau i}^{[\ell]}\}
-\frac{1}{4\ell} B_{\tau i}^{[\ell]}\{ B_{j}^{[\ell]},B_{i}^{[\ell]}\}
+ \frac{1}{4\ell^2} B_{i}^{[\ell]} B_{\tau i}^{[\ell]} B_{j}^{[\ell]}
\\
&\quad +\frac{1-\ell}{4\ell} \{B_{i}^{[\ell]} B_j^{[\ell]},B_{\tau i}^{[\ell]}\}
-\frac{1-\ell}{4\ell^2} B_{i}^{[\ell]} B_{\tau i}^{[\ell]} B_{j}^{[\ell]}
+\frac{1-\ell}{4\ell} B_{\tau i}^{[\ell]}\{ B_{j}^{[\ell]},B_{i}^{[\ell]}\}
-\frac{1-\ell}{4\ell^2} B_{i}^{[\ell]} B_{\tau i}^{[\ell]} B_{j}^{[\ell]}
\\
&\quad+\frac{(1-\ell)^2}{4\ell^2} B_{i}^{[\ell]} B_{\tau i}^{[\ell]} B_{j}^{[\ell]}
+k_i^{-\ell} B_j^{[\ell]}
\\
&=\Fri(\TT_i(\un{B}_j)).
\end{align*}
(The computation involved here is essentially the same as the one in the proof of Proposition~\ref{prop:vdiag}.) The proof is completed.
\end{proof}

\end{document}